\documentclass[10pt,reqno,onepage]{amsart}
\usepackage{amsfonts,amssymb,amsmath}

\usepackage{authblk}

\usepackage{enumitem}
\usepackage{aliascnt}
\usepackage[mathscr]{eucal}
\usepackage{graphicx}
\usepackage{latexsym}
\usepackage{psfrag}
\usepackage{epsfig}
\usepackage[utf8]{inputenc}
\usepackage[english]{babel}
\usepackage{tikz}
\usepackage[foot]{amsaddr}
\usepackage{lineno, blindtext}

\newcommand{\mA}{\mathcal{A}}
\newcommand{\ba}{\bf{a}}
\newcommand{\dist}{\operatorname{dist}}
\renewcommand{\div}{\mbox{div}\,}
\newcommand{\Tr}{\operatorname{tr}}
\newcommand{\abs}[1]{|{#1}|}
\newtheorem{definition}{Definition}
\newcommand{\boundellipse}[3]
{(#1) ellipse (#2 and #3)
}

\newcommand{\norm}[1]{\|{#1}\|}

\newtheorem{theorem}{Theorem}[section]
\newtheorem{lemma}[theorem]{Lemma}

\newcommand{\suchthat}{\;\ifnum\currentgrouptype=16 \middle\fi|\;}

\usetikzlibrary{decorations.pathreplacing}

\usepackage[makeroom]{cancel}

\numberwithin{equation}{section}
\newtheorem{thm}{Theorem}
\numberwithin{thm}{section}

\newaliascnt{lemma}{thm}
\newtheorem{lem}[lemma]{Lemma}
\aliascntresetthe{lemma}

\newaliascnt{proposition}{thm}
\newtheorem{prop}[proposition]{Proposition}
\aliascntresetthe{proposition}

\newaliascnt{corollary}{thm}
\newtheorem{corollary}[corollary]{Corollary}
\aliascntresetthe{corollary}

\newaliascnt{remark}{thm}
\newtheorem{remark}[remark]{Remark}
\aliascntresetthe{remark}

\usepackage[colorlinks=true]{hyperref}
\def\eR{{\mathbb{R} }}

\begin{document}

	\include{formatAndDefs}
	
	\title[Fluid-structure interaction with heat]{Thermal effects in fluid structure interactions}

	\author[ S. Mitra, \& S. Schwarzacher]{
		\small
		Sourav Mitra$^{\dagger}$, and 
		Sebastian Schwarzacher$^\ddagger$$^*$
	}
	\address{
		$^\dagger$Department of Mathematics, 
		Indian Institute of Technology Indore, \\ 
		Simrol, Indore, 453552, Madhya Pradesh, India
	}
	
	\email{souravmitra@iiti.ac.in}
	
	\address{
		$^\ddagger$ Department of Mathematical Analysis, Faculty of Mathematics and Physics, Charles University, Sokolovska 83, 186 75 Praha 8, Czech Republic
	}
	
	\address{
		$^*$ Department of Mathematics, Analysis and Partial Differential Equations, Uppsala University, Lagerhyddsvagen 1, 752 37 Uppsala, Sweden
	}
	
	\email{schwarz@karlin.mff.cuni.cz}
	
	\date{\today}
	
	\begin{abstract}
		In this article we consider two different heat conducting fluids each modeled by the incompressible Navier-Stokes-Fourier system separated by a non-linear elastic Koiter shell. The motion of the shell changes the domain of definition of the two separated fluids. For this setting we show the existence of a weak solution. The heat capacity of the shell is given energetically. It allows to consider transmission laws ranging from insulation to superconductivity.
		We follow a variational approach for fluid-structure interactions. To include temperature a novel two step minimization scheme is used to produce an approximation. The weak solutions are energetically closed and include a strictly positive temperature. 
	\end{abstract}
	
	\maketitle
	\section{Introduction}
	In the present article we are interested about the global existence of weak solutions of a fluid structure interaction problem involving two Newtonian incompressible heat conducting fluids separated by a nonlinear elastic Koiter shell~\cite{CiarletIII}. The elastic shell involved allows heat transfer between the two fluids of various types ranging from superconductivity to insulation~\eqref{eq:heat}.
	The motion of the shell changes the domain of definition of the two fluids. For more details and a derivation, see Subsection~\ref{sec:Geomtry}.
	We show the global existence of weak solutions up to the point of a geometric singularity. This means either the Koiter energy degenerates or the structure undergoes a self intersection, or it touches the rigid boundary $\partial\Omega$. The two fluids involved are assumed to have different prescribed viscosity laws, which each depend on the respective temperatures. The temperature dependence of the viscocities is modeled by the celebrated  Vogel-Fulcher-Tammann equation~\eqref{viscosity}.

	\begin{figure}[h!]
		\centering
		\begin{tikzpicture}[scale=0.75]
			\draw \boundellipse{0,0}{6}{4};
			\coordinate  (O) at (0,0);
			\draw [line width=0pt, black]  plot[smooth, tension=.7] coordinates {(-4,2.5) (-3,3) (-2,2.8) (-0.8,2.5) (-0.5,1.5) (0.5,0) (0,-2)(-1.5,-2.5) (-4,-2) (-3.5,-0.5) (-5,1) (-4,2.5)};
			\draw (3,0)node [below right] {${\Omega^{1}(t)}$};
			\draw (-0.2,0)node [below right] {${\Omega^{2}(t)}$};
			\coordinate (A) at (-0.8,2.5);
			\draw (-2,0)node [] {$\Omega^{2}_{o}$};
			\coordinate (B) at (0,3.5);
			\draw [->,blue](A) -- (B) ;
			\draw (0,3.2)node [below right] {$\nu_{\eta}$};
			\draw (0.5,1)node [below right] {$\Sigma_{\eta}$};
			\draw \boundellipse{-2,0}{1.8}{3};
		\end{tikzpicture}
		\caption{Sketch of the domain $\Omega$ where $\Omega=\overline{{\Omega^{1}(t)}}\cup{\Omega^{2}(t)}$ and the boundary of ${\Omega^{2}(t)}$ is $\Sigma_{\eta}.$}
	\end{figure}
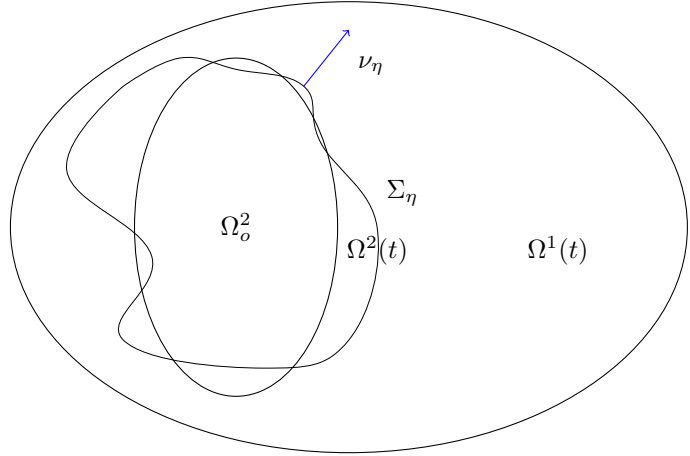 
	The geometry of the fluid domains is variable in time and depends on the deformation of the shell. In explicit we consider a domain $\Omega$ and two incompressible fluids confined in time dependent sub-domains ${\Omega^{1}(t)}$ and ${\Omega^{2}(t)}$ separated by an Koiter shell, $i.e.$
	\begin{equation}\label{Omegat12}
		\begin{array}{ll}
			&\Omega=\overline{{\Omega^{1}(t)}}\cup{\Omega^{2}(t)},\quad \partial{\Omega^{2}(t)}=\Sigma_{\eta},\\[2.mm] &\partial({\Omega^{1}(t)})=\partial\Omega\cup \Sigma_{\eta}.
		\end{array}
	\end{equation}  
	We denote by $\nu_{\eta}$ the unit outward normal to $\Sigma_{\eta}$, that is the time-changing surface representing the Koiter shell. For that we make the assumption that it moves in the normal direction of a reference configuration. 
	From a physical point of view this assumption is made to model the phenomenon where the fluid pressure is the dominant force acting on the structure in which case it is reasonable to assume that the shell is deforming in the normal direction to a reference configuration ($c.f.$ \cite{canic2020moving}).
	
	In order to state the equations some notation is required. We assume that $\Omega^{2}_{o}\Subset \Omega$ (as in Figure $1.$) such that $\dist(\partial\Omega^{2}_{o},\partial\Omega)>0$. Further we assume that $\partial\Omega^{2}_{o}$ (the boundary of $\Omega^{2}_{o}$) is parametrized by a $C^{4}$ injective mapping $\varphi:\Gamma\rightarrow \mathbb{R}^{3},$ where $\Gamma\subset \mathbb{R}^{2}.$ Now $\varphi(\Gamma)=\partial\Omega^{2}_{o}$ is a two dimensional surface, which corresponds to the flat middle surface of the closed shell. Let $\nu:\Gamma\to \mathbb{R}^3$ be the fixed outer normal of the reference surface $\partial\Omega^{2}_{o}$. Then the displacement becomes $\eta\nu$ where $\eta:\Gamma\rightarrow\mathbb{R}$ is the unknown function determining the shell motion.   
	Actually, as the surface is closed we can identify it with a torus inferring that displacement $\eta$ and the velocity $\partial\eta$ of the structure are space periodic.
	The time dependent fluid boundary $\Sigma_{\eta}$ can then be characterized by an injective mapping $\varphi_{\eta}$ such that for all pairs $x=(x_{1},x_{2})\in \Gamma,$ the pair of tangent vectors $\partial_{i}\varphi_\eta(x),$ $i=1,2,$ is linearly independent. More precisely, $\varphi_{\eta}$ is defined as follows
	\begin{equation}\label{varphieta}
		\varphi_{\eta}(t,x)=\varphi(x)+\eta(t,x)\nu(x) \text{ for } x\in \Gamma,
	\end{equation} 
	\begin{equation}\label{Sigmat}
		\Sigma_{\eta}(t)=\{\varphi_\eta(x,t): x\in\Gamma\}.
	\end{equation} 
	It is a well known result on the tubular neighborhood, see e.g. \cite[Section 10]{Lee} that 
	\begin{equation}\label{bjtparam}
		\text{there are }\,a_{\partial\Omega},b_{\partial\Omega},\,\, \mbox{such that for}\,\, \eta\in(a_{\partial\Omega},b_{\partial\Omega})\footnote{By $\eta\in(a_{\partial\Omega},b_{\partial\Omega})$ we mean that $\eta(x,t) \in(a_{\partial\Omega},b_{\partial\Omega})$ for $ x \in \Gamma$ and $t\in I$. We will use this slight abuse of notation through out since no confusion arises.} ,
		\varphi_{\eta}(t,\cdot)\,\,\mbox{ bijectively parametrizes }\,\,\Sigma_\eta(t). 
	\end{equation}
	For that we define 
	the tubular neighborhood of $\partial\Omega^{2}_{o}$ as
	\begin{equation}\label{safedis}
		\begin{array}{l}
			N^b_a=\{\varphi(x)+\nu(x)z; \,\,x\in\Gamma, z\in(a_{\partial\Omega},b_{\partial\Omega}),
		\end{array}
	\end{equation}
	where we assume that 
	\begin{equation}\label{tblrnbd}
		\displaystyle	\partial\Omega\cap N^{b}_{a}=\emptyset.
	\end{equation}
	%


	Contrary to the representation of the solid with respect to the reference configuration, the fluid velocities, $u_{i}$ for $i\in\{1,2\}$ will be written in the time varying domains ${\Omega^{i}(t)}$, $i\in\{1,2\}$. In a nutshell the solid is described in $\mathit{Lagrangian\,\, coordinates}$ whereas the fluids are represented in $\textit{Eulerian\,\, coordinates}.$
	
	We consider the following model: 
	\begin{align}
		&\displaystyle( \partial_{t}{ u_{i}}+ \mbox{div}({u_{i}}\otimes u_{i}))
		-\mbox{div}\,T_{i}=0\quad \displaystyle\mbox{in} \quad \bigcup\limits_{t\in(0,T)}{\Omega^{i}(t)}\times\{t\},\label{momentumbalanceheat1}\\[1.mm]
		&\displaystyle \mbox{div}\,u_{i}=0\quad \displaystyle\mbox{in} \quad \bigcup\limits_{t\in(0,T)}{\Omega^{i}(t)}\times\{t\},\label{divfreeheat1}\\[1.mm]
		&\displaystyle c_i\big(\partial_{t}\theta_{i}+\mbox{div}(\theta_{i} u_{i})\big)-k_i \Delta\theta_{i}= {\mu_{i}(\theta_{i})}|D u_{i}|^{2}\quad \displaystyle\mbox{in} \quad \bigcup\limits_{t\in(0,T)}{\Omega^{i}(t)}\times\{t\},\label{entropyequation1}\\[1.mm]
		&\displaystyle\partial_{tt}\eta+
		DK(\eta)
		=G\quad \displaystyle\mbox{on}\quad \Gamma\times(0,T),\label{plateeqheat1}\\[1.mm]
		&\displaystyle u_{i}\circ\varphi_{\eta}=\partial_{t}\eta\nu\quad \displaystyle\mbox{on}\quad \Gamma\times(0,T),\label{interfaceveheatl1}\\[1.mm]
		&\displaystyle u_{1}=0\quad\mbox{on}\quad \partial\Omega\times(0,T),\label{NVS1heat1}\\[1.mm]
		&\displaystyle k_{i}\partial_{\nu_{\eta}}(\theta_{i})=\lambda(\theta_{1}-\theta_{2})\quad\mbox{on}\quad \bigcup\limits_{t\in(0,T)}\Sigma_{\eta}\times\{t\},\label{transmissioncond}\\[1.mm]
		&\displaystyle \partial_{\nu}\theta_{1}=0\quad\mbox{on}\quad \partial\Omega\times(0,T),\\[1.mm]
		&\displaystyle (u_{i},\theta_{i},\eta,\partial_{t}\eta)(\cdot,0)=(u^{0}_{i},\theta^{0}_{i},\eta^{0},\eta^{0}_{1}),\label{initialcondheat1}
	\end{align}
	where $e_{i}$ denotes the internal energy equation of the $i-$th fluid and \eqref{entropyequation1} models the balance of internal energy. The stress tensor of the $i-$th fluid is denoted by $T_{i}$ and is defined as:
	\begin{equation}
		\begin{array}{l}
			T_{i}=T_{i}(u_{i},p_{i})=\mu_{i}(\theta_{i})D({ u_{i}})-p_{i}\mathbb{I}_{d},
		\end{array}
	\end{equation}
	where $$D(u_{i})=\frac{\nabla u_{i}+\nabla^{T} u_{i}}{2}$$
	denotes the symmetric part of the gradient and $\mathbb{I}_{d}$ is the $d\times d$ identity matrix. The viscosity $\mu_{i}(\theta_{i})$ of the $i-$th fluid is given by the celebrated \textit{Vogel-Fulcher-Tammann equation:}
	\begin{equation}\label{viscosity}
		\begin{array}{l}
			\displaystyle\mu_{i}(\theta_{i})=\mu^{i}_{0}\cdot e^{\frac{\beta^{i}}{(\theta_{i}-\gamma^{i})}},
		\end{array}
	\end{equation}
	where $\mu^{i}_{0}>0,$ $\beta^{i}>0$ and $\gamma^{i}>0$ are empirical material-dependent parameters and are constants for individual fluids. The constant $\lambda\geqslant0$ appearing in \eqref{transmissioncond} stands for the thermal conductivity of the elastic structure and the heat conductivity of the $i-$th fluid ${k}_{i}$ is a positive constant.
		The elastic plate is driven by the resultant force due to the two fluids:
		\begin{equation}\label{defG1}
			\begin{array}{l}
				G=-|\det(\nabla\varphi_{\eta})|\left[(-T_{1}+T_{2})\circ\varphi_{\eta}\right]\nu_{\eta}\cdot\nu,
			\end{array}
		\end{equation} 
		where $\nu_\eta$ is the normal of $\Sigma_\eta$.
		One recalls the assumption that the structure moves normal to the channel direction. By $K(\eta)$ we mean the non-linear Koiter energy penalizing stretching and bending of the deformation, which will be introduced in Section~\ref{GeometryKoi}. By $DK(\eta)$ we denote its Frechet derivative.
		
		In the present article we are interested in the analysis of the FSI problem (\eqref{momentumbalanceheat1}--\eqref{initialcondheat1} with all possible non negative values of $\lambda$ which correspond to the following situations respectively:
		\begin{equation}
			\label{eq:heat}
			\left\{ \begin{array}{ll}
				\lambda=0\,& \mbox{insulating material, i.e. no heat flux, }\,\,\partial_{\nu_{\eta}}(\theta_{i})=0\,\,\mbox{on}\,\,\bigcup\limits_{t\in(0,T)}\Sigma_{\eta}\times\{t\},\\[2.mm]
				\lambda=	\infty\,& \mbox{heat superconductor, i.e.
					identical temperatures}\\
				&\mbox{on both sides of the structure, }\,\,\theta_{1}=\theta_{2}\,\,\mbox{on}\,\bigcup\limits_{t\in(0,T)}\Sigma_{\eta}\times\{t\},\\[2.mm]
				\lambda \in (0,\infty) \,& \mbox{finite heat conductance of the elastic material}.
			\end{array}\right.
		\end{equation}
		The relation $k_{1}\partial_{\nu_{\eta}}\theta_{1}=k_{2}\partial_{\nu_{\eta}}\theta_{2}$ represents the fact that the structure do not store any energy and hence outward heat flux from one of the fluids equals the inward heat flux of the other.\\
		In relation to a more general notation we introduce the internal energy of the $i-$ th fluid $e_{i}=e_i(\theta)$  and the internal energy heat flux $q_{i}=q_{i}(\nabla\theta_{i})$ by 
		\begin{equation}\label{ei1}
			e_{i}(\theta_{i})=c_{i}\theta_{i}
			\text{ and }
			q_{i}=q_{i}(\nabla\theta_{i})=-{k}_{i}\nabla\theta_{i},
		\end{equation}
		where $c_i,k_i$ are given positive numbers.
		The term $\mu_{i}(\theta_{i})$ appearing in the right hand side of \eqref{entropyequation1} represents the contribution of the fluid dissipation which acts as a source term to the heat evolution of the fluid.
		Further the entropy of the $i-$th fluid is denoted by $s_{i}$ and admits of the following dependence on $\theta_{i}:$
		\begin{equation}\label{etai1}
			\begin{array}{l}
				\displaystyle s_{i}=c_{i}\ln\theta_{i}.
			\end{array}
		\end{equation}
		In view of the constitutive relations \eqref{ei1} and \eqref{etai1}, the internal energy balance equation \eqref{entropyequation1} furnishes  the following entropy balance:
		\begin{equation}\label{rephraseentropyinq1}
			\begin{array}{l}
				\displaystyle \partial_{t}s_{i}+\mbox{div}(s_{i}u_{i})-k_i\Delta s_i= \tilde{\mu}_i(s_i)|D u_{i}|^{2}+k_i|\nabla s_i |^{2}\quad \displaystyle\mbox{in} \quad \bigcup\limits_{t\in(0,T)}{\Omega^{i}(t)}\times\{t\},
			\end{array}
		\end{equation}
		where $\tilde{\mu}_i(s):=\frac{\mu_i( e^s)}{e^s}$.
		In the present article as a part of our existence result we will recover both the heat evolution \eqref{entropyequation1} and the entropy equation \eqref{rephraseentropyinq1} as inequalities. On the other hand the discretization strategy used in the present article allows us to obtain heat and entropy balance equalities until the penultimate discretization. This makes our strategy suitable for numerical implementations.
		\subsection{Definition of a bounded energy weak solution and main result} 
		First we introduce \textit{bounded energy weak solutions} concerning the system \eqref{momentumbalanceheat1}-\eqref{initialcondheat1}. 
		\begin{definition}\label{WSDef}
			The triplet $(u_{i},\theta_{i},\eta)$ ($i\in\{1,2\}$) is a suitable weak solution to the problem \eqref{momentumbalanceheat1}-\eqref{initialcondheat1} if
			\begin{equation}\label{regularity}
				\begin{array}{ll}
					&\displaystyle u_{i}\in L^{\infty}(0,T;L^{2}(\Omega^{i}))\cap L^{2}(0,T;W^{1,2}(\Omega^{i})),\,\,\mathrm{div}\,u_{i}=0\,\,\mbox{in}\,\,\Omega^{i},\\
					&\displaystyle\theta_{i}\in L^{\infty}(0,T,L^{1}(\Omega^{i}))\cap L^{q}(\Omega^{i}\times(0,T))\cap L^{s}(0,T;W^{1,s}(\Omega^{i})),\,\,\mbox{for all}\,\,(q,s)\in[1,\frac{5}{3})\times [1,\frac{5}{4}),\\
					&\displaystyle \eta \in L^{\infty}(0,T;W^{2,2}(\Gamma))\cap L^{2}(0,T;W^{2+\sigma^{*},2}(\Gamma)),\,\,\mbox{for}\,\, 0<\sigma^{*}<\frac{1}{2},\\
					&\displaystyle \partial_{t}\eta\in L^{\infty}(0,T,L^{2}(\Gamma))\cap L^{2}(0,T,W^{\sigma^{*},2}(\Gamma)),\,\,\mbox{for}\,\, 0<\sigma^{*}<\frac{1}{2},
				\end{array}
			\end{equation}
			and the following hold.\\
			(i) A decomposition of $\Omega$ of the form \eqref{Omegat12} with \eqref{varphieta}-\eqref{Sigmat} hold.\\[2.mm]
			(ii) {\textit{The momentum balance holds in the sense:}}\\
			\begin{equation}\label{momentplvlhmain}
				\begin{array}{ll}
					&\displaystyle \int_{0}^{t}\bigg(\langle DK(\eta),b\rangle-\int_{\Gamma}\partial_{t}\eta\partial_{t}b\bigg)\\
					&\displaystyle+\sum_{i}\int^{t}_{0}\bigg(-\int_{\Omega^i} u_{i}\cdot(\partial_{t}\psi+u_{i}\cdot\nabla\psi)+\int_{\Omega^{i}}\mu_{i}(\theta_{i})D(u_{i}):D(\psi)\bigg)\\
					&\displaystyle=-\langle \partial_{t}\eta(t,\cdot)b(t,\cdot)\rangle +\langle \eta_{1}^{0},b(0,\cdot)\rangle+\sum_i \int_{\Omega^i(0)} u^{0}_{i}\cdot \psi(\cdot,0)\int_{\Omega^i(t)} \big(u_{i}(t,\cdot)\cdot\psi(t, \cdot)+  u^{0}_{i}\cdot \psi(0,\cdot)
				\end{array}
			\end{equation}
			for $a.e.$ $t\in (0,T),$ for all $b\in L^{\infty}(0,T;W^{2+\sigma^{*},2}(\Gamma))\cap W^{1,\infty}(0,T; L^{2}(\Gamma)),$  $\psi\in C^{\infty}_0([0,T]\times\Omega)),$ $\mathrm{div}\,\psi=0$ and $\Tr_{\Sigma_{\eta}}\phi=b\nu$ (the notion of trace $\Tr_{\Sigma_{\eta}}$ will be made clear in Lemma \ref{Lem:TrOp}).\\[2.mm]
			(iii) The coupling of $u$ and $\partial_{t}\eta$ reads $\mathrm{tr}_{\Sigma_{\eta}}u_1=\partial_{t}\eta \nu=\mathrm{tr}_{\Sigma_{\eta}}u_2,$ where the operator $\mathrm{tr}_{\Sigma_{\eta}}$ denotes the trace of a function on the moving interface $\Sigma_{\eta}$ and will precisely be introduced in Lemma \ref{Lem:TrOp}.\\[2.mm] 
			(iv) \textit{The heat evolution holds in the sense: }\\
			\begin{equation}\label{heatfnlhlevmain}
				\begin{array}{ll}
					&\displaystyle \sum_i \left(c_{i}\int^{t}_{0}\frac{d}{dt}\int_{\Omega^i} \theta_{i}\cdot \zeta_i-c_{i}\int^{t}_{0}\int_{\Omega^i} \theta_{i}\cdot (\partial_{t}+u_{i}\cdot\nabla)\zeta_i+k_{i}\int^{t_{1}}_{0}\int_{\Omega^{i}}\nabla\theta_{i}\cdot\nabla\zeta_i\right)
					\\
					&\displaystyle +\lambda\int^{t}_{0}\int_{\Sigma_{\eta}}(\theta_{1}-\theta_{2})(\zeta_1-\zeta_2)\geqslant\sum_i\int^{t}_{0}\int_{\Omega^{i}}\mu_{i}(\theta_{i})|Du_{i}|^{2}\zeta_i,
				\end{array}
			\end{equation}
			if $\lambda\in [0,\infty)$
			for $a.e.$ $t\in [0,T]$ and for $\zeta_i\in C^{\infty}([0,T]\times \Omega_i;[0,\infty))$, if ``$\lambda=\infty$" we mean that $\theta_1=\theta_2$ on $\Sigma_\eta$ and
			\begin{equation}\label{heatfnlhlevmain*}
				\begin{array}{ll}
					&\displaystyle \sum_i \left(c_{i}\int^{t}_{0}\frac{d}{dt}\int_{\Omega^i} \theta_{i}\cdot \zeta-c_{i}\int^{t}_{0}\int_{\Omega^i} \theta_{i}\cdot (\partial_{t}+u_{i}\cdot\nabla)\zeta+k_{i}\int^{t_{1}}_{0}\int_{\Omega^{i}}\nabla\theta_{i}\cdot\nabla\zeta\right)
					\displaystyle
					\geqslant\sum_i\int^{t}_{0}\int_{\Omega^{i}}\mu_{i}(\theta_{i})|Du_{i}|^{2}\zeta,
				\end{array}
			\end{equation}
			for $a.e.$ $t\in [0,T]$ and for $\zeta\in C^{\infty}([0,T]\times \Omega;[0,\infty))$.
			\\[2.mm] 
			(v) \textit{Positivity of temperature:}\\
			\begin{equation}\label{postempht0levmain}
				\begin{array}{l}
					\displaystyle \theta_{i}\geqslant \gamma\,\,\mbox{in}\,\,\Omega^{i}\times(0,T)\,\,\mbox{provided}\,\,\theta^{0}_{i}\geqslant\gamma>\gamma_i>0\,\,\mbox{on}\,\,\Omega^{i}_{0}.
				\end{array}
			\end{equation}
			(vi) \textit{The
				total energy of the system is conserved:}\\
			\begin{equation}\label{energybalhlevmain}
				\begin{array}{ll}
					\mathbb{E}_{tot}(\theta_{i},u_{i},\partial_{t}\eta,\eta)(t)=\mathbb{E}_{tot}(\theta^{0}_{i},u^{0}_{i},\eta^{0},\eta^{0}_{1}),
				\end{array}
			\end{equation}
			where the total energy of the system is defined as
			\[\mathbb{E}_{tot}(\theta_{i},u_{i},\partial_{t}\eta,\eta):=\sum^{2}_{i=1}\int_{\Omega^{i}}e_i +\int_{\Gamma}\bigg(\frac{|\partial_{t}\eta|^{2}}{2}+K(\eta)\bigg),
			\]
			with \[
			e_i:=c_{i}\theta_{i}+\frac{|u_{i}|^{2}}{2}\]
			(vii) \textit{A general entropy inequality holds in the following sense:}\\
			\begin{equation}\label{entrpybalhlevmain}
				\begin{array}{l}
					\displaystyle\sum_{i=1}^{2}\left(c_{i}\frac{d}{dt}\int_{\Omega^{i}}\varphi(\theta_{i})+k_{i}\int_{\Omega^{i}}|\nabla\theta_{i}|^{2}\varphi''(\theta_{i})\right)+\lambda\int_{\Sigma_{\eta}}\left(\theta_{1}-\theta_{2}\right)\left(\varphi'(\theta_{1})-\varphi'(\theta_{2}))\right)\geqslant 0,
				\end{array}
			\end{equation}
			for any $\varphi(\theta)$ such that $\varphi(\cdot)$ is concave and monotone for $\theta>0,$ such that all the integrals above make sense.
			
		\end{definition}
		\begin{remark}[Compatibility to strong solutions]
			Note that by the above definition of a weak bounded energy solution possessing the necessary regularity would actually become a strong solution. For most parts it follows by the usual integration by parts procedure. This does however at first only imply that 
			\[
			c_i\big(\partial_{t}\theta_{i}+\mathrm{div}(\theta_{i} u_{i})\big)-k_i \Delta\theta_{i}\geq {\mu_{i}(\theta_{i})}|D u_{i}|^{2}\quad \displaystyle\mbox{in} \quad \bigcup\limits_{t\in(0,T)}{\Omega^{i}(t)}\times\{t\},
			\]
			However the total energy balance actually implies that equality is satisfied.
			This can be seen by contradiction. 
			Assume that there is a subset of positive measure $A\subset [0,t)\times \Omega^{i}$, where $c_i\big(\partial_{t}\theta_{i}+\mathrm{div}(\theta_{i} u_{i})\big)-k_i \Delta\theta_{i}> {\mu_{i}(\theta_{i})}|D u_{i}|^{2}$. Then this implies that
			\[
			\sum^{2}_{i=1}\bigg(\int_{\Omega^{i}(t)}\theta_i(t)
			-\int_{\Omega^{i}_{0}}\theta_i^0\bigg)
			>\sum_i\int^{t}_{0}\int_{\Omega^{i}}\mu_{i}(\theta_{i})|Du_{i}|^{2}.
			\]
			Using this inequality, one infers:
			\begin{equation}\nonumber
				\begin{array}{ll}
					&\mathbb{E}_{tot}\displaystyle(\theta_{i},u_{i},\partial_{t}\eta,\eta)(t)-\mathbb{E}_{tot}(\theta^{0}_{i},u^{0}_{i},\eta^{0},\eta^{0}_{1})\\[2.mm]
					&\displaystyle =\bigg\{\sum^{2}_{i=1}\int_{\Omega^{i}}\bigg(c_{i}\theta_{i}+\frac{|u_{i}|^{2}}{2}\bigg)+\int_{\Gamma}\bigg(\frac{|\partial_{t}\eta|^{2}}{2}+K(\eta)\bigg)-\\
					&\displaystyle\qquad\quad\sum^{2}_{i=1}\int_{\Omega^{i}(0)}\bigg(c_{i}\theta^{0}_{i}+\frac{|u^{0}_{i}|^{2}}{2}\bigg)+\int_{\Gamma}\bigg(\frac{|\eta^{0}_{1}|^{2}}{2}+K(\eta^{0}\bigg)
					\bigg\}\\[2.mm]
					&\displaystyle > \sum^{2}_{i=1}\bigg(\int_{\Omega^{i}}\frac{|u_{i}|^{2}}{2}-\int_{\Omega^{i}(0)}\frac{|u^{0}_{i}|^{2}}{2}\bigg)+\int_{\Gamma}\bigg(\frac{|\partial_{t}\eta|^{2}}{2}+K(\eta)-\frac{|\eta^{0}_{1}|^{2}}{2}-K(\eta^{0})\bigg)+\sum^{2}_{i=1}\int^{t}_{0}\int_{\Omega^{i}}\mu_{i}(\theta_{i})|Du_{i}|^{2}\\[2.mm]
					&\displaystyle=0,
				\end{array}
			\end{equation}
			where the last line in the previous calculation is obtained from the energy dissipation identity obtained by testing the mechanical equation with the test function $(u_1,u_2,\partial_t \eta).$
			This contradicts the total energy balance \eqref{energybalhlevmain}.
			This strikingly simple observation was found by Feireisl et. al for the compressible Navier-Stokes Fourier system~\cite{FeiNov} and was already used in the context of fluid-structure interactions in \cite{Breit}.
		\end{remark}
		\begin{remark}
			The fluid-structure interaction system \eqref{momentumbalanceheat1}-\eqref{initialcondheat1} models the governing dynamics of some engineering devices known as `heat exchangers' which are designed to allow transfer of heat between two incompressible fluids (via an elastic/ rigid interface of high thermal conductivity) without allowing them to mix. A variety of these devices are used in different industries: petroleum refineries, furnaces and cryogenics, condensers in power plant, etc.
		\end{remark}
		\begin{remark}[Pressure from temperature]
			Certainly it would be very physical to include some pressure force in dependence of the temperature in the momentum equation of the fluid, for example $c_i \nabla \theta_i$ or similar. As the term can be included without changing any of the arguments of construction we excluded it for the sake of better readability.
		\end{remark}
		
		\begin{remark}
			In view of the constitutive relations \eqref{ei1} and \eqref{etai1}, and the internal energy balance equation \eqref{entropyequation1}, one can formally obtain the entropy identity \eqref{rephraseentropyinq1}.
			Indeed integrating \eqref{rephraseentropyinq1} in the spatial domain $\Omega^{i},$ using by parts integration and Reynold's transport theorem and summing the resulting expressions over $i\in\{1,2\}$ one obtains the inequality \eqref{entrpybalhlevmain} with $\varphi(\theta_{i})=\log(\theta_i)=s_{i}.$ The lack of  compactness of $\nabla u_{i}$ only allows us to recover the entropy inequality \eqref{entrpybalhlevmain} (using weak lower semi-continuity) and not an identity similar to \eqref{rephraseentropyinq1}.
			In addition to the physical entropy \eqref{etai1}, the inequality \eqref{entrpybalhlevmain} also includes entropy functions of the form $\varphi(\theta_{i})=\theta_{i}^{\beta+1},$ for $\beta\in(-1,0),$ furnishing estimates of the temperatures in fractional order Sobolev spaces. These estimates serve as a key tool to conclude the strong compactness of $\theta_{i}.$
		\end{remark}
		\begin{remark}[Entropy-inequality]
			The heat evolution \eqref{heatfnlhlevmain} holds as an inequality due to the lack of compactness of $\nabla u_{i}.$ Despite of this fact the discretization strategy used in the present article allows us to obtain heat and entropy balance as identities at all the discrete levels. 
		\end{remark}
		
		Let us now state the main theorem of the present article on the existence of bounded energy weak solutions to the model \eqref{momentumbalanceheat1}-\eqref{initialcondheat1}.
		\begin{theorem}\label{Th:main}
			Assume that:\\[1.mm] 
			$(i)$ $\Omega\subset\mathbb{R}^{3}$ is a given bounded domain with $C^{\infty}$ boundary $\partial\Omega.$ Further suppose $\Omega^{2}_{o}\Subset \Omega$ (as in Figure $1.$) such that $\dist(\partial\Omega^{2}_{o},\partial\Omega)>0$ and $\partial\Omega^{2}_{o}$ (the boundary of $\Omega^{2}_{o}$) is parametrized by a $C^{4}$ injective mapping $\varphi:\Gamma\rightarrow \mathbb{R}^{3},$ where $\Gamma\subset \mathbb{R}^{2},$ via $\partial\Omega^{2}_{o}=\varphi(\Gamma).$\\[2.mm]
			$(ii)$ \textit{Regularity on initial data:} $\eta_{0}\in W^{2,2}(\Gamma),$ $\eta^{0}_{1}\in L^{2}(\Gamma),$ $\theta^{0}_{i}\in L^{1}(\Omega^{i}_{0})$ and $u^{0}_{i}\in L^{2}(\Omega^{i}_{0}).$ \\[2.mm]
			$(iii)$ Given initial data $\theta^{0}_{i},$ $\eta^{0},$ $\Omega=\Omega^{1}_{0}\cup\overline{\Omega^{2}_{0}},$ with $\partial\Omega^{2}_{0}=\Sigma_{\eta^{0}}=\varphi_{\eta^{0}}(\Gamma,0)$ and $\partial\Omega^{1}_{0}=\Sigma_{\eta^{0}}\cup\partial\Omega.$\\[2.mm]
			$(iv)$ \textit{Positivity of initial temperature:} $\theta^{0}_{i}\geqslant\gamma>\gamma_i>0,$.\\[2.mm]
			$(v)$ $\eta_{0}\in (a_{\partial\Omega},b_{\partial\Omega})$ (where the existence of the numbers $a_{\partial\Omega}$ and $b_{\partial\Omega}$ is asserted in \eqref{bjtparam}) holds, \eqref{tblrnbd}-\eqref{safedis} are true and $\overline{\gamma}(\eta_{0})>0$ where the geometric quantity $\overline{\gamma}$ is related with the structure of the Koiter energy (will be introduced in \eqref{ovgamma}).\\[2.mm]
			Then there is $T_F\in(0,\infty]$ and a weak solution to the problem \eqref{momentumbalanceheat1}-\eqref{initialcondheat1} along with the non-linear Koiter energy (the details of the structure of the Koiter energy is presented in \eqref{Koiterenergy}) on the interval $(0,T)$ for any $T<T_F$ in the sense of Definition \ref{WSDef}.\\
			Finally, $T_F$ is finite only if 
			\begin{equation}\label{degenfstkind}
				\text{ either }\lim_{s\to T_F}\eta(s,y)\searrow a_{\partial\Omega}\text{ or }\lim_{s\to T_F}\eta(s,y)\nearrow b_{\partial\Omega}
			\end{equation}
			for some $y\in\Gamma$ or the Koiter energy degenerates, i.e., 
			\begin{equation}\label{degensndkind}
				\lim_{s\to T_F}\bar\gamma(\eta(s,y))=0 
			\end{equation}
			for some $y\in\Gamma.$
		\end{theorem}
		
		\subsection{Novelty and significance}
		
		In the framework of weak solutions to fluid-structure interactions there are only a few works that consider temperature effects. On the one hand thermo-elastic plates were considered~\cite{TW,TW2} both interacting with an incompressible and a compressible fluid. On the other hand the compressible Navier-Stokes-Fourier system was coupled to a non-linear elastic shell~\cite{Breit} or a thermo-elastic, linear and viscous shell~\cite{MMNRT}. The authors constructed some weak concept of solution up to the point of degeneracy. These works include temperature in the fluid or in the structure or in both. However an existence analysis taking into account the {\em interacting effect of temperature} seems to be missing so far. An exception is maybe~\cite{FMNT} that considers the adiabatic piston problem in the one dimensional setting.
		The aim of this article is to start investigating on this very relevant topic in the physical space of three dimensions. For that we chose a setting where temperature would follow the classical heat equation. Unfortunately this setting is still excluded in the theory for weak solutions for compressible fluids. But as at small speeds the dynamic is expected to be incompressible even when temperature effects are leading order we decided for an incompressible Navier-Stokes-Fourier system~\cite{bulicek}.
		
		Generally we follow an energetic point of view. Already in the construction we compute the heat distribution resulting from the given fluid dissipation. The time-changing interaction surface, which is the shell is assumed to posses a certain heat capacity. The capacity is related to a potential. Here for the sake of simplicity we take a square potential, but the methodology is independent of this structure. By introducing a conduction parameter $\lambda$ the potential ranges in between super conductivity and insulation, which are relevant endpoints that are included in the theory~\eqref{eq:heat}.
		
		The construction follows the variational approach for fluid-structure interaction and hyperbolic PDE developed in~\cite{BenKamSch}. The approach was successfully used before in various scenarios including solid only applications including inertia~\cite{CesGraKam24,CesSch25}, temperature~\cite{ABFS} and fluid-structure interactions of various type~\cite{BenKamSch,BenKamSch2,BreitKamSch,malSpSch}.
		In particular it allows to derive energetically the heat equation from the very start~\cite{ABFS}. 
		
		Technically the procedure is here rather involving, in particular since several artificial dissipation terms are collectively transferred into heat. Indeed, even so the construction imitates the variational physical background three layers of approximation seem to be necessary. It means besides the velocity scale $\tau$ and the acceleration scale $h$ another regularizing  parameter $\kappa$. This is in contrast to~\cite{BenKamSch}, where regularization was necessary but could be coupled to the acceleration scale.

		The first result following the variational approach of fluid-structure interaction in the regime of thin deformable and elastic solids was~\cite{malSpSch}. Here some non-linear beam including tangential deformations inside a two dimensional fluid was considered. While there the focus was on obtaining tangential deformations, in the present work the focus is on temperature effects. Actually since the present paper involves non-linear Koiter shells it methodologically relates also to \cite{Breit,MuhaSch,MitNes}. A technical highlight of the present paper are therefore the adaptation of the compactness theorems developed in \cite{MuhaSch} to the variational setup including temperature. 
		
		In conclusion, the paper introduces a novel energetic scheme including temperature for fluid-structure interactions. Moreover, all necessary technical innovations are performed. This includes respective regularity estimates for the temperature and the shell displacement, as well as compactness results for time-averages. 
		
		\subsection{Ideas, strategy and some further comments} \underline{\textit{Comments on the choice of some test functions:}} The system \eqref{momentumbalanceheat1}-\eqref{initialcondheat1} models the heat transfer dynamics between two Navier-Stokes-Fourier fluids separated by an elastic inteface. The elastic interface possessing high thermal conductivity facilitates heat transfer (resulting in heating or cooling of the fluid streams) without allowing them to mix. The viscosity coefficient of the fluids involved follow a physical law given by \textit{Vogel-Fulcher-Tammann equation,} \eqref{viscosity}. In order to prevent the blow-up of this viscosity coefficients one needs to guarantee that the fluid temperatures do not degenerate to zero provided they have a positive lower bound initially. Informally to prove this \textit{minimal principle} one tests \eqref{entropyequation1} by $$\omega_{i}(x,t)=\mathcal{X}_{[0,\tau]}(t)\min\bigg(0,\theta_{i}-\gamma\bigg)\leqslant 0,$$ 
		sum over $i\in\{1,2\},$ and obtains
		\begin{equation}\nonumber
			\begin{array}{ll}
				&\displaystyle\frac{1}{2}\int_{0}^{\tau}\partial_{t}\bigg(\|\omega_{1}\|^{2}_{L^{2}(\Omega^{1}(t))}+\|\omega_{2}\|_{L^{2}(\Omega^{2}(t))}\bigg)+\int_{0}^{\tau}\bigg(\int_{\Omega^{1}(t)}|\nabla \omega_{1}|^{2}+\int_{\Omega^{2}(t)}|\nabla \omega_{2}|^{2}\bigg)\\
				&\displaystyle+\lambda\int^{\tau}_{0}\int_{\Sigma_{\eta}}(\theta_{1}-\theta_{2})\cdot\bigg(\omega_{1}-\omega_{2}\bigg)\leqslant 0.
			\end{array}
		\end{equation}
		As soon as we realize that in the last inequality, the third term in the left hand side is always non-negative, we infer that $\theta_{i}\geqslant\gamma>0$ (for some positive constant $\gamma$), provided the same holds initially. Although the underlying idea remains the same, it is a little more intricate to prove a similar \textit{minimal principle} at discrete levels.\\ 
		Further in order to deal with the nonlinear dependence of the viscosity coefficients on the temperature one requires the regularity of the temperatures beyond the information furnished by the energy balance. The usual strategy (inspired from \cite{bulicek}) is to test the internal energy balance equation by $\varphi'(\theta_{i}),$ where $\varphi(\theta_{i})=\theta_{i}^{\beta+1}$ (often termed as mathematical entropy) is concave and monotone. Consequently one obtains a uniform bound for $\theta_{i}$ in $L^{s}(W^{1,s}),$  for $s\in[1,\frac{5}{4}).$\\
		Concerning the proof of the existence of a weak solution to our system, we use a variational approximation method motivated from \cite{BenKamSch}. Compared to other  strategies involving Galerkin approximation or a fixed point argument, a variational approach seems to be more natural to deal with an non-linear, non-convex elastic energy functional and the temperature flux. Nevertheless, the involvement of the Fourier equation along with the heat exchange dynamics makes our variational approximation extremely delicate. In the following discussion, let us detail our strategy which extends  $(i)$ \textit{minimizing movements approximation} and $(ii)$ \textit{a two time-scale} approach, (previously introduced in \cite{BenKamSch} for FSI problems without the Fourier equation in the context of heat exchange between two Navier-Stokes-Fourier fluids via an elastic interface with a non-linear, non-convex Koiter energy.\\[2.mm]
		\underline{\textit{Two-time scale approach and minimizing movements:}}
		$\textbf{(i)}$ To begin with, one introduces a time delayed approximation, discretising the inertial/ hyperbolic terms $\partial_{tt}\eta$ and $\frac{Du_{i}}{dt}=( \partial_{t}{ u_{i}}+ \mbox{div}({u_{i}}\otimes u_{i}))$ (the material derivative), in the coupled momentum equation in form of difference quotients involving a time delay parameter $h>0,$ (the acceleration scale) $c.f.$ \eqref{discretemomen1}. This is done by subsequently introducing the flow maps
		\begin{equation}\nonumber
			\left\{ \begin{array}{ll}
				&\displaystyle\partial_{t}\Phi_{i}(y,t)=u_{i}(t+s,\Phi_{i}(y,t)),\\
				&\displaystyle \Phi(\cdot,0)=Id|_{\Omega^{i}_{0}},
			\end{array}\right.
		\end{equation}
		where the role of $\Phi_{i}$ is simply to transport the fluid domains along the respective velocity fields $u_{i}.$ The existence of a unique solution to the last ODE is guaranteed by a higher order artificial dissipation $\displaystyle\kappa\int^{t}_{0}\int_{\Omega^{i}(t)}\nabla^{k_{0}}u_{i}:\nabla^{k_{0}}\psi,$ ($\psi$ is a suitable test function), for $k_{0}>3$ added to the Euler-Lagrange equation ($c.f.$ \eqref{discretemomen1}). Using the same parameter $\kappa$ we further regularize the Koiter energy as: $\displaystyle K_{\kappa}(\eta)=K(\eta)+\frac{\kappa}{2}\int_{\Gamma}|\nabla^{k_{0}+1}\eta|^{2},\,\,\mbox{for}\,\,k_{0}>3.$ This regularization plays a crucial role to prove the compactness of $\displaystyle\int^{t}_{0}\langle DK_{\kappa}(\eta^{h}),b^{h}\rangle,$ for suitably constructed test functions $b^{h}$ ($c.f.$ Section \ref{ctestfn} for details). Indeed this artificial regularization is used only in the intermediate level (for a quasi-stedy evolution problem) and at the final stage we prove the compactness of the Frechet dervative of the Koiter energy by showing an improved regularity of the structural displacement $\eta^{\kappa}$. We further add an artificial structural dissipation $\displaystyle\kappa\int^{t}_{0}\int_{\Gamma}\nabla^{k_{0}+1}\partial_{t}\eta:\nabla^{k_{0}+1}b$ in the weak formulation, which furnishes enough regularity of $\partial_{t}\eta,$ needed to give a sense to the duality pairing $\langle DK_{\kappa}(\eta),\partial_{t}\eta\rangle.$ This term arises while testing the discrete momentum balance equation using the solution $(u_{i},\partial_{t}\eta)$ with a goal to furnish an energy balance. Note that unlike the momentum equation \eqref{discretemomen1} (where we use an \textit{Lagrangian approximation of the material time derivative} on the acceleration scale), we do not use the acceleration scale $h>0$ to discretize the heat convection ($c.f.$ \eqref{intenballh}). In a nut-shell for a fixed $h>0,$ we get a coupled time-delayed momentum balance and an unsteady heat evolution.\\[2.mm]
		\textbf{$\textbf{(ii)}$} In order to solve this approximate problem \eqref{discretemomen1} and \eqref{intenballh} with the correct interfacial coupling, we discretise the structural velocity $\partial_{t}\eta$ and the material derivative $(\partial_{t}\theta_{i}+u_{i}\cdot\nabla\theta_{i})$ of the temperature using a velocity scale $0<\tau<< h.$ Now we have reduced our problem to a completely steady set of equations, where we can use a simultaneous minimization process to show the existence of solutions. More precisely, concerning the momentum balance equation, given an interval $[0,h],$ we divide it into $N$ equidistant time steps of length $\tau,$ and then for $k\in\{0,....,N-1\},$ given $(\eta_{k}, \theta_{i,k})$ and  
		\begin{equation}\label{decompintro}
			\begin{array}{ll}
				&\displaystyle \Omega=\Omega^{1}_{k}\cup\overline{\Omega^{2}_{k}}\,\,
				\mathrm{such\,\, that}\,\,
				\partial\Omega^{2}_{k}=\varphi_{\eta_{k}}(\Gamma)=\Sigma_{\eta_{k}},\,\, \Omega^{1}_{k}=\Omega\setminus\overline{\Omega^{2}_{k}},\mbox{and}\,\,\partial\Omega^{1}_{k}=\Sigma_{\eta_{k}}\cup\partial\Omega,
			\end{array}
		\end{equation}
		we define $(u_{i,k+1},\eta_{k+1})$ (for $i\in\{1,2\}$) to be the minimizer of a suitable functional ($c.f.$ \eqref{min1}). Each minimization produces the subsequent deformation and a new interface is used in the next time step to separate the two fluids. Further an affine coupling condition $	\displaystyle\frac{\eta_{k+1}-\eta_{k}}{\tau}\nu=u_{i,k+1}\circ \varphi_{\eta_{k}}\,\,\mbox{on}\,\,\Gamma,$ holds, which is a discrete analogue of the adherence type condition \eqref{interfaceveheatl1} for the unsteady FSI system \eqref{momentumbalanceheat1}--\eqref{initialcondheat1}. The existence of a minimizer to the functional \eqref{min1} and the corresponding \textit{Euler-Lagrange} equation are stated as a part of Lemma \ref{lemmasolmin1}. The proof uses the \textit{direct method of calculus of variations} which can be adapted without any difficulty from \cite{BenKamSch, BreitKamSch}.\\[2.mm]
		\underline{\textit{Explaining the approximation of the heat equation using the velocity scale $\tau$:}} Next concerning the discrete heat equation, we solve a second minimization problem \eqref{min2}-\eqref{defmA}. At a first glance, the structure of the functional being minimized seems to be a little intricate. So let us discuss step by step the underlying idea behind our consideration. Consider a coupled system of two heat equations (for $i\in\{1,2\}$) in the domain $\Omega^{i}$ solving \eqref{Omegat12}-\eqref{Sigmat}  with interfacial heat exchange (see Figure 1. for the set-up):
		\begin{equation}\label{heatintro}
			\left\{ \begin{array}{ll}
				\displaystyle(\partial_{t}+u_{i}\cdot\nabla)\theta_{i}-\Delta\theta_{i}=g_{i}&\displaystyle\quad\mbox{in}\quad \bigcup\limits_{t\in(0,T)}{\Omega^{i}(t)}\times\{t\},\\
				\partial_{\nu}(\theta_{i})=\lambda(\theta_{1}-\theta_{2})&\displaystyle \quad \mbox{on}\quad \bigcup\limits_{t\in(0,T)}\Sigma_{\eta}\times\{t\},\\
				\partial_{\nu}\theta_{1}=0&\displaystyle \quad\mbox{on}\quad\partial\Omega\times(0,T),
			\end{array}\right.
		\end{equation}
		with a known velocity field $u_{i}.$ One first discretises the convection $(\partial_{t}+u_{i}\cdot\nabla)\theta_{i}$ using a time scale $\tau$ and to solve the approximate problem one considers:\\
		\begin{equation}\label{minheatintro}
			\begin{array}{ll}
				\displaystyle\min_{\theta_{i}}&\displaystyle\bigg\{\sum\limits_{i} \frac{ \tau }{2}\int_{\Omega^{i}_{k}}\bigg|\frac{\theta_{i}\circ\Psi_{u_i,{k+1}}-\theta_{i,k}}{\tau}\bigg|^{2}+\sum_{i}\frac{1}{2}\int_{\Omega^{i}_{k+1}}|\nabla\theta_{i}|^{2}
				+\frac{\lambda}{2} \int_{\partial\Omega^{2}_{k+1}}\left|\theta_{1}-\theta_{2}\right|^{2}\\
				&\displaystyle-\sum\limits_{i}\int_{\Omega^{i}_{k}}g_{i}\theta_{i}\circ\Psi_{u_{i},k+1}\bigg\},
			\end{array}
		\end{equation} 
		where $\Omega^{i}_{k}$ and $\Omega^{i}_{k+1}$ are determined by using the interfacial deformations $\eta_{k}$ and $\eta_{k+1}$ respectively, such that \eqref{decompintro} holds. In \eqref{minheatintro}, $\Psi_{u_{i}}$ is defined in $\overline{\Omega^{i}_{k}}$ and solves $\Psi_{u_{i}}=Id+\tau u_{i}(\cdot,\tau k).$
		Indeed the corresponding \textit{Euler-Lagrange equation} gives us the weak formulation to the discrete, time-delayed approximation of the heat evolution \eqref{heatintro}. Now in \eqref{min2}, $\mu_{i}^{M}(\theta_{i,k})|Du_{i,k+1}|^{2}$ (due to fluid dissipation) plays the role of the source term $g_{i}.$ Further compared to \eqref{minheatintro}, there appears a few extra terms in \eqref{min2}. 
		The fifth to eighth terms in the expression of $T_{k}(\theta_{1},\theta_{2})$ (c.f.\eqref{min2}) signifies the contribution of the artificial dissipations in the formation of heat.\\
		Since throughout our analysis we keep track of all the dissipative forcing (whether physical/ numerical), we  recover the \textit{conservation of total energy} at the approximation layers (first for fixed $h$ and next for $\kappa$). We refer to the obtainment of \eqref{hlevelenergy} and \eqref{energybalhlev} for more details. 
		After we write down the \textit{Euler-Lagrange equations} corresponding to the minimization problems (for \textit{momentum balance} and \textit{heat evolution}), we obtain apriori estimates for $(\eta_{k+1},u_{i,k+1},\theta_{i,k+1})$ by plugging suitable test functions into them ($c.f.$ Lemma \ref{bounduetakp1} and Lemma \ref{estthetakp1}). One next defines suitable interpolants ($c.f.$ Section \ref{interpole}) and pass $\tau\rightarrow 0$ in the equations solved by them using the information obtained from the apriori estimates. At this stage we have achieved the compactness of the fluid temperatures by using the classical \textit{Aubin-Lions theorem} in parabolic cylinders (possible since the interface is uniformly continuous) away from the moving interface ($c.f.$ Section \ref{strongconvtheta}).\\[2.mm] 
		Once we pass $\tau\rightarrow 0,$ we are left with two more parameters, \textit{viz.} the acceleration scale $h$ and the regularizing parameter $\kappa.$ While passing both of these remaining parameters to zero, an intricate part is to guarantee the strong convergences of the fluid and the structural velocities. Although for this part we do not rely on entirely novel arguments but their proofs require non trivial adaptations from the existing literature \cite{Breit, MuhaSch}. Especially we would like to refer to the proof of Proposition \ref{strngconvavg} where we show the compactness of the average velocities via a non standard version of `Theorem \ref{thmBbL}' of the celebrated \textit{Aubin-Lions compactness lemma} suitably tailored for time-delayed approximations of coupled PDE systems with time-variable geometry.\\[2.mm]
		During the final limit passage $\kappa\rightarrow 0,$ we need improved regularity of the structural displacement and velocity, $c.f.$ Lemma \ref{improvebndetadelta}. The idea behind the regularity improvement appeared in \cite{MuhaSch} and in the present article we adapt their result for our set-up.
		\subsection{Bibliographical remarks.}  In view of its wide range of applications in analysing the dynamics of aircraft wings, stability of bridges, blood flow through arteries etc., the mathematical analysis of fluid-structure interaction (FSI) problems has become an active area of research. The fluids involved in FSI models can show incompressible/ compressible nature depending on the Mach number. Majority of studies on the Cauchy theory of FSI models concerns the motion of a rigid body in a viscous incompressible Newtonian fluid described by Navier-Stokes equations ($cf.$ \cite{Conca,Desjardin,Feir3,Maday,TakTuc}). While most of these articles deal with solutions until any possible collision of the rigid body and the container wall, special weak solutions after  collision have been constructed in \cite{MarStrTuc,Feir3}. The contact issues (proving no-collision) further been investigated in \cite{Hillairet1} and \cite{HilTak}.\\
		The study of the dynamics of a bulk elastic structure within a viscous fluid is coparatively limited. Concerning the motion of a three-dimensional elastic structure interacting with a three dimensional fluid, we quote \cite{Galdi, Grandmont} for the stedy-state case and \cite{Cout1, Cout2, Kukavica, Raymond} for the full unsteady case. All the articles latter consider the existence of strong solutions for small enough data locally in time. The structure considered in \cite{Cout2, Raymond} are modelled by Lamé system of linear elasticity while \cite{Cout1} deals with quasilinear elastodynamics. Very recently one of the authors of the present article introduced a novel variational strategy in \cite{BenKamSch} (incompressible case), and \cite{BreitKamSch} (compressible case) to prove the existence of weak solutions to FSI models involving a Navier-Stokes fluid interacting with a bulk elastic structure. The bulk structures considered in these articles are called \textit{generalized standard materials}, for which the first Piola–Kirchhoff stress tensor $\sigma$ can be derived from underlying energy and dissipation potentials; i.e.
		$$\div\,\sigma=DE(\eta)+D_{2}R(\eta,\partial_{t}\eta),$$
		with $E$ being the energy functional describing the elastic properties while $R$ is the dissipation	functional used to model the viscosity of the solid. Here $D$ and $D_{2}$ denote respectively the Fr\'echet derivative
		and the Fr\'echet derivative with respect to the second argument. The \textit{two time scale} approach and the \textit{minimizing movement} approximation constructed in the present article can be considered as an extension of the works \cite{BenKamSch, BreitKamSch} designed to take into account the thermal effects.\\
		In the present article we are concerned about a fluid-shell interaction where the elastic structure appears at the fluid boundary. Most of the articles in the literature dealt with such a interactive model with the assumption that the solid involved is linear elastic/ visco-elastic. The first break through result in the direction of considering a full non-linear hyperbolic Koiter shell appeared in \cite{MuhaSch} (incompressible FSI) and later suitably adapted in \cite{Breit, MitNes} (compressible case). The idea of \cite{MuhaSch} was to suitably use the effect of the fluid dissipation in the governing dynamics of the structure. We would further like to quote some articles devoted to develope the Cauchy theory of similar FSI models. The article \cite{Chambolle} proves the existence of weak solutions of a 3D-2D fluid-plate intercative system when the structure is a damped plate in flexion. In a rather recent article \cite{MuCan}, the authors construct weak solutions to a non-linear FSI problem modelling a 2D incompressible, viscous fluid (driven by a time dependent pressure data) in a cylinder interacting with a deformable wall which is given by a linear viscoelastic Koiter shell equation. In \cite{MuCan}, the fluid and the structure are coupled via the Kinemetic and dynamic boundary conditions describing respectively the continuity of velocities and the balance of contact forces at the fluid solid interfacer. One of the technical novelties of \cite{MuCan} is the use of a semi-discrete operator splitting scheme to analyse the existence issues of  fluid-structure evolutionary dynamics. Concerning the existence of strong solutions for incompressible FSI models involving an elastic structure/ damped wave equation appearing at the fluid boundary, we refer the readers to \cite{Veiga, Leq}. The article \cite{Veiga}
		proves the local in time existence of strong solutions for a 2D-1D FSI model in a channel like geometry. Whereas a local in time existence of strong solutions for arbitrary data/ global in time strong solutions for small initial data for a 3D-2D FSI system is obtained in \cite{Leq}. In a non-flat geometry (very similar to that of ours), the existence of weak solutions (until possible degeneracies) of a FSI problem involving an incompressible Newtonian fluid interacting with a linear elastic Koiter shell is being shown in \cite{LeRu14}. The existence of strong solutions shown in \cite{Veiga, Leq} are local in nature, while the authors of the article \cite{GranHill} have shown that for a 2D-1D FSI model involving an incompressible fluid and a damped beam equation in a channel like geometry, the structure never touches the fluid boundary and the strong solution exist globally in time. We would further like to quote \cite{MaityRoy}, where the existene of a unique maximal strong solution of a FSI problem involving an incompressible fluid confined in a cylindrical domain and interacting with an elastic structure at the interface is being proved by using a semigroup approach. The interested readers may also consult \cite{AbelsLiu1, AbelsLiu2}, where the authors investigate the local in time existence of strong solutions of models of \textit{blood flow through arteries involving plaque formation} via maximal regularity theory of linearized equations and a suitable fixed point argument.\\
		Although limited compared to the incompressible FSI problems, in the past few years there has been some works devoted to the Cauchy theory of compressible fluid-structure interactions. Rather than providing a detailed review of the compressible FSI literature, let us quote \cite{MaityTakahahi, Mitra} (strong solutions with damped elastic structure), \cite{Avalos2} (semigroup well posedness with an undamped structure), \cite{RoyMaity} (strong solution with a wave equation), \cite{Mitra*} (strong solution and controllability issues) for the analysis of compressible fluid structure interaction models with the structure/ wave appearing at the fluid boundary.\\
		To the best of our knowledge there exist no literature dealing with the existence issues of systems involving heat exchange between two Navier-Stokes-Fourier fluids. Although the article \cite{MMNRT} proves the global existence of weak solutions of a compressible FSI model involving a thermal equilibrium between an elastic structure and a Navier-Stokes-Fourier fluid, the model they consider is way different from that of ours where we involve a heat conductive (with different values of the conductivity coefficients ranging from zero to infinity) elastic interface (with non-linear, non-convex Koiter energy) located between two incompressible Navier-Stokes-Fourier fluids. The system we consider models physical devices known as \textit{heat exchangers} widely used now a days in power plant broilers, HVAC systems, refrigeration and many other condensers and chillers. One can further consult \cite{Breit} and \cite{Haak} for analysis of compressible FSI models involving interaction between a Navier-Stokes-Fourier fluid with a thermally insulating elastic boundary. Apart from the context of fluid-structure interaction dynamics, we would like to mention a very interesting article \cite{Helemr}, proving the global in time existence of bounded energy weak solutions to the Maxwell-Stefan-Fourier equations in Fich-Onsager form. The model they consider involves heat flux through boundary which is proportional to the temperature difference of the fluid and the back ground temperatures (which has some resemblance with our situation). 
		\subsection{The structure of the paper.}
		\begin{enumerate}
			\item[Section 2:]Here we introduce the Koiter energy $K(\eta)$. Moreover, we will define the notion of trace for a moving domain, introduce embedding of Sobolev spaces in non-Lipschitz set-up, Solenoidal extension properties, notion of convergence and an Aubin-Lions type result.
			\item[Section 3:]In this section the double minimisation is performed on the so-called $\tau$ layer for a fixed acceleration scale $h$. In order to possess a positive temperature a couple of carefully chosen $\tau$-dependent regularizers are included. These suffice to derive uniform in $\tau,h$ and $\kappa$ estimates for $\theta$. A time-delayed solution is established including a weak equation of the temperature and an energy equality.
			\item[Section 4:]Here a regularised solution is established. This includes a subtle argument of Aubin-Lions for the time-averages of the velocities.
			\item[Section 5:] The finite limit passage is the removal of the regulariser, the so-called $\kappa$ level. For that higher regularity for the shell has to be shown. As in this setting it is not possible anymore to test the limit equation with the velocity the heat equation is lost and replaced by an entropy inequality.
		\end{enumerate}
		In the sections 3-5 we consider the case $\lambda\in [0,\infty)$. The insulating case "$\lambda=\infty$" is explained at the end of section 5.
		
		\section{Non-linear Koiter energy, geometry near the interface and some key lemmas}\label{sec:Geomtry}
		In the following we will discuss the structure of the Koiter energy, the geometry of the time dependent moving interface, embedding and trace results for domains which are not uniformly Lipschitz. We will further define the notion of convergences for our set up and recall an Aubin-Lions type result which is a key tool in proving the strong compactness of the fluid and structural velocities.
		\subsection{Non-linear Koiter energy}\label{GeometryKoi}
		The following discussion on the non-linear Koiter shell energy and its properties is a summary of \cite[Section 4]{MuhaSch} and \cite{Breit} (some of them are inspired from the reference literature \cite{CiarletIII}).\\
		The non-linear Koiter model is given in terms of the difference of the first and the second fundamental forms of $\Sigma_{\eta}$ and $\Gamma.$ We recall that $\nu_{\eta}$ denotes the normal-direction to the deformed middle surface $\varphi_{\eta}(\Gamma)$ at the point $\varphi_{\eta}(x)$ and is given by
		$$\nu_{\eta}(x)=\partial_{1}\varphi_{\eta}(x)\times \partial_{2}\varphi_{\eta}(x)={\bf{a}}_{1}(\eta)\times{\bf{a}_{2}}(\eta).$$
		In view of \eqref{varphieta}, these tangential derivatives $\bf{a}_{i}(\eta)$ can be computed as follows
		\begin{equation}\label{compdertan}
			\begin{array}{ll}
				\displaystyle {\bf{a}}_{i}(\eta)=\partial_{i}\varphi_{\eta}={{a}}_{i}+\partial_{i}\eta\nu+\eta\partial_{i}\nu,\quad\mbox{in}\quad i\in\{1,2\},
			\end{array}
		\end{equation}
		where ${{a}}_{i}=\partial_{i}\varphi(x).$\\
		Hence the components of the first fundamental form of the deformed configuration are given by
		\begin{equation}\label{fff}
			\begin{array}{ll}
				\displaystyle a_{ij}(\eta)={\ba}_{i}(\eta)\cdot{\ba}_{j}(\eta)
				=a_{ij}+\partial_{i}\eta\partial_{j}\eta+\eta({a}_{i}\cdot\partial_{j}\nu+{a}_{j}\cdot\partial_{i}\nu)+\eta^{2}\partial_{i}\nu\cdot\partial_{j}\nu,
			\end{array}
		\end{equation}
		where $a_{ij}=\partial_{i}\varphi(x)\cdot\partial_{j}\varphi(x).$\\
		Now in order to introduce the elastic energy $K=K(\eta)$ associated with the non-linear Koiter model we will use the description presented in \cite{MuhaSch} (which is inspired from \cite{CiarletIII}).\\
		In order to introduce the Koiter shell energy we first define two quantities $\mathbb{G}(\eta)$ and $\mathbb{R}(\eta).$ The change of metric tensor $\mathbb{G}(\eta)=(G_{ij}(\eta))_{i,j}$ is defined as follows
		\begin{equation}\label{Geta}
			\begin{array}{ll}
				\displaystyle G_{ij}(\eta)&\displaystyle=\partial_{i}\varphi_{\eta}\cdot\partial_{j}\varphi_{\eta}-\partial_{i}\varphi\cdot\partial_{j}\varphi=a_{ij}(\eta)-a_{ij}\\
				&\displaystyle =\partial_{i}\eta\partial_{j}\eta+\eta({a}_{i}\cdot\partial_{j}\nu+{a}_{j}\cdot\partial_{i}\nu)+\eta^{2}\partial_{i}\nu\cdot\partial_{j}\nu.
			\end{array}
		\end{equation}
		Further we define the tensor $\mathbb{R}(\eta)=(R_{ij}(\eta))_{i,j}$ which is a variant of the second fundamental form to measure the change of curvature
		\begin{equation}\label{Rsharp}
			\begin{array}{ll}
				\displaystyle R_{ij}(\eta)=\frac{\partial_{ij}\varphi_{\eta}\cdot \nu_{\eta}}{|\partial_{1}\varphi\times\partial_{2}\varphi|}-\partial_{ij}\varphi\cdot\nu=\frac{1}{{a}_{1}\times {a}_{2}}\partial_{i}{\ba}_{j}(\eta)\cdot\nu_{\eta}-\partial_{i}a_{j}\cdot\nu,\quad i,j=1,2.
			\end{array}
		\end{equation}
		Next, the elasticity tensor is defined as
		\begin{equation}\label{elastictytensor}
			\begin{array}{ll}
				&\displaystyle \mathcal{A}\mathbb{E}=\frac{4\lambda_{s}\mu_{s}}{\lambda_{s}+2\mu_{s}}(\mathbb{A}:\mathbb{E})\mathbb{A}+4\mu_{s}\mathbb{A}\mathbb{E}\mathbb{A},\qquad \mathbb{E}\in \mbox{Sym}(\mathbb{R}^{2\times 2}),
			\end{array}
		\end{equation}
		where $\mathbb{A}$ is the contra-variant metric tensor associated with $\partial\Omega^{2}$ and the constants $\lambda_{s},\mu_{s}>0$ are the Lam\'{e} coefficients. The Koiter energy of the shell is given by:
		\begin{equation}\label{Koiterenergy}
			\displaystyle K(\eta)=\frac{h_{th}}{4}\int_{\Gamma}\mathcal{A}\mathbb{G}(\eta(\cdot,t)):\mathbb{G}(\eta(\cdot,t))+\frac{h_{th}^{3}}{48}\int_{\Gamma}\mathcal{A}\mathbb{R}(\eta(\cdot,t))\otimes \mathbb{R}(\eta(\cdot,t)),
		\end{equation} 
		where $h_{th}>0$ is the thickness of the shell.\\ 
		Now in view of the Koiter energy \eqref{Koiterenergy} we write (following \cite{MuhaSch}) the elasticity operator $K'(\eta)$ as follows
		\begin{equation}\label{elasticityoperator}
			\begin{array}{ll}
				&\displaystyle \langle K'(\eta),b\rangle=\langle DK(\eta),b\rangle = a_{G}(t,\eta,b)+a_{R}(t,\eta,b),\qquad\forall b\in W^{2,p}(\Gamma)\,\,\mbox{where}\,\, p>2.
			\end{array}
		\end{equation}
		In the previous expression $a_{G}(t,\eta,b)$ and $a_{R}(t,\eta,b)$ are defined respectively as
		\begin{equation}\label{amab}
			\begin{array}{ll}
				&\displaystyle a_{G}(t,\eta,b)=\frac{h}{2}\int_{\Gamma}\mathcal{A}\mathbb{G}(\eta(\cdot,t)):\mathbb{G}'(\eta(\cdot,t))b,\\[2.mm]
				&\displaystyle a_{R}(t,\eta,b)=\frac{h^{3}}{24}\int_{\Gamma}\mathcal{A}\mathbb{R}(\eta,\cdot,t):\mathbb{R}^{'}(\eta(\cdot,t))b,
			\end{array}
		\end{equation}
		where $\mathbb{G}'$ and $\mathbb{R}'$ represent respectively the Fr\'{e}chet derivative of $\mathbb{G}$ and $\mathbb{R}.$ 
		It is important to know the structure of $a_{G}(t,\eta,b)$ and $a_{R}(t,\eta,b),$ which will play a key role during the passage of limits in suitably constructed approximate equations. Since $G_{ij}(\eta)$ is given by \eqref{Geta}, $G'_{ij}(\eta)b$ can simply be calculated as
		\begin{equation}\label{Gij'}
			\begin{array}{l}
				G_{ij}'(\eta)b=\partial_{i}b\partial_{j}\eta+\partial_{i}\eta\partial_{j}b+b(a_{i}\cdot\partial_{j}\nu+a_{j}\cdot\partial_{i}\nu)+2\eta b\partial_{i}\nu\cdot\partial_{j}\nu.
			\end{array}
		\end{equation}
		Hence one checks that $a_{G}(t,\eta,b)$ is a polynomial in $\eta$ and $\nabla\eta$ of order three and further the coefficients are in $L^{\infty}(\Gamma).$\\
		As in \cite[Section 4.1.]{MuhaSch}, $R_{ij}(\eta)$ (introduced in \eqref{Rsharp}) can be written in the following form which is easier to handle
		\begin{equation}\label{Rijcomp}
			\begin{array}{ll}
				R_{ij}(\eta)=\overline{\gamma}(\eta)\partial^{2}_{ij}\eta+P_{0}(\eta,\nabla\eta),
			\end{array}
		\end{equation}
		where $P_{0}$ is a polynomial of order three in $\eta$ and $\nabla\eta$ such that all terms are at most quadratic in $\nabla\eta$ and the coefficients of $P_{0}$ depend on $\varphi$ and the geometric quantity $\overline{\gamma}(\eta)$ (depending on $\partial\Omega$ and $\eta$) is defined as follows
		\begin{equation}\label{ovgamma}
			\begin{array}{l}
				\displaystyle\overline{\gamma}(\eta)=\frac{1}{|a_{1}\times a_{2}|}\bigg(|a_{1}\times a_{2}|+\eta(\nu\cdot(a_{1}\times\partial_{2}\nu+\partial_{1}\nu\times a_{2}))+\eta^{2}\nu\cdot(\partial_{1}\nu\times\partial_{2}\nu)\bigg).
			\end{array}
		\end{equation}
		Hence $R_{ij}(\eta)$ can be written as follows
		\begin{equation}\label{rewriteRij}
			\begin{array}{l}
				\displaystyle R_{ij}'(\eta)b=\overline{\gamma}(\eta)\partial^{2}_{ij}b+(\overline{\gamma}'(\eta)b)\partial^{2}_{ij}\eta+P'_{0}(\eta,\nabla\eta)b,
			\end{array}
		\end{equation}
		$i.e.$ we have
		\begin{equation}\label{exaR}
			\begin{array}{ll}
				\displaystyle a_{R}(t,\eta,b)=&\displaystyle\frac{h^{3}}{24}\int_{\Gamma}\bigg[\mA(\overline{\gamma}(\eta)\nabla^{2}\eta):(\overline{\gamma}(\eta)\nabla^{2}b)+\mA(\overline{\gamma}(\eta)\nabla^{2}\eta):(\overline{\gamma}'(\eta)b\nabla^{2}\eta)\\
				&\displaystyle +\mA(\overline{\gamma}(\eta)\nabla^{2}\eta):P'_{0}(\eta,\nabla\eta)b+\mA(P_{0}(\eta,\nabla\eta)):(\overline{\gamma}(\eta)\nabla^{2}b)\\
				&\displaystyle +\mA(P_{0}(\eta,\nabla\eta)):(\overline{\gamma}'(\eta)b\nabla^{2}\eta)+\mA(P_{0}(\eta,\nabla\eta)):(P'_{0}(\eta,\nabla\eta))b\bigg].
			\end{array}
		\end{equation}
		One notices that $\overline{\gamma}'(\eta)$ is a linear in $\eta$ and $\overline\gamma({\eta})$ is quadratic in $\eta.$\\[4.mm] 	
		The following section is a summary about the description of a moving domain and a few lemmas on the Sobolev embedding, extension operators and Korn's inequality concerning domains with H\"{o}lder continuous boundary.
		\subsection{Geometry of the moving interface, trace and embeddings}\label{Sec:GEE}
		This section contains a collection of facts related to domains with a moving boundary.\\
		Recall that in the introductory part of the present article we have introduced the tubular neighborhood of $\partial\Omega^{2}_{o}$ as
		\begin{equation*}
			N^b_a=\{\varphi(x)+\nu(x)z; \,\,x\in\Gamma, z\in(a_{\partial\Omega},b_{\partial\Omega})\}.
		\end{equation*}
		Further let us define the projection $\pi:N^b_a\to\partial\Omega^{2}$ as a mapping that assigns to each $x$ a unique $\pi(x)\in \partial\Omega^{2}_{o}$ such that there holds $\pi(x)=\varphi(y(x))$ where
		\begin{equation}
			\displaystyle 	y(x)=\arg\min_{y\in\Gamma}|x-\varphi(y)|.
		\end{equation}
		The signed distance function $d:N^b_a\to(a_{\partial\Omega},b_{\partial\Omega})$ is defined as
		\begin{equation*}
			d:x\mapsto (x-\pi(x))\cdot\nu(\varphi^{-1}(\pi(x))).
		\end{equation*}
		Since it is assumed that $\varphi\in C^4(\Gamma)$, it is well known that $\pi$ is well defined and possesses the $C^3$--regularity and $d$ is $C^4$ in a neighborhood of $\partial\Omega^{2}_{o}$ containing $N^b_a$, see \cite[Theorem 1 and Lemma 2]{Foote84}.
		Let $\eta:[0,T]\times\Gamma\to\eR$ be a given displacement function with $a_{\partial\Omega}<m\leq \eta\leq M<b_{\partial\Omega}$. We fix such a pair $\{m,M\}$ from the beginning.\\
		Let  $f_{m,M}\in C^\infty_{c}(\eR)$ be a cut-off function defined as follows
		\begin{equation}\label{propfmM}
			\begin{array}{l}
				0\leqslant f_{m,M}\leqslant 1,\,\,	f_{m,M}(d)=1 \,\,\mbox{for}\,\,d\in(m',M'),\,\,\mbox{and}\,\,\mbox{supp}\,f_{m,M}\Subset (m'',M'').
			\end{array}
		\end{equation}
		where $a_{\partial_\Omega}<m''<m'<m< \eta< M<M'<M''<b_{\partial\Omega}.$\\
		Further observe that
		\begin{equation}\label{suppf'}
			\begin{array}{ll}
				\displaystyle \mbox{supp}\,f'_{m,M}(d(\cdot))\subset \mathcal{A}_{m,M}=\{\varphi(x)+z\nu(x)\suchthat (z,x)\in [m'',M'']\times \Gamma\}.
			\end{array}
		\end{equation}
		For $\eta(x)\in (m,M),$ this allows us to introduce the mapping
		$$\widetilde{\varphi}_{\eta}(\cdot,t):\overline{\Omega^{2}_{o}}\times[0,T]\rightarrow \overline{\Omega^{2}(t)}$$ by
	\begin{equation}\label{deftilphi}
		\widetilde\varphi_\eta(x,t)=(1-f_{m,M}(d(x))x+f_{m,M}(d(x))(\pi(x)+(d(x)+\eta(t,\varphi^{-1}(\pi(x))))\nu(\varphi^{-1}(\pi(x))).
	\end{equation}
	Hence for the inverse $(\widetilde\varphi_\eta)^{-1}: {\Omega^{2}(t)}\times[0,T]\rightarrow\Omega^{2}_{o}$ one has
	\begin{align*}
		(\widetilde\varphi_\eta)^{-1}(z,t)=(1-f_{m,M}(d(z))z+f_{m,M}(d(z))(\pi(z)+(d(z)-\eta(t,\varphi^{-1}(\pi(z))))\nu(\varphi^{-1}(\pi(z))).
	\end{align*}
	Obviously, the mapping $\widetilde\varphi_\eta$ and its inverse inherit the regularity of $\eta$. \\
	Now one notes that $\widetilde{\varphi}$ admits of an invertible extension $i.e.$ one can define $\widetilde{\varphi}_{\eta}: \mathbb{R}^{3}\times [0,T]\rightarrow \mathbb{R}^{3}$ and $(\widetilde{\varphi}_{\eta})^{-1}: \mathbb{R}^{3}\times[0,T]\rightarrow \mathbb{R}^{3}$ as
	\begin{equation}\label{extentphi}
		\begin{array}{l}
			\widetilde\varphi_\eta(x,t)=(1-f_{m,M}(\mathfrak{d}(x))x+f_{m,M}(\mathfrak{d}(x))(\pi(x)+(\mathfrak{d}(x)+\eta(t,\varphi^{-1}(\pi(x))))\nu(\varphi^{-1}(\pi(x))),\\[2.mm]
			(\widetilde\varphi_\eta)^{-1}(z,t)=(1-f_{m,M}(\mathfrak{d}(z))z+f_{m,M}(\mathfrak{d}(z))(\pi(z)+(\mathfrak{d}(z)-\eta(t,\varphi^{-1}(\pi(z))))\nu(\varphi^{-1}(\pi(z))),
		\end{array}
	\end{equation}
	where $\mathfrak{d}$ is defined as
	\begin{equation*}
		\mathfrak d(x,\partial\Omega)=\begin{cases}-\dist(x,\partial\Omega)&\text{ if }x\in\overline\Omega\\
			\dist(x,\partial\Omega)&\text{ if }x\in\eR^3\setminus\Omega.\end{cases}
	\end{equation*} 
	Note that $d$ and $\mathfrak d$ coincide in $N^b_a$.\\ 
		Let us summarize the assumptions on the geometry for the assertions in the rest of this section.\\
		Since we have already characterized ${\Omega^{2}(t)}$ as an image of $\Omega^{2}_{o}$ under the invertible map $\widetilde{\varphi}_{\eta},$ (both of the maps $\widetilde{\varphi}_{\eta}$ and its inverse inherit regularity from $\eta$) we can now write the outer moving domain ${\Omega^{1}(t)}$ as follows
		\begin{equation}\label{chOmt1}
			{\Omega^{1}(t)}=\Omega\setminus \overline{{\Omega^{2}(t)}}=\Omega\setminus\overline{\widetilde{\varphi}_{\eta}(\Omega^{2}_{o})}.
		\end{equation}
		\hypertarget{Assumption}{\textbf{Assumptions (A)}}: Let $\Omega^{2}_{o}\Subset \Omega\subset\eR^3$ be two domain of class $C^4$. We define $\Omega^{1}_{o}=\Omega\setminus\overline{\Omega^{2}_{o}}$, see Figure 1. Let the boundary $\partial\Omega^{2}_{o}$ of $\Omega^{2}_{o}$ be parametrized as $\partial\Omega^{2}_{o}=\varphi(\Gamma)$, where $\varphi$ is a $C^4$ injective--mapping and $\Gamma\subset\eR^2$ is the torus. We take $N_a^b$ as the tubular neighborhood of $\partial\Omega^{2}_{o}$ defined in \eqref{safedis} 
		and assume that $$\partial\Omega\cap N^{b}_{a}=\emptyset$$.
		\\[3.mm]
		Under the validity of Assumption (A) for $\eta\in C([0,T]\times\Gamma;[a_{\partial \Omega} ,b_{\partial \Omega}])$ we define the underlying function spaces on variable domains in the following way for $p,r\in[1,\infty]$
		\begin{align*}
			L^p(0,T;L^r({\Omega^{i}(t)}))=&\{v(x,t): v(t)\in L^r({\Omega^{i}(t)})\text{ for a.e. }t\in (0,T),
			\,\,\|v(t)\|_{L^r({\Omega^{i}(t)})}\in L^p((0,T))\},\\
			L^p(0,T;L^r({\Omega^{i}(t)}))=&\{v(x,t):v\in L^p(0,T;L^r({\Omega^{i}(t)})),\,\,\nabla v\in L^p(0,T;L^r(\Omega^{i}(t)))\}.
		\end{align*}
		For the purposes of further discussion of this section we define 
		\begin{equation}\label{defX}
			X=L^\infty(0,T;W^{2,2}(\Gamma;[a_{\partial \Omega} ,b_{\partial \Omega}]))\cap W^{1,\infty}(0,T;L^2(\Gamma;[a_{\partial \Omega} ,b_{\partial \Omega}])).
		\end{equation}
		Since 
		\begin{equation*}
			L^\infty(0,T;W^{2,2}(\Gamma))\cap W^{1,\infty}(0,T;L^2(\Gamma))\hookrightarrow C^{0,1-\theta}([0,T];C^{0,2\theta-1}(\Gamma;[a_{\partial \Omega} ,b_{\partial \Omega}]))
		\end{equation*}
		for $\theta\in(\frac{1}{2},1)$, cf. \cite[(2.29)]{LeRu14}, we get $X\hookrightarrow C([0,T]\times\Gamma)$ and the above defined function spaces are meaningful for $\eta\in X$.\\
		In the following we will introduce the notion of trace for a moving interface which is not Lipschitz uniformly in time. We will further comment on Sobolev embeddings for such kind of moving domains. In what follows in the rest of this article we will be using these results quite often without reference.
		\begin{lemma}\label{Lem:TrOp}
			\cite[Lemma 2.3.]{Breit}	Let \hyperlink{Assumption}{Assumptions (A)} hold true with $\eta\in X$ and $p\in[1,\infty]$, $q\in(1,3)$. Then the linear mapping $\Tr_{\Sigma_\eta}:v\mapsto v \circ \widetilde\varphi_\eta|_{\partial\Omega^{2}_{o}}$ (recall that the extension of the map $\widetilde{\varphi}_{\eta}$ as defined in \eqref{extentphi}) is well defined and continuous from $L^p(0,T;W^{1,q}(\Omega^{i}(t)))$ to $L^p(L^r(\partial\Omega^{2}_{o}))$ for all $r\in (1,\frac{2q}{3-q})$, respectively from $L^p(0,T;W^{1,q}(\Omega^{i}(t)))$ to $L^p(0,T;W^{1-\frac{1}{r},r}(\Sigma_\eta(t))$ for any $1\leq r<q$.
		\end{lemma}
		Let us now state a Sobolev embedding result for moving domains.
		\begin{lemma}
			\cite[Corollary 2.10.]{LeRu14}	Let $1<p<3$ and let  \hyperlink{Assumption}{Assumptions (A)} hold true with $\eta\in X.$ Then one has the following compact embedding 
			$$W^{1,p}({\Omega^{i}(t)})\hookrightarrow\hookrightarrow L^{s}({\Omega^{i}(t)})$$
			for $1\leqslant s<p^{*}=\frac{3p}{(3-p)}.$ The embedding constant depends only on $\Omega,$ $\Omega^{2}_{o},$ $p,$ $s$ and $\|\eta\|_{X}.$ 
		\end{lemma} 
		
			\subsection{Solenoidal extension}
			In this section we recall two results from \cite{MuhaSch} and \cite{Breit} on the solenoidal extension of a vector field defined at the moving fluid interface.\\
			Before stating the results let us recall from \cite{MuhaSch} the construction of a suitable corrector map which  will be used to construct a solenoidal extension operator (this is essential since not all functions at the interface admits of a solenoidal extension).\\
			In that direction let us first recall the notations introduced in Section \ref{Sec:GEE}. We define for a function $\xi:\Gamma\rightarrow \mathbb{R}$
			\begin{equation}\label{txi}
				\begin{array}{l}
					\widetilde{\xi}: \partial\Omega^{2}_{o} \rightarrow\mathbb{R}\,\,\mbox{with}\,\, \widetilde{\xi}(\pi(x))=\widetilde{\xi}(x)=\xi(\varphi^{-1}(\pi(x))),
				\end{array}
			\end{equation}
			where the notion of the projection $\pi$ is introduced in Section \ref{Sec:GEE}.\\
			In order to construct a solenoidal extension, we apply the Bogovskii theorem (in the spirit of \cite{MuhaSch}) in the set $\mathcal{A}_{m,M}$ (one recalls the construction of $\mathcal{A}_{m,M}$ from \eqref{suppf'}). Note that $\mathcal{A}_{m,M}$ is a $C^{2}$ and recall the support of $f'_{m,M}(d(\cdot))$ from \eqref{suppf'}.\\
			Next one introduces the following weighted mean-value over $\mathcal{A}_{m,M}.$ Let  $\varLambda\in L^{\infty}(\mathcal{A}_{m,M}),$ $\varLambda\geqslant 0$ and $\int_{\mathcal{A}_{m,M}}\varLambda(x) >0$ be a given weight function. Then
			\begin{equation}\label{defavgLam}
				\begin{array}{l}
					\displaystyle \langle \psi\rangle_{\varLambda}=\frac{\int_{\mathcal{A}_{m,M}}\psi(x)\varLambda(x)}{\int_{\mathcal{A}_{m,M}}\varLambda(x)}\,\,\mbox{for}\,\,\psi\in L^{1}(\mathcal{A}_{m,M}).
				\end{array}
			\end{equation}
			Further denote 
			\begin{equation}\label{varLameta}
				\begin{array}{l}
					\displaystyle \varLambda_{\eta}(x,t)=e^{(d(x)-\eta(\varphi^{-1}(\pi(x)),t))\div(\nu(\pi(x)))}f'_{m,M}(d(x))\geqslant 0,
				\end{array}
			\end{equation}
			which has compact support in $\mathcal{A}_{m,M}$ and satisfies (uniformly in t):
			\begin{equation}\label{bndvarlameta}
				\begin{array}{l}
					c_{1}\leqslant \|\varLambda_{\eta}\|_{L^{1}(\mathcal{A}_{m,M})}\leqslant c_{2}\|\varLambda_{\eta}\|_{L^{\infty}(\mathcal{A}_{m,M})}\leqslant c_{3}
				\end{array}
			\end{equation}
			for some positive constants $c_{1}\leqslant c_{2}\leqslant c_{3}$ depending on $m,M$ and lower bounds of $\eta.$\\
			Next let us define the corrector map. Assume $a_{\partial\Omega}<m\leq \eta\leq M<b_{\partial\Omega}$ then the corrector map
			\begin{equation}\label{corrector}
				\begin{array}{l}
					\displaystyle	\mathcal{K}_{2,\eta}:L^{1}(\Gamma)\rightarrow \mathbb{R},\,\,\xi\mapsto\mathcal{K}_{2,\eta}(\xi)=\langle \widetilde{\zeta}\rangle_{\varLambda_{\eta}}=\frac{\int_{\mathcal{A}_{m,M}}\widetilde{\xi}(\pi(x))\varLambda_{\eta}(x,t)}{\int_{\mathcal{A}_{m,M}}\varLambda_{\eta}(x,t)}
				\end{array}
			\end{equation}
			Note that $\mathcal{K}_{2,\eta}(\xi)$ is only a function of time and solves 
			$$\displaystyle \|\mathcal{K}_{2,\eta}(\xi)\|_{L^{q}(0,T)}\lesssim \|\xi\|_{L^{q}(0,T;L^{1}(\Gamma))}$$
			and $$\displaystyle \|\partial_{t}\mathcal{K}_{2,\eta}(\xi)\|_{L^{q}(0,T)}\lesssim \left(\|\partial_{t}\xi\|_{L^{q}(0,T;L^{1}(\Gamma))}+\|\xi\partial_{t}\eta\|_{L^{q}(0,T;L^{1}(\Gamma))}\right)$$
			for $q\in[1,\infty].$
			For the proof of the estimates above we refer to \cite[Corollary 3.2]{MuhaSch}.\\
			We will next see that for a function $\xi\in L^{1}(0,T;W^{1,1,}(\Gamma)),$ one can construct a divergence free lifting of $(\xi-\mathcal{K}_{2,\eta}(\xi))$ in a neighborhood $\Omega^{2}_{m,M}$ of $\Omega^{2}_{o}.$ In the same spirit one constructs $\mathcal{K}_{1,\eta}$ (indeed the construction here requires negative of the signed distance function and unit normals used in \eqref{varLameta}-\eqref{corrector}) and a solenoidal lifting of $(\xi-\mathcal{K}_{1,\eta}(\xi))$ in a neighborhood  $\Omega^{1}_{m,M}$ of $\Omega^{1}_{o}.$\\
			To be precise, one has the following.
			\begin{prop}\label{smestdivfrex}
				Let \hyperlink{Assumption}{Assumptions (A)} hold true for a given $\eta \in X$ with $a_{\partial\Omega}<m\leqslant \eta\leqslant M< b_{\partial\Omega},$ there exists a tubular neighborhood $S_{m,M}$ of $\partial\Omega^{2}_{o}$ (in view of the construction done in Section \ref{Sec:GEE}, one can define $S_{m,M}$ as $$S_{m,M}=\{\varphi(y)+d\nu(y):(d,y)\in [m,M]\times\Gamma\}$$) such that
				\begin{equation}\label{tbnbd}
					\{\varphi(x)+z\nu(x)\suchthat m\le z\le M\}\Subset S_{m,M}
				\end{equation}
				and there are linear operators
				$$\mathcal{K}_{i,\eta}:L^{1}(\Gamma)\rightarrow \mathbb{R},\,\,\mathcal{F}^{\div}_{i,\eta}:\{\xi\in L^{1}(0,T;W^{1,1}(\Gamma))\suchthat \mathcal{K}_{i,\eta}(\xi)=0\}\rightarrow L^{1}(0,T;W^{1,1}_{\div}(\Omega^{i}_{m,M})),$$ 
				such that the couple $(\mathcal{F}^{\div}_{i,\eta}(\xi-\mathcal{K}_{i,\eta}(\xi)),\xi-\mathcal{K}_{i,\eta}(\xi))$ solves
				\begin{equation}\nonumber
					\begin{array}{ll}
						&\displaystyle\mathcal{F}^{\div}_{i,\eta}(\xi-\mathcal{K}_{i,\eta}(\xi))\in L^{\infty}(0,T;L^{2}(\Omega^{i}(t)))\cap L^{2}(0,T;W^{1,2}_{\div}(\Omega^{i}(t))),\\
						& \displaystyle\xi-\mathcal{K}_{i,\eta}(\xi)\in L^{\infty}(0,T;W^{2,2}(\Gamma))\cap W^{1,\infty}(0,T;L^{2}(\Gamma)),\\
						&\displaystyle tr_{\Sigma_\eta}(\mathcal{F}^{\div}_{i,\eta}(\xi-\mathcal{K}_{i,\eta}(\xi)))=(\xi-\mathcal{K}_{i,\eta}(\xi))\nu(\cdot),\\
						&\displaystyle \mathcal{F}^{\div}_{i,\eta}(\xi-\mathcal{K}_{i,\eta}(\xi))(t,x)=0\,\,\mbox{for}\,\,(t,x)\in (0,T)\times (\Omega^{i}_{o}\setminus S_{m,M}),
					\end{array}
				\end{equation}
				where $\Omega^{i}_{m,M}=\Omega^{i}_{o}\cup S_{m,M}.$
				Provided that $\eta,\xi\in L^{\infty}(0,T;W^{2,2}(\Gamma))\cap W^{1,\infty}(0,T;L^{2}(\Gamma)),$ one has the following estimates
				\begin{equation}\label{Fdivetaest1}
					\begin{split}
						&\|\mathcal{F}^{\div}_{i,\eta}(\xi-\mathcal{K}_{i,\eta}(\xi))\|_{L^{q}(0,T;W^{1,p}(\Omega^{i}_{m,M}))}\lesssim \|\xi\|_{L^{q}(0,T;W^{1,p}(\Gamma))}+\|\xi\nabla\eta\|_{L^{q}(0,T;L^{p}(\Gamma))},\\
						& \|\partial_{t}\mathcal{F}^{\div}_{i,\eta}(\xi-\mathcal{K}_{i,\eta}(\xi))\|_{L^{q}(0,T;L^{p}(\Omega^{i}_{m,M}))}\lesssim\|\partial_{t}\xi\|_{L^{q}(0,T;L^{p}(\Gamma))}+\|\xi\partial_{t}\eta\|_{L^{q}(0,T;L^{p}(\Gamma))},
					\end{split}
				\end{equation}
				for any $p\in (1,\infty),$ $q\in (1,\infty].$\\
				Furthermore in the same spirit of the proof of \cite[Proposition 3.3]{MuhaSch} (using chain rule of derivatives) and with the assumption $\eta(s),\xi(s)\in W^{k_{0},2}(\Gamma)$ for $a.e.$ $s\in (0,T),$ ($k_{0}\in\mathbb{N}$ and $k_{0}>3$) one in particular proves that for $a.e.$ $s\in(0,T)$ the following holds
				\begin{equation}\label{higherderest1*}
					\begin{split}
						&\displaystyle\|\nabla^{k_{0}}\mathcal{F}^{\div}_{\eta(s)}(\xi-\mathcal{K}_{i,\eta}(\xi))\|_{L^{2}(\Omega^{i}_{m,M})}\\
						&\displaystyle \lesssim \|\nabla^{k_{0}}\xi(s)\|_{L^{2}(\Gamma)}+\sum^{k_{0}}_{i=1}\||\nabla^{k_{0}-i}\xi(s)||\nabla\eta|^{i}(s)\|_{L^{2}(\Gamma)}\\
						&\displaystyle\qquad+\sum^{k_{0}-1}_{i,j=1}\||\nabla^{k_{0}-i}\xi(s)||\nabla\eta(s)|^{i-j}|\nabla^{j}\eta(s)|\|_{L^{2}(\Gamma)} +\||\xi(s)||\nabla^{k_{0}}\eta(s)|\|_{L^{2}(\Gamma)} \\
						&\displaystyle \lesssim \|\nabla^{k_{0}}\xi(s)\|_{L^{2}(\Gamma)}+\sum^{k_{0}}_{i=1}\|\nabla^{k_{0}-1}\xi(s)\|_{L^{4}(\Gamma)}\|\nabla\eta(s)\|^{i}_{L^{\infty}(\Gamma)}\\
						&\displaystyle \qquad + \sum^{k_{0}-1}_{i,j=1,i\geqslant j}\|\nabla^{k_{0}-1}\xi(s)\|_{L^{4}(\Gamma)}\|\nabla\eta(s)\|^{i-j}_{L^{\infty}(\Gamma)}\|\nabla^{k_{0}-1}\eta(s)\|_{L^{4}(\Gamma)}\\
						&\displaystyle\qquad+\|\xi(s)\|_{L^{\infty}(\Gamma)}\|\nabla^{k_{0}}\eta(s)\|_{L^{2}(\Gamma)}.
					\end{split}
				\end{equation}
		\end{prop}
		The proof of Proposition \ref{smestdivfrex} until \eqref{higherderest1*} can be found in \cite[Proposition 3.3]{MuhaSch}. Further performing chain rule of derivatives and suitably estimating the product terms one concludes the proof of  \eqref{higherderest1*}. Indeed one has to use the structure of $\mathcal{F}^{div}_{i,\eta}$ for the purpose of proving estimates \eqref{Fdivetaest1} and \eqref{higherderest1*}. In particular for $i=2,$ $\mathcal{F}^{div}_{2,\eta}$ is given as (we refer to \cite[Section 3]{MuhaSch})
		\begin{equation}\label{consf2eta}
			\begin{array}{ll}
				\displaystyle \mathcal{F}^{div}_{2,\eta}(\xi)(x)={\mathcal{F}_{2,\eta}}(\xi)(x)-\mbox{Bog}_{m,M}(\div(\mathcal{F}_{2,\eta}(\xi)))(x)
			\end{array}
		\end{equation} 
		where 
		\begin{equation}\label{deff2eta}
			\begin{array}{ll}
				\mathcal{F}_{2,\eta}(\xi)(x)=e^{(d(x)-\eta(\varphi^{-1}(\pi(x))))\div(\nu(\pi(x)))}\widetilde{\xi}(\pi(x))\nu(\pi(x))
			\end{array}
		\end{equation}
		and $\mbox{Bog}_{m,M}$ is the classical Bogovskii operator constructed for the $C^{2}$ domain $\mathcal{A}_{m,M}.$\\
		The next result is a corollary to Proposition \ref{smestdivfrex} and provides a key tool towards proving the compactness of velocities.
		\begin{corollary}\label{corextension}
			Let $a,r\in[2,\infty],$ $p,q\in(1,\infty)$ and $s\in[0,1].$ Assume that $\eta\in L^{r}(0,T;W^{2,a}(\Gamma))\cap W^{1,r}(0,T;L^{a}(\Gamma))$ such that  $a_{\partial\Omega}<m\leqslant \eta\leqslant M< b_{\partial\Omega}.$ Further let $b\in W^{s,p}(\Gamma)$ and take $(b)_{\delta}$ as a smooth approximation of $b$ in $\Gamma.$ Then $E_{i,\eta,\delta}(b)=\mathcal{F}^{\div}_{i,\eta}((b)_{\delta}-\mathcal{K}_{i,\eta}((b)_{\delta}))$ satisfies all the regularity of Proposition \ref{smestdivfrex}. In particular
			\begin{equation}\label{estEetadel}
				\begin{array}{ll}
					&\displaystyle \|E_{i,\eta(t),\delta}(b)-\mathcal{F}^{\div}_{i,\eta}(b-\mathcal{K}_{i,\eta}(b))\|_{L^{p}(\Omega^{i}_{m,M})}\leqslant c\|(b)_{\delta}-b\|_{L^{p}(\Gamma)},\\
					&\displaystyle \|\partial_{t} E_{i,\eta(t),\delta}(b)\|_{L^{r}(0,T;L^{a}(\Omega^{i}_{m,M}))}\leqslant c \|(b)_{\delta}\partial_{t}\eta(t)\|_{L^{r}(0,T;L^{a}(\Gamma))}
				\end{array}
			\end{equation} 
			uniformly in $t\in (0,T).$
		\end{corollary}
		We devote the next section to introduce the notion of convergence in moving domains and we would further state a compactness result of Aubin-Lions type which is a key tool to show compactness of fluid and structural velocities later in the present article.
		\subsection{Notion of convergence and an Aubin-Lions type theorem}
		Let us first introduce the notion of convergence in the framework of time dependent variable domains (in the spirit of \cite{Breit}) .\\[2.mm] 
		\textbf{Definition:} Let  \hyperlink{Assumption}{Assumptions (A)} hold true with the role of $\eta$ replaced by $\eta_{j},$ where $(\eta_{j})_{j}\subset C(\Gamma\times[0,T];(a_{\partial\Omega},b_{\partial\Omega})).$ Further assume that $\eta_{j}\rightarrow\eta$ uniformly in $\Gamma\times[0,T].$ Let $p,q\in[1,\infty]$ and $k\in\mathbb{N}_{0}.$ Corresponding to the displacement $\eta_{j},$ we will denote the moving domains by $\Omega^{i}_{j}(t).$ \\
		$(a)$ We say that a sequence $(g_{j})\subset L^{p}(0,T;L^{q}(\Omega^{i}_{j}(t)))$ (indeed $i\in\{1,2\}$ is fixed) converges to $g$ in $L^{p}(0,T;L^{q}(\Omega^{i}_{j}(t)))$ strongly with respect to $(\eta_{j}),$ in symbols $g_{j}\rightarrow g$ in $L^{p}(0,T;L^{q}(\Omega^{i}_{j}(t))),$ if 
		$$\chi_{\Omega^{i}_{j}(t)}g_{j}\rightarrow \chi_{\Omega^{i}(t)}g\,\,\mbox{in}\,\, L^{p}(0,T;L^{q}(\mathbb{R}^{3})).$$\\[3.mm]
		$(b)$ Let $p,q<\infty.$ We say that a sequence $(g_{j})\subset L^{p}(0,T;L^{q}(\Omega^{i}_{j}(t)))$ converges to $g$ in $L^{p}(0,T;L^{q}(\Omega^{i}_{j}(t)))$ weakly with respect to $(\eta_{j}),$ in symbols $g_{j}\rightharpoonup g$ in $L^{p}(0,T;L^{q}(\Omega^{i}_{j}(t))),$ if 
		$$\chi_{\Omega^{i}_{j}(t)}g_{j}\rightharpoonup \chi_{\Omega^{i}(t)}g\,\,\mbox{in}\,\, L^{p}(0,T;L^{q}(\mathbb{R}^{3})).$$\\[3.mm]
		$(c)$ let $p=\infty$ and $q<\infty.$ We say that a sequence $(g_{j})\subset L^{\infty}(0,T;L^{q}(\Omega^{i}_{j}(t)))$ converges to $g$ in $L^{\infty}(0,T;L^{q}(\Omega^{i}_{j}(t)))$ weak$^{*}$ with respect to $(\eta_{j}),$ in symbols $g_{j}\rightharpoonup^{*}g$ in $L^{\infty}(0,T;L^{q}(\Omega^{i}_{j}(t))),$ if 
		$$\chi_{\Omega^{i}_{j}(t)}g_{j}\rightharpoonup^{*} \chi_{\Omega^{i}(t)}g\,\,\mbox{in}\,\, L^{\infty}(0,T;L^{q}(\mathbb{R}^{3})).$$\\[3.mm]
		Further note that the space $L^{p}(0,T;L^{q}({\Omega^{i}(t)}))$ (with $1\leqslant p< \infty$ and $1<q<\infty$) is reflexive and we have the usual duality pairing
		$$L^{p}(0,T;L^{q}({\Omega^{i}(t)}))\cong L^{p'}(0,T;L^{q'}({\Omega^{i}(t)}))$$
		provided $\eta$ is smooth enough. Further the previous definition can be extended in a canonical way to Sobolev spaces. We say that a sequence $(g_{j})\subset L^{p}(0,T;W^{1,q}(\Omega^{i}_{j}(t)))$ converges to $g\in L^{p}(0,T;W^{1,q}({\Omega^{i}(t)}))$ strongly with respect to $(\eta_{j}),$ in symbols
		$$g_{j}\rightarrow  g\,\,\mbox{in}\,\, L^{p}(0,T;W^{1,q}(\Omega^{i}_{j}(t))),$$
		iff both $g_{j}$ and $\nabla g_{j}$ converges (to $g$ and $\nabla g$ respectively) in $L^{p}(0,T;W^{1,q}(\Omega^{i}_{j}(t)))$ strongly with respect to $(\eta_{j})$ (in the sense of item $(a)$ of the previous definition). We can further define weak and weak$^{*}$ convergence in Sobolev spaces with respect to $(\eta_{j})$ with an obvious meaning.\\[2.mm]  
		Now let us  recall a compactness result of Aubin-Lions type which is suitably constructed to obtain compactness by using minimal assumptions on the functions involved.  The following theorem is taken from \cite[Theorem 5.1.]{MuhaSch}.
		\begin{thm}\label{thmBbL}
			Let $X$ and $Z$ be two Banach spaces, such that $X^{*}\subset Z^{*}$ (they respectively represent the dual spaces of $X$ and $Z$). Assume that $f_{n}:(0,T)\rightarrow X$ and $g_{n}: (0,T)\rightarrow X^{*}$ such that $g_{n}\in L^{\infty}(0,T;Z^{*})$ uniformly. Moreover assume the following:\\
			$(1)$ The weak convergence: for some $s\in [1,\infty]$ we have that $f_{n}\rightharpoonup^* f$ in $L^{s}(X)$ and $g_{n}\rightharpoonup^*g$ in $L^{s'}(X^{*}).$\\
			$(2)$ The approximability condition is satisfied: For every $\delta\in(0,1]$ there exists a $f_{n,\delta}\in L^{s}(0,T;X)\cap L^{1}(0,T;Z),$ such that for every $\epsilon\in (0,1)$ there exists a $\delta_{\epsilon}\in (0,1)$ (depending only on $\epsilon$) such that
			$$\displaystyle \|f_{n}-f_{n,\delta}\|_{L^{s}(0,T;X)}\leqslant \epsilon\,\,\mbox{for all}\,\,\delta\in (0,\delta_{\epsilon}]$$
			and for every $\delta\in (0,1]$ there is a $C(\delta)$ such that
			$$\|f_{n,\delta}\|_{L^{1}(0,T;Z)}dt\leqslant C(\delta).$$
			Moreover, we assume that for every $\delta$ there is a function $f_{\delta}$ and a subsequence such that $f_{n,\delta}\rightharpoonup^* f_{\delta}$ in $L^{s}(0,T;X).$\\
			$(3)$ The equi-continuity of $g_{n}:$ We require that there exists an $\alpha\in (0,1]$ a functions $A_{n}$ with $A_{n}\in L^{1}(0,T)$ uniformly, such that for every $\delta>0$ that there exist a $C(\delta)>0$ and an $n_{\delta}\in\mathbb{N}$ such that for $\tau>0$ and $a.e.$ $t\in [0,T-\tau]$
			$$\sup_{n\geqslant n_{\delta}}\bigg|\frac{1}{\tau}\int^{\tau}_{0}\langle g_{n}(t)-g_{n}(t+s),f_{n,\delta}(t)\rangle_{X',X}ds \bigg|\leqslant C(\delta)\tau^{\alpha}(A_{n}(t)+1).$$
			$(4)$ The compactness assumption is satisfied: $X^{*}\hookrightarrow\hookrightarrow Z^{*}.$ More precisely, every uniformly bounded sequence in $X^{*}$ has a strongly converging subsequence in $Z^{*}.$\\
			Then there is a sub-sequence, such that
			$$\displaystyle \int^{T}_{0}\langle f_{n}, g_{n}\rangle _{X,X^{*}}dt\rightarrow \int^{T}_{0}\langle f,g \rangle _{X,X^{*}}dt.$$
		\end{thm}
		
		\section{A time delayed approximation layer}
		At this stage we solve a time-delayed problem which results after discretizing the material derivative of the fluid and its elastic counterpart at level $h>0$ (indeed $h$ is fixed at this stage). 
		\subsection{Definition of solution to a time delayed problem and existence result}
		Let us now introduce the set up of weak solutions at a discrete layer with the discretization parameter $h>0.$
		\begin{definition}\label{Def:taulayer}
			Given 
			a known previous solid velocity $\beta:  \Gamma\times[0,h] \rightarrow \mathbb{R}$ and a corresponding quantity $w_{i}:\Omega^{i}_{0}\times[0,h]\rightarrow\mathbb{R}^{3}$ for the fluids, we call a triplet $(u_{i},\theta_{i},\eta)$ to be a solution to a time delayed problem in $(0,h)$ if
			\begin{equation}\label{regularity2}
				\begin{array}{ll}
					&\displaystyle u_{i}\in  L^{2}(0,h;W^{k_{0},2}(\Omega^{i})),\,\,\mathrm{div}\,u_{i}=0\,\,\mbox{in}\,\,\Omega^{i},\\
					& \displaystyle\theta_{i}\in L^{\infty}(0,h;L^{1}(\Omega^{i}))\cap L^{p}(\Omega^{i}\times(0,T))\cap L^{s}(0,h;W^{1,s}(\Omega^{i})),\,\,\mbox{for all}\,\,(p,s)\in[1,\frac{5}{3})\times [1,\frac{5}{4}),\\
					&\displaystyle \eta \in L^{\infty}(0,h;W^{k_{0}+1,2}(\Gamma))\cap W^{1,2}(0,h;W^{k_{0}+1,2}(\Gamma)),\,\,\mbox{for}\,\, k_{0}>3,\\
				\end{array}
			\end{equation}
			and the following hold\\[2.mm]
			$(i)$ A decomposition of $\Omega$ of the form \eqref{Omegat12} with \eqref{varphieta}-\eqref{Sigmat} hold.\\[2.mm]
			$(ii)$ \textit{The following time delayed momentum equation holds:}
		\begin{equation}\label{discretemomen1}
			\begin{array}{ll}
				&\displaystyle \int^{t_{1}}_{0}\langle DK_{\kappa}(\eta),b\rangle+\frac{1}{h}\int^{t_{1}}_{0}\int_{\Gamma}\left(\partial_{t}{\eta}-\beta\right)b+\sum_{i}\frac{1}{h}\int^{t_{1}}_{0}\int_{\Omega^{i}(t)}\left(u_{i}-w_{i}\circ(\Phi_{i})^{-1}\right)\psi\\
				&\displaystyle+{\kappa}\sum_{i}\int^{t_{1}}_{0}\int_{\Omega^{i}(t)}\nabla^{k_{0}}u_{i}:\nabla^{k_{0}}\psi+\sum_{i}\int^{t_{1}}_{0}\int_{\Omega^{i}(t)}\mu_{i}(\theta_{i})D(u_{i}):D(\psi)\\
				&\displaystyle +\kappa\int^{t_{1}}_{0}\int_{\Gamma}\nabla^{k_{0}+1}\partial_{t}\eta:\nabla^{k_{0}+1}b=0.
			\end{array}
		\end{equation}
		for $a.e.$ $t_{1}\in[0,h],$ for all $b\in L^{\infty}(0,h;W^{k_{0}+1,2}(\Gamma))\cap W^{1,\infty}(0,h;L^{2}(\Gamma)),$ for all $\psi\in C^{\infty}(\overline{[0,h]\times\Omega})$  with $\mbox{div}\,\psi=0$ on $\Omega$ satisfying
		\begin{equation}
			\begin{array}{l}
				\displaystyle \Tr_{\Sigma_{\eta}}u_{i}=\partial_{t}\eta\nu\quad\mbox{and}\quad \Tr_{\Sigma_{\eta}}\psi=b\nu\quad\mbox{on}\quad \Gamma.
			\end{array}
		\end{equation}
		In \eqref{discretemomen1}, $\Phi_{i}:\overline{\Omega^{i}_{0}}\times[0,h]\rightarrow \overline{\Omega^{i}}$ denotes the flow maps which solve
		\begin{equation}\label{flowmap}
			\left\{ \begin{array}{ll}
				&\displaystyle\partial_{t}\Phi_{i}=u_{i}\circ\Phi_{i},\\
				&\displaystyle \Phi(\cdot,0)=Id|_{\Omega^{i}_{0}}
			\end{array}\right.
		\end{equation} 
		and
		\begin{equation}\label{strKkappa}
			\displaystyle K_{\kappa}(\eta)=K(\eta)+\frac{\kappa}{2}\int_{\Gamma}|\nabla^{k_{0}+1}\eta|^{2},\,\,\mbox{for}\,\,k_{0}>3.
		\end{equation}
		Indeed \eqref{strKkappa} implies that
		\begin{equation}\label{DefDk}
			\displaystyle \langle DK_{\kappa}(\eta),b\rangle=\langle DK(\eta),b\rangle +\kappa\langle \nabla^{k_{0}+1}\eta,\nabla^{k_{0}+1}b \rangle.
		\end{equation}
		$(iii)$ The following heat evolution holds
		\begin{equation}\label{intenballh}
			\begin{array}{ll}
				&\displaystyle c_{i}\int^{t_{1}}_{0}\frac{d}{dt}\int_{\Omega^{i}(t)}\theta_{i}(x,t)\zeta_{i}(x,t)-c_{i}\int^{t_{1}}_{0}\int_{\Omega^{i}(t)}\theta_{i}(x,t)(\partial_{t}+u_{i}\cdot\nabla)\zeta_{i}+k_{i}\int^{t_{1}}_{0}\int_{\Omega^{i}(t)}\nabla\theta_{i}\cdot\nabla\zeta_{i}\\
				&\displaystyle+\lambda\mathcal{I}_{i}\int^{t_{1}}_{0}\int_{\partial\Omega^{2}(t)}(\theta_{1}-\theta_{2})\zeta_{i}\\
				\displaystyle  
				&\displaystyle = \int^{t_{1}}_{0}\int_{\Omega^{i}}\mu_{i}(\theta_{i})|Du_{i}|^{2}\zeta_{i}+ {\kappa}\int^{t_{1}}_{0}\int_{\Omega^{i}}|\nabla^{k_{0}}u_{i}|^{2}\zeta_{i}+\mathcal{X}_{i}\frac{1}{2h}\int^{t_{1}}_{0}\int_{\Gamma}\bigg|\partial_{t}{\eta}-\beta\bigg|^{2}\zeta_{i}\\[4.mm]
				&\displaystyle+\frac{1}{h}\int^{t_{1}}_{0}\int_{\Omega^{i}}|u_{i}-w_{i}\circ(\Phi^{i})^{-1}|^{2}\zeta_{1}
				+\kappa\mathcal{X}_{i}\int^{t_{1}}_{0}\int_{\Gamma}\bigg|\nabla^{k_{0}+1}\partial_{t}{\eta}\bigg|^{2}\zeta_{1},
			\end{array}
		\end{equation}
		for all $\zeta_{i}\in C^{\infty}_{c}([0,h];C^{\infty}(\overline{\Omega^{i}})),$ where $\mathcal{I}_{1}=1,$ $\mathcal{I}_{2}=-1$ and $\mathcal{X}_{1}=1,$ $\mathcal{X}_{2}=0.$ \\[2.mm]
		$(iv)$ The themperature solves the positivity criterian, $\theta_{i}\geqslant\gamma>\gamma_i>0$ in $\Omega^{i}\times(0,h),$.\\[2.mm]
		$(v)$ The total energy balance holds in the sense
			\begin{equation}\label{hlevelenergy}
				\begin{array}{ll}
					&\displaystyle \sum_{i}c_{i}\int_{\Omega^{i}(t_{1})}\theta_{i}(t_{1})+K_{\kappa}(t_{1})+\frac{1}{2h}\int^{t_{1}}_{0}\bigg(\int_{\Gamma}|\partial_{t}\eta|^{2}+\sum_{i}\int_{\Omega^{i}(t)}|u_{i}|^{2}\bigg)\\
					&\displaystyle = K_{\kappa}(0)+\frac{1}{2h}\int^{t_{1}}_{0}\int_{\Gamma}|\beta|^{2}+\frac{1}{2h}\int^{t_{1}}_{0}\int_{\Omega^{i}(t)}|w_{i}\circ(\Phi^{i})^{-1}|^{2}
				\end{array}
			\end{equation}
			for $a.e.$ $t_{1}\in(0,h).$\\[2.mm]
			$(vi)$ The following identity involving the dissipative terms (both physical and artificial) holds:\\
			\begin{equation}\label{frommomentumaftertestingdiss}
				\begin{array}{ll}
					&\displaystyle K_{\kappa}\eta(t_{1})+\frac{1}{2h}\int^{t_{1}}_{0}\int_{\Gamma}|\partial_{t}\eta|^{2}
					\displaystyle+\sum_{i}\frac{1}{2h}\int^{t_{1}}_{0}\int_{\Omega^{i}(t)}|u_{i}|^{2}+\sum_{i}\left\langle \mathcal{D}^{i},1\right\rangle\\
					&\displaystyle =K_{\kappa}(0)+\frac{1}{2h}\int^{t_{1}}_{0}\int_{\Gamma}|\beta|^{2}+\sum_{i}\frac{1}{2h}\int^{t_{1}}_{0}\int_{\Omega^{i}(t)}|w_{i}\circ(\Phi^{i})^{-1}|^{2}\\
				\end{array}
			\end{equation}
			for $a.e.$ $t_{1}\in(0,h),$ where the dissipation $\left\langle \mathcal{D}^{i},1\right\rangle$ is given as
			\begin{equation}
				\begin{array}{ll}
					&\displaystyle \left\langle \mathcal{D}^{i},1\right\rangle=\bigg(\frac{1}{2h}\int^{t_{1}}_{0}\int_{\Gamma}|\partial_{t}\eta-\beta|^{2}+\frac{1}{2h}\int^{t_{1}}_{0}\int_{\Omega^{i}(t)}|u_{i}-w_{i}\circ(\Phi^{i})^{-1}|^{2}\\
					&\displaystyle\qquad\qquad+\int^{t_{1}}_{0}\int_{\Omega^{i}(t)}\mu^{M}_{i}(\theta_{i})|Du_{i}|^{2}+\kappa\int^{t_{1}}_{0}\int_{\Gamma}|\nabla^{k_{0}+1}\partial_{t}\eta|^{2}+\kappa\int^{t_{1}}_{0}\int_{\Omega^{i}(t)}|\nabla^{k_{0}}u_{i}|^{2}\bigg).
				\end{array}
			\end{equation}
			$(vii)$ An entropy evolution holds in the sense
			\begin{equation}\label{entropybal}
				\begin{array}{ll}
					\displaystyle\sum_{i=1}^{2}\left(c_{i}\frac{d}{dt}\int_{\Omega^{i}}\varphi(\theta_{i})+k_{i}\int_{\Omega^{i}}|\nabla\theta_{i}|^{2}\varphi''(\theta_{i})\right)+\lambda\int_{\partial\Omega^{2}}\left(\theta_{1}-\theta_{2}\right)\left(\varphi'(\theta_{1})-\varphi'(\theta_{2}))\right)\geqslant 0.
				\end{array}
			\end{equation}
			where $\varphi(\theta)$ is monotone and concave for $\theta>0$ such that all the integrals in \eqref{entropybal} make sense.
		\end{definition}
		Now let us state the result on the existence of a solution $(u_{i},\theta_{i},\eta)$ of the weak formulation \eqref{discretemomen1}-\eqref{entropybal}.
		\begin{theorem}\label{Thmmaintaulev}
			Suppose $h>0$ is sufficiently small. Let all the assumptions $(i)-(v)$ stated in Theorem \ref{Th:main} hold along with the further regularity $\eta_{0}\in W^{k_{0}+1,2}(\Gamma).$ Moreover, suppose a known/ previous solid velocity $\beta:  \Gamma\times[0,h] \rightarrow \mathbb{R}^{n}$ and a corresponding quantity $w_{i}:\Omega^{i}_{0}\times[0,h]\rightarrow\mathbb{R}^{n}$ for the fluids, such that $\beta\in L^{2}(\Gamma\times(0,h))$ and $w_{i}\in L^{2}(\Omega^{i}_{0}\times(0,h)).$ Then there exists a triplet $(u_{i},\theta_{i},\eta)$ solving \eqref{regularity2} and the criteria $(i)$ to $(vii)$ listed in Definition \ref{Def:taulayer} hold.
		\end{theorem}
		In order to solve for $(u_{i},\theta_{i},\eta)$ satisfying \eqref{discretemomen1}-\eqref{entropybal} we split $[0,h]$ into small time steps of length $\tau<< h$ and solve suitably constructed time discrete equations. 
		\subsection{$\tau$ layer, the itirative approximations}
		We assume that $h,\kappa>0$ are fixed and define
		\begin{equation}\label{betawk}
			\begin{array}{l}
				\displaystyle\beta_{k}=\frac{1}{\tau}\int\limits_{\tau k}^{\tau(k+1)}\beta:\Gamma\rightarrow \mathbb{R},\,\,\mbox{and}\,\,\displaystyle w_{i,k}=\frac{1}{\tau}\int\limits_{\tau k}^{\tau(k+1)}w_{i}:\Omega^{i}_{0}\rightarrow\mathbb{R}^{3}
			\end{array}
		\end{equation}
		are given. We will also fix a suitably small $\tau$ after the proof of Corollary \ref{lemmasolmin2}.  Further let a known function $\eta_{k}\nu:\Gamma\rightarrow \Omega$ decomposes $\Omega$ into two disjoint sets as follows $$\Omega=\Omega^{1}_{k}\cup\overline{\Omega^{2}_{k}}$$
		such that
		$\partial\Omega^{2}_{k}=\varphi_{\eta_{k}}(\Gamma,\cdot)=\Sigma_{\eta_{k}},$ $\Omega^{1}_{k}=\Omega\setminus\overline{\Omega^{2}_{k}}$ and $\partial\Omega^{1}_{k}=\Sigma_{\eta_{k}}\cup \partial\Omega.$ We also assert the existence of $\theta_{i,k}:\Omega^{i}_{k}\rightarrow \mathbb{R}$ and a diffeomorphism $\Phi^{i}_{k}:\Omega^{i}_{0}\rightarrow \Omega^{i}_{k}.$\\ 
		Let us introduce an approximate of the fluid viscosities in a way so that they do not blow off as the temperature approaches $\gamma_{i}$ (as introduced in \eqref{viscosity}) as follows
		\begin{equation}\label{approxmui}
			\mu_{i}^{M}(x)=\left\{ \begin{array}{ll}
				&\displaystyle \mu_{i}(x),\,\,\mbox{when}\,\,x\geqslant \gamma_{i}+\frac{1}{M},\\
				&\displaystyle \mu_0^ie^{M\beta_i},\,\,\mbox{otherwise}.
			\end{array}\right.
		\end{equation}
		We need this approximation since we a priori do not have the positivity of $\theta_{i}$ at our disposal. Once this relation is guaranteed in an intermediate step we can get rid of the artificial approximation of the viscosity. 
		Observe that
		\begin{equation}\label{muiN}
			\begin{array}{l}
				\displaystyle\mu^{i}_{0}\leqslant \mu^{M}_{i}(x) \leqslant \mu_0^ie^{M\beta_i}=:\mu_i\left(\gamma_{i}+\frac{1}{M}\right).
			\end{array}
		\end{equation}
		Next we define $\eta_{k+1}:\Gamma\rightarrow \Omega,$ $u_{i,k+1}:\Omega^{i}_{k}\rightarrow\mathbb{R}^{3}$ and $\theta_{i,k+1}:\Omega^{i}_{k+1}\rightarrow \mathbb{R}$ so that they consecutively solve the following minimization problems.\\
		$(a)$ {\textit{First minimization problem:}}
		solve for $(u_{i,k+1},\eta_{k+1})$ (for $i\in\{1,2\}$):
		\begin{equation}\label{min1}
			\begin{array}{ll}
				\displaystyle\min_{(u_{i},\eta)}&\displaystyle \bigg\{K_{\kappa}(\eta)+\frac{\tau}{2h}\bigg[\int_{\Gamma}\bigg|\frac{\eta-\eta_{k}}{\tau}-\beta_{k}\bigg|^{2}+\sum_{i}\int_{\Omega^{i}_{0}} \left|u_{i}\circ\Phi^{i}_{k}-w_{i,k}\right|^{2}\bigg]+\frac{\tau\kappa}{2}\sum_{i}\int_{\Omega^{i}_{k}}\left|\nabla^{k_{0}}u_{i}\right|^{2}\\
				&\displaystyle +\frac{\tau}{2}\sum_{i} \int_{\Omega^{i}_{k}} \mu^{M}_{i}(\theta_{i,k})||Du_{i}|^{2}+\frac{\tau\kappa}{2}\int_{\Gamma}\left|\nabla^{k_{0}+1}\bigg(\frac{\eta-\eta_{k}}{\tau}\bigg)\right|^{2}\bigg\}
			\end{array}
		\end{equation}
		where we require $\eta\in  W^{k_{0}+1,2}(\Gamma),$ $u_{i}\in W^{k_{0},2}(\Omega_{k}),$ $\mbox{div}\,u_{i}=0$ on $\Omega^{i}_{k}$ and subject to the inter-facial coupling condition 
		\begin{equation}\label{interfacecoupdis}
			\begin{array}{ll}
				\displaystyle\frac{\eta-\eta_{k}}{\tau}\nu=u_{i}\circ \varphi_{\eta_{k}}\,\,\mbox{on}\,\,\Gamma.
			\end{array}
		\end{equation}
		Next we define $\Psi_{u_{i,k+1}}$ in $\overline{\Omega^{i}_{k}}$ which solves $\Psi_{u}=Id+\tau u$ and update $\Phi^{i}_{k}$ to $\Phi^{i}_{k+1}$ by setting 
		$$\Phi^{i}_{k+1}=\Psi_{u_{i,k+1}}\circ \Phi^{i}_{k}.$$ 
		$(b)$ \textit{Second minimization problem:} solve for $\theta_{i,k+1}$ as $\displaystyle \min_{\theta_{i}\in\mathcal{A}_{i,k+1}}T_{k}(\theta_{1},\theta_{2})$ (for $i\in\{1,2\}$), where
		\begin{align}\label{min2}
			\begin{aligned}
				T_k:&\mathcal{A}_{1,k+1}\times \mathcal{A}_{2,k+1}\to \mathbb{R},
				\\
				T_k(\theta_1,\theta_2):=&\displaystyle\bigg\{\sum\limits_{i} \frac{ \tau c_{i}}{2}\int_{\Omega^{i}_{k}}\bigg|\frac{\theta_{i}\circ\Psi_{u_i,{k+1}}-\theta_{i,k}}{\tau}\bigg|^{2}+\sum_{i}\frac{k_{i}}{2}\int_{\Omega^{i}_{k+1}}|\nabla\theta_{i}|^{2}
				\\
				&\displaystyle {+\frac{\lambda}{2} \int_{\partial\Omega^{2}_{k+1}}\left|\theta_{1}-\theta_{2}\right|^{2}-\sum\limits_{i}\int_{\Omega^{i}_{k}}{\mu_{i}^{M}(\theta_{i,k})}|D u_{i,k+1}|^{2}\theta_{i}\circ\Psi_{u_{i,k+1}}}\\
				&\displaystyle{
					-\frac{1}{2h}\int_{\Gamma}\theta_{1}\left|\frac{\eta_{k+1}-\eta_{k}}{\tau}-\beta_{k}\right|^{2}
				}\\
				&\displaystyle -{{\kappa}{}\sum_{i}\int_{\Omega^{i}_{k}}\min\{|\nabla^{k_{0}}u_{i,k+1}|^{2},\tau^{-1}\}\theta_{i}\circ\Psi_{u_{i,k+1}}}{	-\frac{1}{2h}\sum_{i}\int_{\Omega^{i}_{0}}\left|u_{i,k+1}\circ\Phi^{i}_{k}-w_{i,k}\right|^{2}\theta_{i}\circ\Phi^{i}_{k+1}}\\
				&\displaystyle {-\kappa\int_{\Gamma}\min\Big\{\left|\nabla^{k_{0}+1}\bigg(\frac{\eta_{k+1}-\eta_{k}}{\tau}\bigg)\right|^{2},\tau^{-1}\Big\}\theta_1\bigg\}
				}
			\end{aligned}
		\end{align}
		and
		\begin{equation}\label{defmA}
			\begin{array}{l}
				\displaystyle \mathcal{A}_{i,k+1}=\{\theta_{i}\in W^{1,2}(\Omega^{i}_{k+1})\suchthat \theta_{i}\geqslant\gamma\,\,\mbox{in}\,\,\overline{\Omega^{i}_{k+1}}\}
			\end{array}
		\end{equation}
		\begin{remark}
			The above introduced double minimization allows for the following physical interpretation/speculation. First in (a) we produce an optimal direction to follow (for a given setting of a state with given temperature/viscosity/elasticity). The direction is chosen optimally in order to reduce the energy under minimal dissipative cost. The produced dissipation then is turned into heat via the second minimzation procedure. Here the dissipative forcing is appropriately turned into heat distributed w.r.t. the assumed heat conductivity. 
		\end{remark}
		The next two sections are devoted for solving the minimization problems \eqref{min1} and \eqref{min2}, deriving the related Euler-Lagrange equations and obtaining some a priori estimates for $(u_{i,k+1},\eta_{k+1},\theta_{i,k+1}).$
		\subsubsection{Solving \ref{min1} and some a priori estimates}
		The first lemma of this section corresponds to the existence of a minimizer $(u_{i,k+1},\eta_{k+1})$ of \eqref{min1} and the Euler-Lagrange equation.
		\begin{lem}\label{lemmasolmin1}
			$(i)$	Let $\eta_{k}$ $w_{i,k},$ $\beta_{k}$ and $\theta_{i,k}$ are given functions such that
			\begin{equation}\label{assumthwek}
				\begin{array}{ll}
					&\displaystyle \eta_{k}\in  W^{k_{0}+1,2}(\Gamma),\,\, \theta_{i,k}\in W^{1,2}(\Omega_{k}),\,\,w_{i,k}\in L^{2}(\Omega_{0}^{i}),\,\,\beta_{k}\in L^{2}(\Gamma),\\[3.mm]
					&\displaystyle \Omega=\Omega^{1}_{k}\cup\overline{\Omega^{2}_{k}}\,\,
					\mathrm{such\,\, that}\,\,
					\partial\Omega^{2}_{k}=\varphi_{\eta_{k}}(\Gamma)=\Sigma_{\eta_{k}},\,\, \Omega^{1}_{k}=\Omega\setminus\overline{\Omega^{2}_{k}}\\[3.mm]
					&\displaystyle\mbox{and}\,\,\partial\Omega^{1}_{k}=\Sigma_{\eta_{k}}\cup\partial\Omega.
				\end{array}
			\end{equation}
			Then there exists a positive constant $\tau_{1}<1$ such that for all $\tau\in (0,\tau_{1})$ the minimization problem \eqref{min1} has  a solution $(u_{i,k+1},\eta_{k+1})$ such that
			\begin{equation}\label{etauikp1}
				\begin{array}{ll}
					&\displaystyle u_{i,k+1}\in W^{k_{0},2}(\Omega_{k}),\,\,\eta_{k+1}\in  W^{k_{0}+1,2}(\Gamma),\,\,\inf_{\Omega_{k}}\det\left(Id+\tau\nabla u_{i,k+1}\right)>0,\\
					&\displaystyle \Omega=\Omega^{1}_{k+1}\cup\overline{\Omega^{2}_{k+1}}\,\,
					\mathrm{such\,\, that}\,\,
					\partial\Omega^{2}_{k+1}=\varphi_{\eta_{k+1}}(\Gamma)=\Sigma_{\eta_{k+1}},\,\, \Omega^{1}_{k+1}=\Omega\setminus\overline{\Omega^{2}_{k+1}}\\[3.mm]
					&\displaystyle\mbox{and}\,\,\partial\Omega^{1}_{k+1}=\Sigma_{\eta_{k+1}}\cup\partial\Omega.
				\end{array}
			\end{equation}
			$(ii)$ Any solution of the minimization problem \eqref{min1} solves
			\begin{equation}\label{ELagmin1}
				\begin{array}{ll}
					&\displaystyle \langle DK_{\kappa}(\eta_{k+1}),b\rangle+\frac{1}{h}\int_{\Gamma}\bigg(\frac{\eta_{k+1}-\eta_{k}}{\tau}-\beta_{k}\bigg)b+\kappa\int_{\Gamma}\nabla^{k_{0}+1}\bigg(\frac{\eta_{k+1}-\eta_{k}}{\tau}\bigg):\nabla^{k_{0}+1}b\\
					&\displaystyle+\sum_{i}\bigg\langle\frac{u_{i,k+1}\circ\Phi^{i}_{k}-w_{i,k}}{h},\psi\circ\Phi^{i}_{k}\bigg\rangle_{\Omega^{i}_{0}}\displaystyle+{\kappa}\sum_{i} \int_{\Omega^{i}_{k}}\nabla^{k_{0}}u_{i,k+1}:\nabla^{k_{0}}\psi\\
					&\displaystyle+\sum_{i}\int_{\Omega^{i}_{k}}\mu^{M}_{i}(\theta_{i,k})D(u_{i,k+1}):D(\psi)=0
				\end{array}
			\end{equation}
			and the interfacial coupling condition $	\displaystyle\frac{\eta_{k+1}-\eta_{k}}{\tau}\nu=u_{i,k+1}\circ \varphi_{\eta_{k}}\,\,\mbox{on}\,\,\Gamma,$ for $a.e$ $t\in[0,h]$ and all $b\in W^{k_{0}+1,2}(\Gamma),$ $\psi\in W^{k_{0},2}(\overline{\Omega})$ with $\mathrm{div}\,\psi=0$ on $\Omega$ satisfying $\psi\circ\varphi_{\eta_{k}}=b$ on $\Gamma.$
		\end{lem}
		The proof of Lemma \ref{lemmasolmin1} can be adapted from \cite[Lemma 3.2, Corollary 3.3]{BreitKamSch} (please see also \cite[Theorem 4.2, Proposition 4.3]{BenKamSch}) in a straight forward manner. Note that while deriving the Euler-Lgrange equation \eqref{ELagmin1} from \eqref{min1} one requires the perturbation 
		$(\eta_{k+1}+\epsilon b, u_{i,k+1}+\epsilon\frac{\psi}{\tau})$
		for some divergencfree smooth testfunction $\psi$ and $\epsilon>0$. The different scaling of the structural displacement and the fluid velocities allows us to remove most of the occurrences of $\tau$ in the Euler-Lagrange equation.\\
		We would next show some bounds of $(u_{i,k},\eta_{k})^{N}_{k=1}$ in from of the following lemma and its corollary.
		\begin{lem}\label{bounduetakp1}
			Suppose $\eta^{0},$ $\beta$ and $w_{i}$ (one also recalls the definition of $\beta_{k}$ and $w_{i,k}$ from \eqref{betawk}) are given such that
			\begin{equation}\label{assuminbw}
				\begin{array}{l}
					\displaystyle\eta^{0}\in  W^{k_{0}+1,2}(\Gamma),
					\displaystyle \beta\in L^{2}(\Gamma\times[0,h]),\displaystyle w_{i}\in L^{2}(\Omega^{i}_{0}\times[0,h])
				\end{array}
			\end{equation}
			holds. Let $\tau\in (0,\tau_{1})$ (where $\tau_{1}$ is as obtained in Lemma \ref{lemmasolmin1}).
			Further suppose that $(\theta_{i,k})^{N-1}_{k=0}$ are known and $\theta_{i,k}$ belong to $W^{1,2}(\Omega^{i}_{k}).$ Then the solution $(u_{i,k},\eta_{k})^{N}_{k=1}$ to \eqref{ELagmin1} satisfy
			\begin{subequations}
				\begin{alignat}{2}
					&\displaystyle K_{\kappa}(\eta_{N})+\sum\limits_{k=0}^{N-1}\tau{\kappa}\sum\limits_{i} \int\limits_{\Omega^{i}_{k}}|\nabla^{k_{0}}u_{i,k+1}|^{2}+\sum\limits_{k=0}^{N-1}\tau\sum\limits_{i}\int_{\Omega^{i}_{k}}\mu^{M}_{i}(\theta_{i,k})|D(u_{i,k+1})|^{2}\nonumber\\
					&\displaystyle+\sum\limits_{k=0}^{N-1}\frac{\tau}{2h}\bigg[\int\limits_{\Gamma} \bigg|\frac{\eta_{k+1}-\eta_{k}}{\tau}\bigg|^{2}+\sum\limits_{i}\int\limits_{\Omega^{i}_{0}}\bigg|u_{i,k+1}\circ\Phi^{i}_{k}\bigg|^{2}\bigg]+\sum_{k=0}^{N-1}\tau\kappa\int_{\Gamma}\left|\nabla^{k_{0}+1}\bigg(\frac{\eta_{k+1}-\eta_{k}}{\tau}\bigg)\right|^{2}\nonumber\\
					&\displaystyle+\sum_{k=0}^{N-1}D^{2}K_{\kappa}(\xi_{\eta_{k},\eta_{k+1}})\bigg((\eta_{k+1}-\eta_{k}),(\eta_{k+1}-\eta_{k})\bigg)+\sum_{k=0}^{N-1}\frac{\tau}{2h}\int_{\Gamma}\bigg|\frac{\eta_{k+1}-\eta_{k}}{\tau}-\beta_{k}\bigg|^{2}\nonumber\\[2.mm]
					&\displaystyle+\sum_{k=0}^{N-1}\frac{\tau}{2h}\int_{\Omega^{i}_{0}}|u_{i,k+1}\circ\Phi^{i}_{k}-w_{i,k}|^{2}\nonumber\\
					&\displaystyle = K_{\kappa}(\eta^{0})+\sum\limits_{k=0}^{N-1}\frac{\tau}{2h}\bigg[\int\limits_{\Gamma}|\beta_{k}|^{2}+\int\limits_{\Omega^{i}_{0}} |w_{i,k}|^{2}\bigg]\label{consmin1}\\
					&\displaystyle \leqslant C\bigg(\eta^{0},h,\|\beta\|_{L^{2}(\Gamma\times[0,h])},\|w_{i}\|_{L^{2}(\Omega^{i}_{0}\times[0,h])}\bigg)\label{consmin1*}.
				\end{alignat}
			\end{subequations}
		\end{lem}  
		\begin{proof}
			First to show \eqref{consmin1} we use the test function $(\psi,b)=(u_{1,k+1}\chi_{\Omega_k^1}+u_{2,k+1}\chi_{\Omega_k^2},\frac{\eta_{k+1}-\eta_{k}}{\tau})$ in \eqref{ELagmin1} to find
			\begin{equation}\label{testingdmomen}
				\begin{array}{ll}
					&\displaystyle \frac{1}{\tau}\bigg(K_{\kappa}(\eta_{k+1})-K_{\kappa}(\eta_{k})+D^{2}K_{\kappa}(\xi_{\eta_{k},\eta_{k+1}})\bigg((\eta_{k+1}-\eta_{k}),(\eta_{k+1}-\eta_{k})\bigg) \bigg)\\
					&\displaystyle+\frac{1}{2h}\int_{\Gamma}\bigg(\left|\frac{\eta_{k+1}-\eta_{k}}{\tau}\right|^{2}-|\beta_{k}|^{2}
					\displaystyle	+\left|\frac{\eta_{k+1}-\eta_{k}}{\tau}-\beta_{k}\right|^{2}\bigg)\\
					&\displaystyle+\frac{1}{2h}\sum_{i}\int_{\Omega^{i}_{0}}\bigg(|u_{i,k+1}\circ\Phi^{i}_{k}|^{2}-|w_{i,k}|^{2}+|u_{i,k+1}\circ\Phi^{i}_{k}-w_{i,k}|^{2}\bigg)\\
					&\displaystyle +{\kappa}\sum_{i}\int_{\Omega^{i}_{k}}|\nabla^{k_{0}}u_{i,k+1}|^{2}+\sum_{i}\int_{\Omega^{i}_{k}}\mu^{M}_{i}(\theta_{i,k})|D(u_{i,k+1})|^{2}+\kappa\int_{\Gamma}\bigg|\nabla^{k_{0}+1}\bigg(\frac{\eta_{k+1}-\eta_{k}}{\tau}\bigg)\bigg|^{2}=0,
				\end{array}
			\end{equation}
			where we have used  the formula
			$$A\cdot(A-B)=\frac{1}{2}\left(|A|^{2}-|B|^{2}+|A-B|^{2}\right),\,\,\mbox{for\,\,all}\, A,B\in \mathbb{R}^{m},\,m\geqslant 1$$
			to expand $\displaystyle\left(\frac{\eta_{k+1}-\eta_{k}}{\tau}-\beta_{k}\right)\frac{\eta_{k+1}-\eta_{k}}{\tau}$ and $\displaystyle \left(u_{i,k+1}\circ\Phi^{i}_{k}-w_{i,k}\right) u_{i,k+1}\circ\Phi^{i}_{k}.$ In \eqref{testingdmomen} $\xi_{k,k+1}$ is a point lying on the line joining $\eta_{k}$ and $\eta_{k+1}$ such that there holds the identity
			\begin{equation}\label{meanvalue}
				\begin{array}{ll}
					&	\displaystyle K_{\kappa}(\eta_{k+1})-K_{\kappa}(\eta_{k})+D^{2}K_{\kappa}(\xi_{\eta_{k},\eta_{k+1}})\bigg((\eta_{k+1}-\eta_{k}),(\eta_{k+1}-\eta_{k})\bigg)\\
					&\displaystyle=\left\langle DK_{\kappa}(\eta_{k+1}),(\eta_{k+1}-\eta_{k})\right\rangle.
				\end{array}
			\end{equation} 
			In the expression above the notation $D^{2}K_{\kappa}(\xi_{\eta_{k},\eta_{k+1}})\bigg((\eta_{k+1}-\eta_{k}),(\eta_{k+1}-\eta_{k})\bigg)$ stands for the action
			\begin{equation}\nonumber
				\begin{array}{ll}
					&\displaystyle D^{2}K_{\kappa}(\xi_{\eta_{k},\eta_{k+1}})\bigg((\eta_{k+1}-\eta_{k}),(\eta_{k+1}-\eta_{k})\bigg)\\
					&\displaystyle= \left(D^{2}K_{\kappa}(\xi_{\eta_{k},\eta_{k+1}}) \bigg(\eta_{k+1}-\eta_{k}\bigg),(\eta_{k+1}-\eta_{k})\right).
				\end{array}
			\end{equation}
			Now summing \eqref{testingdmomen} over $k=0,...,N-1$ we conclude the proof of \eqref{consmin1}.\\
			The proof of \eqref{consmin1*} follows from \eqref{consmin1} since
			\begin{equation}\nonumber
				\begin{array}{l}
					\displaystyle \sum_{k=0}^{N-1}\tau\|w_{i,k}\|^{2}_{L^{2}(\Omega^{i}_{0})}=\sum_{k=0}^{N-1}\tau\bigg\|\frac{1}{\tau}\int_{k\tau}^{(k+1)\tau}w_{i}\bigg\|^{2}_{L^{2}(\Omega^{i}_{0})}\leqslant \sum_{k=0}^{N-1}\tau\frac{1}{\tau}\int_{k\tau}^{(k+1)\tau}\|w_{i}\|^{2}_{L^{2}(\Omega^{i}_{0})}=\int_{0}^{h}\|w_{i}\|^{2}_{L^{2}(\Omega^{i}_{0})}
				\end{array}
			\end{equation}
			(as a consequence of Jensen's inequality) and a similar estimate holds for $\sum_{k=0}^{N-1}\tau\|\beta_{k}\|^{2}_{L^{2}(\Gamma)}.$ 
			
		\end{proof}	
		\begin{corollary}\label{Coretaukp1}
			Let all the assumptions of Lemma \ref{bounduetakp1} hold. Then $(u_{i,k})^{N}_{k=1}$ (where $(u_{i,k+1},\eta_{k+1})^{N-1}_{k=0}$ solves \eqref{ELagmin1}) satisfy the following inequalities
			\begin{equation}\label{C1alphaestuikp1}
				\begin{array}{l}
					\displaystyle\sum_{k=0}^{N-1}\tau\sum_{i}\|u_{i,k+1}\|^{2}_{C^{1,\alpha}(\Omega^{i}_{k})}\leqslant C\bigg(\eta^{0},h,\|\beta\|_{L^{2}(\Gamma\times[0,h])},\|w_{i}\|_{L^{2}(\Omega^{i}_{0}\times[0,h])}\bigg)
				\end{array}
			\end{equation}
			for some $\alpha>0,$
			more precisely
			\begin{equation}\label{LinftyDuikp1}
				\begin{array}{ll}
					&\displaystyle\sum_{k=0}^{N-1}\tau\sum_{i}\|Du_{i,k+1}\|^{2}_{L^{\infty}(\Omega^{i}_{k})}\\
					&\displaystyle\leqslant C\bigg(\eta^{0},h,\|\beta\|_{L^{2}(\Gamma\times[0,h])},\|w_{i}\|_{L^{2}(\Omega^{i}_{0}\times[0,h])}\bigg).
				\end{array}
			\end{equation} 
		\end{corollary}
		\begin{proof}
			One derives the following from \eqref{consmin1*}
			\begin{equation}\label{1stinq}
				\begin{array}{l}
					\displaystyle\sum\limits_{k=0}^{N-1}\tau{\kappa}\sum\limits_{i} \int\limits_{\Omega^{i}_{k}}|\nabla^{k_{0}}u_{i,k+1}|^{2}+\sum_{k=0}^{N-1}\tau\sum_{i}\int_{\Omega^{i}_{k}}|u_{i,k+1}|^{2}\leqslant C(\eta^{0},h,\|\beta\|_{L^{2}(\Gamma\times[0,h])},\|w_{i}\|_{L^{2}(\Omega^{i}_{0}\times[0,h])}).
				\end{array}
			\end{equation}
			Next we use the fact that $k_{0}\geqslant 3$ and the continuous embedding $W^{3,2}(\Omega^{i}_{k})\hookrightarrow C^{1,\alpha}(\Omega^{i}_{k})$ for some $\alpha>0$ to render \eqref{C1alphaestuikp1} and consequently \eqref{LinftyDuikp1} from \eqref{1stinq}.
		\end{proof}
		\subsubsection{Solving \ref{min2}} 
		Before we state and prove a lemma corresponding to the solution of the minimization problem \ref{min2}, we introduce a result on the diffeomorphic properties of the map $\Phi^{i}_{k}:\Omega^{i}_{0}\rightarrow\Omega^{i}_{k}$ for suitably small $\tau.$
		\begin{prop}\label{diffPhi}
			There is a $\tau_{2}\in (0,\tau_{1})>0$ (where $\tau_{1}$ is as obtained in Lemma \ref{lemmasolmin1}) such that for all $\tau\in(0,\tau_{2}),$ we have that $\Phi^{i}_{k}:\Omega_{0}\rightarrow\Omega^{i}_{k}$ is a diffeomorphism with
			\begin{equation}\label{bnddetnablaPhi}
				\begin{array}{ll}
					1-C\tau\leqslant \det{\nabla\Phi^{i}_{k}}\leqslant 1+C\tau,\quad\forall \quad k\leqslant N<\frac{h}{\tau}.
				\end{array}
			\end{equation}
			Also $\Psi_{u_{i},k}:\Omega_{k-1}^{i}\rightarrow\Omega^{i}_{k}$ is a diffeomorphism and
			\begin{equation}\label{bnddetnablaPsi}
				\begin{array}{l}
					1-C\tau^2\leqslant\det\nabla\Psi^{i}_{k}\leqslant 1+C\tau^2,\quad\forall\quad k\leqslant N<\frac{h}{\tau}.
				\end{array}
			\end{equation}
		\end{prop}
		\begin{proof}[Comments on the proof of Proposition \ref{diffPhi}] That the map $\Phi^{i}_{k}$ is a diffeomorphism for small enough values of $\tau$ and the fact \eqref{bnddetnablaPhi} holds can be derived by using \eqref{C1alphaestuikp1}, c.f. \cite[Prop. 4.6]{BenKamSch} (applied for the $i-$th fluid, $i\in\{1,2\}$).\\
			The inequality \eqref{bnddetnablaPsi} follows directly from the proof of \cite[Prop. 4.6]{BenKamSch}. 
		\end{proof}
		
		\begin{prop}\label{propexminth}
			Assume that $\theta_{i,k}\in W^{1,2}(\Omega_{k})$, $\Psi_{u_i,k}\in C^1(\Omega_k^i,\Omega_{k+1}^{i+1})$ is a diffeomorphism,  $\theta_{i,k}\geq \gamma_0\geqslant 0$ and
			\\
			$
			|\nabla^{k_{0}}u_{i,k+1}|^{2},\left|u_{i,k+1}\circ\Phi^{i}_{k}-w_{i,k}\right|^{2}\in L^1(\Omega_k^i)$,
			$\left|\frac{\eta_{k+1}-\eta_{k}}{\tau}-\beta_{k}\right|^{2},\left|\nabla^{k_{0}+1}\bigg(\frac{\eta_{k+1}-\eta_{k}}{\tau}\bigg)\right|^{2}\in L^1(\Gamma)$,
			then there exists a minimizer of $T_k :\mathcal{A}_{1,k+1}\times\mathcal{A}_{2,k+1} \to [\gamma_0,\infty), $ defined in \eqref{min2}-\eqref{defmA}. 
		\end{prop}
		\begin{proof}
			The existence of $\theta_i$ follows by the direct method in the calculus of variations. Indeed, by the choices of $k_0$ the functional $T$ is bounded from below a-priori depending on $\tau$. Indeed one uses the following inequality
			\begin{equation}\label{1stmin2}
				\begin{array}{ll}
					\displaystyle\sum_{i}\frac{c_{i}}{2\tau}\|\theta_{i}\circ\Psi_{u_{i},k+1}-\theta_{i,k}\|^{2}_{L^{2}(\Omega^{i}_{k})}&\displaystyle\geqslant\sum_{i}\frac{c_{i}}{2\tau}\bigg(\|\theta_{i}\circ\Psi_{u_{i},k+1}\|_{L^{2}(\Omega^{i}_{k})}-\|\theta_{i,k}\|_{L^{2}(\Omega^{i}_{k})}\bigg)^{2},\\
					&\displaystyle \geqslant  \sum_{i}\frac{1}{2\tau}\bigg(c_{1}\|\theta_{i}\|^{2}_{L^{2}(\Omega^{i}_{k+1})}-C_{1}\|\theta_{i,k}\|^{2}_{L^{2}(\Omega^{i}_{k})}\bigg)
				\end{array}
			\end{equation}
			to infer the following concerning a lower bound of $T_{k}:$
			$$\displaystyle T_{k}(\theta_{1},\theta_{2})\geqslant c\bigg(\sum_{i}{\tau}\|\theta_{i}\|^{2}_{W^{1,2}(\Omega^{i}_{k+1})}-c_{k}(\tau)\bigg),$$
			for possibly small positive constant $c.$
			Now based on this coercivity estimate, it is now standard  to show the weak lower semicontiunuity of the functional concerned and hence a desired minimiser can be found. 
			
			The positivity follows by the simple observation, that
			\[
			T_{k}(\max\{\theta_{1,k+1},\theta_{0}\},\max\{\theta_{2,k+1},\theta_0\})\leq T_{k}(\theta_{1,k+1},\theta_{2,k+1}).
			\]
			Indeed, on the set where $\max\{\theta_{1,k+1},\theta_{0}\}=\theta_0$ the gradient term becomes $0$. For the other terms there is a sign, as on the set where  $\max\{\theta_{1,k+1},\theta_{0}\}\leq \theta_0$, we have in particular that
			$\theta_{1,k+1}\circ{\bf \Psi}_u\leq \theta_0\leq  \theta_{1,k}$, but this implies that
			\[
			\abs{\theta_{1,k+1}\circ{ \Psi}_{u_{1,k+1}}-\theta_{1,k}}=\theta_{1,k}-\theta_{1,k+1}\circ{ \Psi}_{u_{1,k+1}}\geq \theta_{1,k}-\theta_0.
			\]
			Similarly if $\theta_{2,k+1}\geq \theta_0\geq \theta_{1,k+1}$ we find
			\[
			\abs{\theta_{1,k+1}-\theta_{2,k+1}}=\theta_{2,k+1}-\theta_{1,k+1}\geq \theta_{2,k+1}-\theta_0.
			\]
			As in case $\theta_0\geq \theta_{1,k+1}$ and $\theta_0\geq \theta_{2,k+1}$ the term vanishes, we are left only with the terms that have a negative sign, but they obviously contribute to decrease the functional as the temperature rises.
			
			Moreover, observe that the inequality would become strict, if $\theta_{1,k+1}(x)<\theta_{0}$ on a set of positive measure as on this set
			\[
			\abs{\theta_{1,k+1}\circ{\bf \Psi}_u-\theta_{1,k}}=\theta_{1,k}-\theta_{1,k+1}\circ{\bf \Psi}_u>\theta_{1,k}-\theta_0,
			\]
			hence the minimizer satisfies the bound.
		\end{proof}
		The following result is a direct corollary of Proposition \ref{propexminth}.
		\begin{corollary}\label{lemmasolmin2}
			$(i)$ Let $(u_{i,k+1},\eta_{k+1})$ is a solution to \eqref{min1} as obtained in Lemma \ref{lemmasolmin1}.
			Any solution $\theta_{i,k+1}$ ($i\in\{1,2\}$) of the minimization problem \eqref{min2} satisfies
			
			\begin{align}\label{ELagmin2}
				\begin{aligned}
					&c_{1}\int_{\Omega^{1}_{k}}\frac{\theta_{1}\circ\Psi_{u_1,{k+1}}-\theta_{1,k}}{\tau}\zeta_{1}\circ\Psi_{u_{1},k+1}+k_{1}\int_{\Omega^{1}_{k+1}}\nabla\theta_{1,k+1}\cdot\nabla\zeta_{1}+{\lambda}\int_{\partial\Omega^{2}_{k+1}}(\theta_{1,k+1}-\theta_{2,k+1})\zeta_{1}\\
					&-\int_{\Omega^{1}_{k}}\mu_{1}^{M}(\theta_{1,k})|Du_{1,k+1}|^{2}\zeta_{1}\circ\Psi_{u_{1,k+1}}
					-{\kappa}\int_{\Omega^{1}_{k}}
					\min\Big\{|\nabla^{k_{0}}u_{1,k+1}|^{2},\tau^{-1}\Big\}\zeta_{1}\circ\Psi_{u_{1,k+1}}
					\\
					&-\frac{1}{2h}\int_{\Gamma}\left|\frac{\eta_{k+1}-\eta_{k}}{\tau}-\beta_{k}\right|^{2}\zeta_{1}-\kappa\int_{\Gamma}\min\bigg\{\left|\nabla^{k_{0}+1}\bigg(\frac{\eta_{k+1}-\eta_{k}}{\tau}\bigg)\right|^{2},\tau^{-1}\bigg\}\zeta_{1}
					\\
					&-\frac{1}{2h}\int_{\Omega^{1}_{0}}\left|u_{1,k+1}\circ\Phi^{1}_{k}-w_{1,k}\right|^{2}\zeta_{1}\circ\Phi^{1}_{k+1}
				\end{aligned}
			\end{align}
			for all $\zeta_{1}\in  \mathcal{A}_{1,k+1}\cap W^{1,2}(\Omega)$.
			Further
			\begin{align}\label{ELagmin2*}
				\begin{aligned}
					& c_{2}\int_{\Omega^{2}_{k}}\frac{\theta_{2,k+1}\circ\Psi_{u_{2},k+1}-\theta_{2,k}}{\tau}\zeta_{2}\circ\Psi_{u_{2},k+1}+k_{2}\int_{\Omega^{2}_{k+1}}\nabla\theta_{2,k+1}\cdot\nabla\zeta_{2}-{\lambda}\int_{\partial\Omega^{2}_{k+1}}(\theta_{1,k+1}-\theta_{2,k+1})\zeta_{2}
					\\
					&-\int_{\Omega^{2}_{k}}\mu_{2}^{M}(\theta_{2,k})|Du_{2,k+1}|^{2}\zeta_{2}\circ\Psi_{u_{2,k+1}}
					-{\kappa}\int_{\Omega^{2}_{k}}\min\Big\{|\nabla^{k_{0}}u_{2,k+1}|^{2},\tau^{-1}\}\zeta_{2}\circ\Psi_{u_{2,k+1}}
					\\
			&	-\frac{1}{2h}\int_{\Omega^{2}_{0}}\left|u_{2,k+1}\circ\Phi^{2}_{k}-w_{2,k}\right|^{2}\zeta_{2}\circ\Phi^{2}_{k+1}
			=0.
		\end{aligned}
	\end{align}
	for all $\zeta_{2}\in \mathcal{A}_{2,k+1}\cap W^{1,2}(\Omega),$ where the set $\mathcal{A}_{i,k+1}$ is introduced in \eqref{defmA}. Note that we have considered test functions $\zeta_{i}$ in $\Omega,$ since we need them to be defined on $\Gamma$ to give proper sense to the terms involving integrals over $\Gamma$. Indeed any $\zeta_{i}\in \mathcal{A}_{i,k+1}$ modulo a canonical extension of $W^{1,2}(\Omega^{i}_{k+1})$ functions to $W^{1,2}(\Omega)$ can be used as a test function in the above weak formulations. 
\end{corollary}
The proof of the lemma follows by writing the Euler-Lagrange equation corresponding to the functional associated with the minimization problem \eqref{min2}-\eqref{defmA}.  
Next we derive an estimate of $\theta_{i,k+1}$ ($k=1,..,N-1$) which will be used later in Section \ref{interpole} to obtain suitable $\tau-$independent estimates of the interpolants approximating temperature. The estimate is a key tool to show that the gradient of the fluid temperatures admits of some integrability in fractional order Lebesgue spaces even with $L^{1}$ source terms.

\begin{lem}\label{estthetakp1}
For any convex $\phi\in C^2((0,\infty),(0,\infty)$, with $\phi' \geqslant 0$ and $\phi'\in L^\infty((0,\infty),(0,\infty))$, We find
\begin{equation}\label{estL2W12tg}
	\begin{array}{l}
		\displaystyle \sum_{i=1}^{2}\bigg(\int_{\Omega^{i}_{N-1}}\phi(\theta_{i,N-1})+\sum_{k=0}^{N-1}\tau\int_{\Omega^{i}_{k}}\kappa_i\phi''(\theta_{i,k})\abs{\nabla \theta_{i,k}}^2\bigg)\leq C,
	\end{array}
\end{equation}
where the constant on the right hand side only depends on the initial data, $w_{i}$ and $\beta.$\\
In particular using $\phi$, such that $\phi'(\theta)=1-\frac{1}{(1+\theta)^\beta}$ for $\beta>0$ implies that 
\begin{equation}\label{estL2W12tp}
	\begin{array}{l}
		\displaystyle \sum_{i=1}^{2}\bigg(\int_{\Omega^{i}_{N-1}}\bigg(\theta^{i}_{N-1}-\frac{(1+\theta^{i}_{N-1})^{1-\beta}}{(1-\beta)}\bigg)+\sum_{k=0}^{N-1}\tau\int_{\Omega^{i}_{k}}\kappa_i\frac{\abs{\nabla \theta_{i,k}}^2}{(1+\theta_{i,k})^{\beta+1}}\bigg)\leq C.
	\end{array}
\end{equation}
\end{lem}
\begin{proof}
For that observe the following identity:
\begin{align}
	\label{eq:observ}
	\phi(b)-\phi(a)\leq \phi'(b)(b-a)
\end{align}
Now for $i\in\{1,2\}$ we test \eqref{ELagmin2}--\eqref{ELagmin2*} by $\zeta_{i}=\phi'(\theta_{i,k+1})$, with $\phi\in C^2((0,\infty);(0,\infty))$ and $\phi',\phi''\geq 0$. This implies  the forllowing for $\theta_{2}$ by using \eqref{bnddetnablaPsi}:
\begin{align*}
	&\frac{c_2}{\tau}\bigg(\int_{\Omega^{2}_{k+1}}\phi(\theta_{2,k+1})-\int_{\Omega_{2,k}}\phi(\theta_{2,k})\bigg)
	= 
	c_2\int_{\Omega^{2}_{k}}\frac{\phi(\theta_{2,k+1}\circ\Psi_{u_{2,k+1}})\det(\nabla \Psi_{u_{2,k+1}})-\phi(\theta_{2,k})}{\tau}
	\\
	&\leq c_2\int_{\Omega^{2}_{k}}\frac{\phi(\theta_{2,k+1}\circ\Psi_{u_{2,k+1}})-\phi(\theta_{2,k})}{\tau} + C\tau\int_{\Omega^{2}_{k}}\phi(\theta_{2,k+1}\circ\Psi_{u_{2,k+1}}).
\end{align*}
Now this inequality, \eqref{eq:observ} and \eqref{ELagmin2*} implies that
\begin{align*}
	&\frac{c_2}{\tau}\bigg(\int_{\Omega^{i}_{k+1}}\phi(\theta_{2,k+1})-\int_{\Omega^2_k}\phi(\theta_{2,k})\bigg) 
	-C\tau\int_{\Omega^{2}_{k}}\phi(\theta_{2,k+1}\circ\Psi_{u_{2,k+1}})
	\\
	&\leq c_2\int_{\Omega^{2}_{k}}\frac{\phi'(\theta_{2,k+1}\circ\Psi_{u_{2,k+1}})(\theta_{2,k+1}\circ\Psi_{u_{2,k+1}}-\theta_{2,k})}{\tau} 
	\\
	&\leq  -k_2\int_{\Omega^{i}_{k+1}} \phi''(\theta_{2,k+1})\abs{\nabla \theta_{2,k+1}}^2 + \int_{\Omega^{2}_{k}}\mu^{M}_{i}(\theta_{2,k})|Du_{2,k+1}|^{2}\phi'(\theta_{2,k+1}\circ\Psi_{u_{2,k+1}})\\
	&\qquad\displaystyle +{\kappa}\int_{\Omega^{2}_{k}}\min\Big\{|\nabla^{k_{0}}u_{2,k+1}|^{2},\tau^{-1}\}\phi'(\theta_{2,k+1}\circ\Psi_{u_{2,k+1}})
	+\frac{1}{2h}\int_{\Omega^{2}_{0}}\left|u_{2,k+1}\circ\Phi^{2}_{k}-w_{2,k}\right|^{2}\phi'(\theta_{2,k+1}\circ\Phi^{2}_{k+1})\\
	&\qquad\displaystyle +\lambda\int_{\partial\Omega^{2}_{k+1}}\bigg(\phi(\theta_{1,k+1})-\phi(\theta_{2,k+1})\bigg)
\end{align*}
Now this inequality combined with the fact that $\phi'\in L^\infty((0,\infty),(0,\infty)),$ \eqref{consmin1*} and further summing over $0\leq k\leq N-1$ renders
\begin{equation}\label{aftusel1est}
	\begin{array}{ll}
		&\displaystyle{c_1}\int_{\Omega^{2}_{N-1}}\phi(\theta_{2,N-1})+ \sum^{N-1}_{k=0}k_2\tau\int_{\Omega^{i}_{k+1}} \phi''(\theta_{i}^{k+1})\abs{\nabla \theta_{2,k+1}}^2\\
		&\displaystyle \leq C+C\sum^{N-1}_{k=0}\tau^{2}\int_{\Omega^{2}_{k}}\phi(\theta_{2,k+1}\circ\Psi_{u_{2,k+1}})+\sum^{N-1}_{k=0}\lambda\tau\int_{\partial\Omega^{2}_{k+1}}\bigg(\phi(\theta_{1,k+1})-\phi(\theta_{2,k+1})\bigg),
	\end{array}
\end{equation}
where the generic constant $C$ may depend on the initial data, $h,$ $\beta$ and $w_{i}.$\\
In the same spirit the testing of \eqref{ELagmin2} infers:
\begin{equation}\label{aftrtest1steq}
	\begin{array}{ll}
		&\displaystyle {c_1}\int_{\Omega^{1}_{N-1}}\phi(\theta_{1,N-1})+ \sum^{N-1}_{k=0}k_1\tau\int_{\Omega^{1}_{k+1}} \phi''(\theta_{1,k+1})\abs{\nabla \theta_{1,k+1}}^2\\
		&\displaystyle \leq C+C\sum^{N-1}_{k=0}\tau^{2}\int_{\Omega^{1}_{k}}\phi(\theta_{1,k+1}\circ\Psi_{u_{1,k+1}})-\sum^{N-1}_{k=0}\lambda\tau\int_{\partial\Omega^{2}_{k+1}}\bigg(\phi(\theta_{1,k+1})-\phi(\theta_{1,k+1})\bigg).
	\end{array}
\end{equation}
Adding \eqref{aftusel1est} and \eqref{aftrtest1steq} and further applying discrete Gronwall inequality one obtains \eqref{estL2W12tg}.
\end{proof}

	Next using the so far constructed discrete in time solutions we are going to define suitable interpolants ($\tau-$layer) and pass limit $\tau\rightarrow 0$ in order to prove \eqref{discretemomen1}-\eqref{entropybal}.
	\subsection{Interpolants and limits}\label{Intlim}
	So far we have constructed $(u_{i,k+1},\eta_{k+1},\theta_{i,k+1})^{N-1}_{k=0}$ at discrete time points. Using them let us now define the following piecewise constant and piecewise affine interpolants and investigate their convergence properties.
	\subsubsection{Definition of interpolants, $\tau-$ independent bounds and weak limits}\label{interpole}	
	\textit{Piecewise constant interpolants:}
	\begin{equation}\label{pcivelo}
		\left\{\begin{array}{lll}
			&\displaystyle u_{i,\tau}(x,t)=u_{i,k+1}(x)\,\,&\displaystyle\mbox{for}\,\,t\in[\tau k,\tau(k+1))\cap[0,h],\,\,x\in\Omega_{k},\,\,k\in\{0,1,...,N-1,..\},\\
			&\displaystyle u_{i,\tau}(x,t)=\frac{\eta_{k+1}-\eta_{k}}{\tau}\circ(\eta_{k})^{-1}\,\,&\displaystyle\mbox{for}\,\,t\in[\tau k,\tau(k+1))\cap[0,h],\,\,x\in\partial\Omega^{2}_{k}.
		\end{array}\right.
	\end{equation}
	The interpolants corresponding to $\theta_{i}$ and $\eta$ will be defined on a interval slightly bigger than $(0,h),$more precisely we will define them on $(0,h+\tau).$
	\begin{equation}\label{pcitheta}
		\left\{ \begin{array}{lll}
			&\displaystyle\theta_{i,\tau}(x,t)=\theta_{i,k}\,\,&\displaystyle\mbox{for}\,\,t\in[\tau k,\tau(k+1))\cap[0,h+\tau],\,\,x\in\Omega^{i}_{k},\,\,k\in\{1,...,N-1,...\},\\
			&\displaystyle \theta_{i,\tau}(x,t)=\theta^{0}_{i}\,\,&\displaystyle\mbox{for}\,\,0\leqslant t<\tau,\,\,x\in\Omega_{0},\\
			&\displaystyle \theta_{i,\tau}^{+}(x,t)=\theta(x,t+\tau)\,\,&\displaystyle \mbox{for}\,\,t\in[\tau k,\tau(k+1)),\,\,x\in\Omega^{i}_{k+1},\,\,k\in\{1,...,N-1,...\}.
		\end{array}\right.
	\end{equation}
	Further
	\begin{equation}\label{pcieta}
		\left\{ \begin{array}{lll}
			&\displaystyle \eta_{\tau}(y,t)=\eta_{k}\,\,&\displaystyle\mbox{for}\,\,t\in[\tau k,\tau(k+1))\cap[0,h+\tau],\,\,y\in\Gamma,\,\,k\in\{1,...,N-1,...\},\\
			&\displaystyle \eta_{\tau}^{+}(y,t)=\eta_{k+1}\,\,&\displaystyle\mbox{for}\,\,t\in[\tau k,\tau(k+1)),\,\,y\in\Gamma,\\
			&\displaystyle \eta_{\tau}(y,t)=\eta^{0}\,\,&\displaystyle\mbox{for}\,\,0\leqslant t<\tau,\,\,y\in\Gamma
		\end{array}\right.
	\end{equation}
	and hence $$\eta_{\tau}^{+}(y,t)=\eta_{\tau}(y,t+\tau),\,\,\mbox{for any}\,\,t\in(0,h),\,\,y\in\Gamma.$$
	We also define the interpolant approximating the time dependent domain and the diffeomorphism connecting the initial domain with the one at a time $t\in(0,h)$ as follows:
	\begin{equation}\label{domaindef}
		\begin{array}{lll}
			&\displaystyle\Omega^{i}_{\tau}(t)=\Omega^{i}_{k}\,\,&\mbox{for}\,\, t\in[\tau k,\tau(k+1))\cap [0,h+\tau],\,\,k\in\{0,1...,N-1,...\}\\
			&\displaystyle\Omega^{i,+}_{\tau}(t)=\Omega^{i}_{k+1}\,\,&\mbox{for}\,\, t\in[\tau k,\tau(k+1)),\\
			&\displaystyle \Phi^{i}_{\tau}(x,t)=\Phi^{i}_{k}(x)\,\,&\mbox{for}\,\, t\in[\tau k,\tau(k+1))\cap[0,h+\tau],\,\,k\in\{0,1...,N-1,...\},\,\,x\in\Omega^{i}_{0},\\
			&\displaystyle \Phi^{i,+}_{\tau}(x,t)=\Phi^{i}_{k+1}(x)\,\,&\mbox{for}\,\, t\in[\tau k,\tau(k+1)),\,\,,x\in\Omega^{i}_{0},\\
			&\displaystyle \Psi^{i}_{\tau}(x,t)=\Psi_{u_{i},k+1}(x)\,\,&\mbox{for}\,\,t\in[\tau k,\tau(k+1))\cap[0,h+\tau],\,\,k\in\{0,1...,N-1,...\},\,\,x\in\Omega^{i}_{k}.
		\end{array}
	\end{equation}
	We further define some piecewise affine interpolants associated with the temperature and the structural displacement.\\
	\textit{Piecewise affine interpolants:}\\
	\begin{equation}\label{paieta}
		\begin{array}{ll}
			\displaystyle\widetilde{\eta}_{\tau}(y,t)=\frac{\tau(k+1)-t}{\tau}\eta_{k}(y)+\frac{t-\tau k}{\tau}\eta_{k+1}(y)\,\,\mbox{for}\,\,&\displaystyle t\in[\tau k,\tau(k+1))\cap[0,h],\\
			&\displaystyle\,\,y\in\Gamma,\,\,k\in\{0,...,N-1,...\}
		\end{array}
	\end{equation}	
	where $\eta_{0}(y)=\eta^{0}(y)$ for all $y\in\Gamma.$
	\begin{equation}
		\begin{array}{ll}
			\displaystyle \widetilde{\theta}_{i,\tau}(x,t)=\frac{\tau(k+1)-t}{\tau}\theta_{i,k}(x)+\frac{t-\tau k}{\tau}\theta_{i,k+1}\circ\Psi_{u_{i,k+1}}(x)\,\,\mbox{for}\,\,&\displaystyle t\in[\tau k,\tau(k+1))\cap[0,h],\\
			&\displaystyle x\in\Omega_{k},\,\,k\in\{0,...,N-1,...\}
		\end{array}
	\end{equation}
	where $\theta_{i,0}(x)=\theta^{0}_{i}(x)$ for all $x\in \Omega^{i}_{0}.$\\
	
	With the interpolants defined so far, we present the following corollary which can be proved by using Corollary \ref{Coretaukp1} and Lemma \ref{estthetakp1}.
	\begin{corollary}\label{estimatesinterpole1}
		Under the assumptions of Corollary \ref{Coretaukp1} and Lemma \ref{estthetakp1} we have 
		\begin{subequations}
			\begin{alignat}{2}
				&\displaystyle{(a)}{\mbox{A dissipation type estimate:}}\nonumber\\
				&\displaystyle K_{\kappa}(\eta^{+}_{\tau}(t_{1}))+{\kappa}\sum\limits_{i}\int_{0}^{t_{1}}\int_{\Omega^{i}_{\tau}(t)}|\nabla^{k_{0}}u_{i,\tau}|^{2}+\sum\limits_{i}\int_{0}^{t_{1}}\int_{\Omega^{i}_{\tau}(t)}\mu^{M}_{i}(\theta_{i,\tau})|D u_{i,\tau}|^{2}\nonumber\\
				&\displaystyle+\int_{0}^{t_{1}}\frac{1}{2h}\int_{\Gamma}|\partial_{t}\widetilde{\eta}_{\tau}|^{2} +\sum\limits_{i}\int_{0}^{t_{1}}\frac{1}{2h}\int_{\Omega^{i}_{\tau}}|u_{i,\tau}|^{2}(\det\nabla\Phi_{\tau})^{-1}+\kappa\int^{t_{1}}_{0}\int_{\Gamma}|\nabla^{k_{0}+1}\partial_{t}\widetilde{\eta}_{\tau}|^{2}\nonumber\\
				&\displaystyle + \int^{t_{1}}_{0}\tau\bigg(\kappa 1\bigg(\nabla^{k_{0}+1}\partial_{t}\widetilde{\eta}_{\tau},\nabla^{k_{0}+1}\partial_{t}\widetilde{\eta}_{\tau}\bigg)+ D^{2}K(\xi_{k\tau,(k+1)\tau})\bigg(\partial_{t}\widetilde{\eta}_{\tau},\partial_{t}\widetilde{\eta}_{\tau}\bigg)\bigg)\label{L1bndpenultimate}\\
				&\displaystyle+\frac{1}{2h}\int^{t_{1}}_{0}\int_{\Gamma}\bigg|\partial_{t}\widetilde{\eta}_{\tau}-\beta_{\tau}\bigg|^{2}+\frac{1}{2h}\sum_{i}\int^{t_{1}}_{0}\int_{\Omega^{i}_{\tau}}|u_{i,\tau}-w_{i,\tau}\circ(\Phi^{i}_{\tau})^{-1}|^{2}(\det\nabla\Phi^{i}_{\tau})^{-1}\label{L1bndlastline}\\
				&\displaystyle =K_{\kappa}(\eta^{0})+\frac{1}{2h}\int^{t_{1}}_{0}\bigg(\int_{\Gamma}|\beta_{\tau}|^{2}+\sum_{i}\int_{\Omega^{i}_{\tau}}|w_{i,\tau}\circ(\Phi^{i}_{\tau})^{-1}|(\det\nabla\Phi^{i}_{\tau})^{-1}\bigg)\label{leveltauintereq}\\
				&\displaystyle\leqslant C\bigg(\eta^{0},\theta^{0}_{i},h,\|\beta\|_{L^{2}(\Gamma\times[0,h])},\|w_{i}\|_{L^{2}(\Omega^{i}_{0}\times[0,h])}\bigg),\,\,\mathrm{for\,\,all}\,\,t_{1}\in  [0,h],\label{fromCoret}
			\end{alignat}
		\end{subequations}
		\begin{equation}\label{fromCoret2}
			\begin{array}{ll}
				(b)&\displaystyle \sum\limits_{i}\int_{0}^{t_{1}}\|D u_{i,\tau}\|^{2}_{L^{\infty}(\Omega^{i}_{\tau})}\leqslant C\bigg(\eta^{0},\theta^{0}_{i},h,\|\beta\|_{L^{2}(\Gamma\times[0,h])},\|w_{i}\|_{L^{2}(\Omega^{i}_{0}\times[0,h])}\bigg),\\
				&\displaystyle\mathrm{for\,\,all}\,\,t_{1}\in  [0,h]
			\end{array}
		\end{equation}
		and
		\begin{subequations}
			\begin{alignat}{2}
				& \displaystyle (c) \sum^{2}_{i=1}\int_{\Omega^{i}_{\tau}}\theta_{i,\tau}(\cdot,t_{1})+\int^{t}_{0}\int_{\Omega^{i,+}_{\tau}}|\theta_{i,\tau}|^{p}+\int^{t_{1}}_{0}\int_{\Omega^{i,+}_{\tau}}{|\nabla\theta_{i,\tau}|^{q}}\bigg)\nonumber\\
				&\displaystyle \qquad\leq C\bigg(\eta^{0},\theta^{0}_{i},h,\|\beta\|_{L^{2}(\Gamma\times[0,h])},\|w_{i}\|_{L^{2}(\Omega^{i}_{0}\times[0,h])}\bigg),	\label{fromesttheta}
			\end{alignat}
			for all $t_{1}\in[0,h],$ $\beta>1$ where $p\in[1,\frac{5}{3})$ and $q\in [1,\frac{5}{4}).$\\
			\begin{equation}\label{positivitytheta}
				\begin{array}{l}
					\displaystyle (d)\,\,	\theta_{i,\tau}\geqslant \gamma_{0}>0\,\,\mbox{a.e. in}\,\,\Omega_{\tau}\times[0,h].
				\end{array}
			\end{equation}
			
		\end{subequations}
			
		\end{corollary}
		\begin{proof}
			Using the definition of the interpolants the proof of item $(a)$ and $(b)$ for $t_{1}\in\tau\mathbb{N}\cap[0,h]$ follow respectively from \eqref{consmin1*} and \eqref{LinftyDuikp1}. For any $t_{1}\in[0,h],$ all the estimates still hold since the interpolants are constants in the intervals $[\tau k,\tau(k+1))\cap[0,h].$\\
			Now considering in particular $\beta>1$ in \eqref{estL2W12tp}, the bound of $\int_{\Omega^{i}_{\tau}}\theta_{i,\tau}(\cdot,t_{1})$ in \eqref{fromesttheta} follows directly. Now considering $\beta\in(0,1)$ in \eqref{estL2W12tp} one has the follwoing: 
			\begin{equation}\label{gradthetaibnd2}
				\begin{array}{ll}
					\displaystyle	\bigg\|\nabla (1+\theta_{i,\tau})^{\frac{1-\beta}{2}}\bigg\|^{2}_{L^{2}(\Omega^{i}_{\tau}\times (0,h))}\leqslant C.
				\end{array}
			\end{equation}
			Consequently
			\begin{equation}\label{thetaibndH12}
				\begin{array}{ll}
					\displaystyle\bigg\| (1+\theta_{i,\tau})^{\frac{1-\beta}{2}}\bigg\|^{2}_{L^{2}(0,h;W^{1,2}(\Omega^{i}_{\tau}))}\leqslant C\Rightarrow \bigg\| (1+\theta_{i,\tau})^{\frac{1-\beta}{2}}\bigg\|^{2}_{L^{2}(0,h;L^{6}(\Omega^{i}_{\tau}))}\leqslant C\,\,\mbox{for}\,\,\beta\in (0,1),
				\end{array}
			\end{equation}
			$i.e.$ one has $\theta_{i,\tau}\in L^{1-\beta}(0,h;L^{3(1-\beta)}(\Omega^{i}_{\tau}))$
			One can now use standard interpolation of Lebesgue spaces to furnish that
			\begin{equation}\label{Lpp2}
				\begin{array}{ll}
					\displaystyle\|\theta_{i,\tau}\|_{L^{p}(0,h;L^{p}(\Omega^{i}_{\tau}))}\leqslant C\,\,\mbox{for}\,\,p\in[1,\frac{5}{3}).
				\end{array}
			\end{equation}
			Further in view of  \eqref{Lpp2} one observes from
			\begin{equation}\label{gradthetas2}
				\begin{array}{ll}
					\displaystyle\int_{\Omega^{i}_{\tau}\times (0,h)}|\nabla\theta_{i,\tau}|^{s}&\displaystyle= \int_{\Omega^{i}_{\tau}\times (0,h)} |\nabla \theta_{i,\tau}|^{s}(1+\theta_{i,\tau})^{(1+\beta)\frac{s}{2}}(\theta_{i,\tau})^{-(1+\beta)\frac{s}{2}}\\
					&\displaystyle \leqslant\left(\int_{\Omega^{i}_{\tau}\times (0,h)}\frac{|\nabla\theta_{i,\tau}|^{2}}{(1+\theta)^{(1+\beta)}} \right)^{\frac{s}{2}}\left(\int_{\Omega^{i}_{\tau}\times (0,h)}(1+\theta_{i,\tau})^{(1+\beta)\frac{s}{2-s}}\right)^{\frac{2-s}{2}}
				\end{array}
			\end{equation}
			that the following inequality holds
			\begin{equation}\label{thetaW1s2}
				\begin{array}{l}
					\displaystyle\int_{\Omega^{i}_{\tau}\times(0,h)}\|\theta_{i,\tau}\|^{s}_{W^{1,s}(\Omega^{i}_{\tau})}\leqslant C\,\,\mbox{for all}\,\, s\in[1,\frac{5}{4}).
				\end{array}
			\end{equation}
			The uniform positive lower bound $(d)$ of the temperature is a consequence of Proposition \ref{propexminth}.
		\end{proof}

	Next we are going to list the weak type convergences of the interpolants as $\tau\rightarrow 0.$ Since $D^{2}K(\xi_{k\tau,(k+1)\tau})$ is bounded, for small enough values of $\tau,$ the term $\displaystyle\tau\int^{t_{1}}_{0}D^{2}K(\xi_{k\tau,(k+1)\tau})\bigg(\partial_{t}\widetilde{\eta}_{\tau},\partial_{t}\widetilde{\eta}_{\tau}\bigg)$ can be absorbed in the left hand side of \eqref{fromCoret} to conclude the following:
	\begin{equation}\label{weakconv}
		\begin{array}{lll}
			&\displaystyle \eta_{\tau}, \widetilde{\eta}_{\tau}\rightharpoonup^* \eta\,\,&\displaystyle\mbox{in}\,\, L^{\infty}(0,h;W^{k_{0}+1,2}(\Gamma))\,\,k_{0}>3,\\[2.mm]
			&\displaystyle \tilde{\eta}_{\tau}\rightarrow \eta\,\,&\displaystyle\mbox{in}\,\,C^{0}([0,h];C^{1,\alpha}(\Gamma))\,\,\mbox{for some}\,\,\alpha>0,\\[2.mm]
			&\displaystyle \partial_{t}\widetilde{\eta}_{\tau}\rightharpoonup \partial_{t}\eta\,\,&\displaystyle\mbox{in}\,\, L^{2}(0,h;W^{k_{0}+1,2}(\Gamma)),\\[2.mm]
			&\displaystyle \theta_{i,\tau}\rightharpoonup \theta_{i}\,\,&\displaystyle\mbox{in}\,\,L^{p}(0,h;L^{p}(\Omega^{i}_{\tau}))\,\,\mbox{for}\,\,p\in[1,\frac{5}{3}),\\[2.mm]
			&\displaystyle \theta_{i,\tau}\rightharpoonup^{*} \theta_{i}\,\,&\displaystyle \mbox{in}\,\,L^{\infty}(0,h;L^{1}(\Omega^{i}_{\tau})),\\[2.mm]
			&\displaystyle \theta_{i,\tau}\rightharpoonup \theta_{i}\,\,&\displaystyle \mbox{in}\,\,L^{s}(0,h;W^{1,s}(\Omega^{i}_{\tau}))\,\,\mbox{for}\,\,s\in[1,\frac{5}{4}),\\[2.mm]
			&\displaystyle u_{i,\tau}\rightharpoonup u_{i}\,\,&\displaystyle \mbox{in}\,\, L^{2}(0,h;W^{k_{0},2}(\Omega^{i}_{\tau})),\\[2.mm]
			&\displaystyle \Phi^{i}_{\tau}\rightarrow \Phi^{i}\,\,&\displaystyle \mbox{in}\,\, C^{0}([0,h];C^{1,\alpha}(\Omega^{i}_{0}))\,\,\mbox{for some}\,\,\alpha>0.
		\end{array}
	\end{equation}
	The fourth convergence above is a consequence of \eqref{fromesttheta} and \eqref{positivitytheta}.
	Further in the same spirit of \cite[Corollary 4.7]{BenKamSch}, one proves that
	$$\lim_{\tau\rightarrow 0}\det\nabla\Phi^{i}_{\tau}=1.$$
	\subsubsection{Equations solved by the interpolants}
	In this section we list the weak formulations of the equations solved by the interpolants defined in \eqref{interpole}. In the later stage we will pass limit $\tau\rightarrow 0$ in these weak formulations by using the convergences \eqref{weakconv}.\\[2.mm]
	For any pair of functions $(\psi,b)\in C^{0}([0,h];W^{k_{0},2}(\Omega))\times C^{0}([0,h];W^{k_{0}+1,2}(\Gamma))$ solving $\psi\circ\eta=b$ we consider test functions of the form $\int^{(k+1)\tau}_{k\tau}\psi$ and $\int^{(k+1)\tau}_{k\tau}b$ in \eqref{ELagmin1} and sum over $k=0,...,N-1$ to have the following 
	\begin{equation}\label{equetatau}
		\begin{array}{ll}
			&\displaystyle \int^{t_{1}}_{0}\langle DK_{\kappa}(\eta^{+}_{\tau}),b\rangle+\frac{1}{h}\int^{t_{1}}_{0}\int_{\Gamma}\left(\partial_{t}\widetilde{\eta}_{\tau}-\beta_{\tau}\right)b+\sum_{i}\frac{1}{h}\int^{t_{1}}_{0}\int_{\Omega^{i}_{\tau}}\left(u_{i,\tau}-w_{i,\tau}\circ(\Phi^{i}_{\tau})^{-1}\right)\psi(\det\nabla\Phi_{\tau})^{-1}\\
			&\displaystyle+\kappa\int^{t_{1}}_{0}\int_{\Gamma}\nabla^{k_{0}+1}\partial_{t}\widetilde{\eta}_{\tau}:\nabla^{k_{0}+1}b+{\kappa}\sum_{i}\int^{t_{1}}_{0}\int_{\Omega^{i}_{\tau}}\nabla^{k_{0}}u_{i,\tau}:\nabla^{k_{0}}\psi_{i}\\
			&\displaystyle+\sum_{i}\int^{t_{1}}_{0}\int_{\Omega^{i}_{\tau}}\mu^{M}_{i}(\theta_{i,\tau})D(u_{i,\tau}):D(\psi)=0
		\end{array}
	\end{equation}
	for a.e. $t_{1}\in[0,h]$ where $\psi$ is a divergence free test function in time-space.\\
	Similarly from \eqref{ELagmin2} and \eqref{ELagmin2*} we furnish the following  
	\begin{equation}\label{eqtheta1tau1}
		\begin{array}{ll}
			&\displaystyle c_{1}\int^{t_{1}}_{0}\int_{\Omega^{1}_{\tau}}\frac{(\theta_{1,\tau}(x+\tau u_{1,\tau},t+\tau)-\theta_{1,\tau}(x,t)}{\tau}\zeta_{1}(t+\tau)\circ\Psi^{1}_{\tau}+k_{1}\int^{t_{1}}_{0}\int_{\Omega^{1,+}_{\tau}}\nabla\theta^{+}_{1,\tau}\cdot\nabla\zeta_{1}\\
			&\displaystyle+\lambda\int^{t_{1}}_{0}\int_{\partial\Omega^{2,+}_{\tau}}(\theta^{+}_{1,\tau}-\theta^
			{+}_{2,\tau})\zeta_{1} -\int^{t_{1}}_{0}\int_{\Omega^{1}_{\tau}}\mu^{M}_{1}(\theta_{1,\tau})|Du_{1,\tau}|^{2}\zeta_{1}\circ\Psi^{1}_{\tau}\\
			&\displaystyle
			-{\kappa}\int^{t_{1}}_{0}\int_{\Omega^{1}_{\tau}}\min\bigg\{|\nabla^{k_{0}}u_{1,\tau}|^{2},\tau^{-1}\bigg\}\zeta_{1}\circ\Psi^{1}_{\tau}
			\displaystyle-\frac{1}{2h}\int^{t_{1}}_{0}\int_{\Gamma}\bigg|\partial_{t}\widetilde{\eta}_{\tau}-\beta_{\tau}\bigg|^{2}\zeta_{1}\\
			&\displaystyle-\frac{1}{2h}\int^{t_{1}}_{0}\int_{\Omega^{1}_{\tau}}|u_{1,\tau}-w_{1,\tau}\circ(\Phi^{1}_{\tau})^{-1}|^{2}\zeta_{1}\circ\Psi^{1}_{\tau}(\det\nabla\Phi^{1}_{\tau})^{-1}\\
			&\displaystyle-\kappa\int^{t_{1}}_{0}\int_{\Gamma}\min\bigg\{\bigg|\nabla^{k_{0}+1}\partial_{t}\widetilde{\eta}_{\tau}\bigg|^{2},\tau^{-1}\bigg\}\zeta_{1}=0,
		\end{array}
	\end{equation}
	and 
	\begin{equation}\label{eqtheta1tau2}
		\begin{array}{ll}
			&\displaystyle c_{2}\int^{t_{1}}_{0}\int_{\Omega^{2}_{\tau}}\frac{(\theta_{2}(x+\tau u_{2,\tau},t+\tau)-\theta_{2,\tau}(x,t)}{\tau}\zeta_{2}(t+\tau)\circ\Psi^{2}_{\tau}+k_{2}\int^{t_{1}}_{0}\int_{\Omega^{2,+}_{\tau}}\nabla\theta^{+}_{2,\tau}\cdot\nabla\zeta_{2}\\
			&\displaystyle-\lambda\int^{t_{1}}_{0}\int_{\partial\Omega^{2,+}_{\tau}}(\theta^{+}_{1,\tau}-\theta^{+}_{2,\tau})\zeta_{2} -\int^{t_{1}}_{0}\int_{\Omega^{2}_{\tau}}\mu^{M}_{2}(\theta_{2,\tau})|Du_{2,\tau}|^{2}\zeta_{2}\circ\Psi^{2}_{\tau}\\
			&\displaystyle
			-{\kappa}\int^{t_{1}}_{0}\int_{\Omega^{2}_{\tau}}\min\bigg\{|\nabla^{k_{0}}u_{2,\tau}|^{2},\tau^{-1}\bigg\}\zeta_{2}\circ\Psi^{2}_{\tau}\\
			&\displaystyle-\frac{1}{2h}\int^{t_{1}}_{0}\int_{\Omega^{2}_{\tau}}|u_{2,\tau}-w_{i,\tau}\circ(\Phi^{2}_{\tau})^{-1}|^{2}\zeta_{2}\circ\Psi^{2}_{\tau}(\det\nabla\Phi^{2}_{\tau})^{-1}=0
		\end{array}
	\end{equation}
	where $\zeta_{i}\in L^{\infty}(0,h;W^{1,p}(\Omega^{i}_{nbd})),$ for $p>3$ and $\zeta_{i}\geqslant \gamma$  for some neighborhood $\Omega^{i}_{nbd}$ containing $\Omega^{i}_{\tau}$ for all small enough $\tau>0.$\\  
	The next goal is to perform the limit passage $\tau\rightarrow 0$ in the discrete formulations \eqref{equetatau}, \eqref{eqtheta1tau1} and \eqref{eqtheta1tau2}. In this direction we would further obtain some strong convergences of the unknowns involved. For that the idea is to use variant of Aubin-Lions theorem away from the moving interface which requires estimating the time derivative of the concerning unknown. At this stage we will only be able to recover the strong convergences of $\theta_{i,\tau}$ and $\eta_{\tau}.$ The strong convergence of $\theta_{i,\tau}$ plays a crucial role in the passage of limit to \eqref{equetatau}, specifically to the nonlinear term $\displaystyle \int^{h}_{0}\int_{\Omega^{i}_{\tau}}\mu^{M}_{i}(\theta_{i,\tau})D(u_{i,\tau}):D(\psi_{i,\tau}).$
	\subsubsection{Strong convergence of $\theta_{i,\tau}$}\label{strongconvtheta}
	This section is devoted for the proof of 
	\begin{equation}\label{claimstrngconvuptobndry}
		\begin{array}{l}
			\displaystyle\theta_{i,\tau}\rightarrow \theta_{i}\,\,\mbox{in}\,\,L^{s}(\Omega^{i}_{\tau}\times(0,h)),\,\,\mbox{for any}\,s\in [1,\frac{5}{3})
		\end{array}
	\end{equation}
	and in that direction we estimate the time derivative of $\theta_{i,\tau}$ in a domain away from the moving interface using the equation \eqref{eqtheta1tau1}.\\
	First let us consider the equation \eqref{ELagmin2} solved by $\theta_{1,\tau}.$
	We consider a parabolic cylinder of the form $B\times J$ such that $2B\times 2J\Subset \Omega^{1}_{\tau}\times (0,h),$ for small enough $\tau,$ which is possible because of the uniform convergence \eqref{weakconv}$_{2}$. Now consider $\zeta_{1}\in C^{\infty}_{c}(B\times J).$\\ 
	Next one observes that
	\begin{equation}\label{writingtimeder}
		\begin{array}{ll}
			I&\displaystyle=\int^{h}_{0}\int_{\Omega^{1}_{\tau}}\frac{(\theta_{1,\tau}(x+\tau u_{1,\tau},t+\tau)-\theta_{1,\tau}(x,t)}{\tau}\zeta_{1}(t+\tau)\circ\Psi^{1}_{\tau}\\
			&\displaystyle =\int^{h}_{0}\int_{\Omega^{i,+}_{\tau}}\frac{\theta_{1,\tau}(x,t+\tau)-\theta_{1,\tau}(x-\tau u_{1,\tau},t)}{\tau}\zeta_{1}(\det\nabla\Psi^{1}_{\tau})^{-1}.\\
			&\displaystyle =\int_{J}\int_{B} \frac{\theta_{1,\tau}(x,t+\tau)-\theta_{1,\tau}(x,t)}{\tau}\zeta_{1}(\det\nabla\Psi^{1}_{\tau})^{-1}\\
			&\displaystyle\qquad+\int_{J}\int_{B}\frac{\theta_{1,\tau}(x,t)-\theta_{1,\tau}(x-\tau u_{1,\tau},t)}{\tau}\zeta_{1}(\det\nabla\Psi^{1}_{\tau})^{-1}=I_{1}+I_{2}.
		\end{array}
	\end{equation}
	The second summand of the final line of \eqref{writingtimeder} can be written as follows
	\begin{equation}\label{reqriteI2}
		\begin{array}{ll}
			I_{2}&\displaystyle = \int_{J}\int_{B}\frac{\theta_{1,\tau}(x,t)-\theta_{1,\tau}(x-\tau u_{1,\tau},t)}{\tau}\zeta_{1}(\det\nabla\Psi^{1}_{\tau})^{-1}\\
			&\displaystyle =\int_{J}\int_{B}\frac{\theta_{1,\tau}(x,t)}{\tau}\zeta_{1}(\det\nabla\Psi^{1}_{\tau})^{-1}-\int_{J}\int_{B}\frac{\theta_{1,\tau}(x,t)}{\tau}\zeta_{1}(x+\tau u_{1,\tau},t)\\
			&\displaystyle =\int_{J}\int_{B}\frac{\theta_{1,\tau}}{\tau}\zeta_{1}\left((\det\nabla\Psi^{1}_{\tau})^{-1}-1\right)+\int_{J}\int_{B}\frac{\theta_{1,\tau}}{\tau}\left(\zeta_{1}-\zeta_{1}(x+\tau u_{1,\tau},t)\right)\\[4.mm]
			&\displaystyle=I^{1}_{2}+I^{2}_{2}.
		\end{array}
	\end{equation}
	Since $(\det\nabla\Psi^{1}_{\tau})^{-1}=1+o(\tau),$ in view of \eqref{Lpp2}, $I^{1}_{2}$ admits of the following estimate
	$$|I^{1}_{2}|\leqslant C\|\zeta_{1}\|_{L^{\infty}(J;W^{1,p}(B))}.$$
	Whereas $I^{2}_{2}$ can be estimated as follows
	\begin{equation}\label{est2ndterm}
		\begin{array}{ll}
			\displaystyle |I_{2}^{2}|&\displaystyle \leqslant C\bigg|\int_J\int_{B}\frac{1}{\tau}\bigg(\int^{1}_{0}\partial_{s}\zeta_{1}(x+s\tau u_{1,\tau},t)ds\bigg)\theta_{1,\tau}\bigg|\\
			&\displaystyle \leqslant C\int^{1}_{0} \bigg|\int_{J}\int_{B}\nabla\zeta_{1}(x+s\tau u_{1,\tau},t)\cdot u_{1,\tau}\theta_{1,\tau}\bigg|ds\\[2.mm]
			&\displaystyle \leqslant C\|\theta_{1,\tau}\|_{L^{2}(J;L^{p^{*}}(B))}\|u_{1,\tau}\|_{L^{2}(J;L^{\infty}(B))}\|\nabla\zeta_{1}\|_{L^{\infty}(J;L^{p}(B))},\\[2.mm]
			&\displaystyle \leqslant C\bigg(\eta^{0},\theta^{0}_{i},h,\|\beta\|_{L^{2}(\Gamma\times[0,h])},\|w_{i}\|_{L^{2}(\Omega^{i}_{0}\times[0,h])}\bigg)\|\zeta_{1}\|_{L^{\infty}(J;W^{1,p}(B))},
		\end{array}
	\end{equation}
	for $p>3.$ In the above calculation to obtain the third line from the second we have used that $\theta_{1,\tau}$ is bounded in $L^{\infty}(J;L^{1}(B))\cap L^{\frac{5}{3}-}(B\times J)$ (we refer to \eqref{weakconv}$_{4}$ and \eqref{Lpp2}), which by interpolation implies that $\theta_{1,\tau}$ is bounded in particular in $L^{2}(J;L^{p^{*}}),$ for $p>3.$\\
	Hence we have that
	\begin{equation}\label{I2fnl}
		\begin{array}{l}
			\displaystyle|I_{2}|\leqslant C\|\zeta_{1}\|_{L^{\infty}(J;W^{1,p}(B))},
		\end{array}
	\end{equation}
	for $p$ sufficiently large.\\
	Using \eqref{writingtimeder}, $I_{1}$ can be estimated as $|I_{1}|\leqslant |I|+|I_{2}|.$ Since $|I_{2}|$ is already majorized in \eqref{I2fnl}, it only remains to estimate $|I|$ by using the equation \eqref{eqtheta1tau1}. We claim that
	\begin{equation}\label{estimateI}
		\begin{array}{ll}
			|I|\leqslant C\bigg(\eta^{0},\theta^{0}_{i},h,\|\beta\|_{L^{2}(\Gamma\times[0,h])},\|w_{i}\|_{L^{2}(\Omega^{i}_{0}\times[0,h])}\bigg)\|\zeta_{1}\|_{L^{\infty}(J;W^{2,p}(B))}.
		\end{array}
	\end{equation} 
	To prove our claim we estimate the second to eighth summands of \eqref{eqtheta1tau1}. While performing these estimates we will repeatedly use the bounds obtained in Corollary \ref{estimatesinterpole1} without mentioning them all the time.  For clarity of presentation we denote
	\begin{equation}\label{decompI}
		\displaystyle I=\int_{J}\int_{B}\frac{\theta_{1,\tau}(x,t+\tau)-\theta_{1,\tau}(x-\tau u_{1,\tau},t)}{\tau}\zeta_{1}(\det\nabla\Psi^{1}_{\tau})^{-1}=-\sum_{j=2}^{8}J_{i},
	\end{equation}
	where $J_{j}$ denotes respectively the second to eighth summands appearing in the left hand side of \eqref{eqtheta1tau1}.
	Now $|J_{2}|$ is dominated by $C\|\nabla\theta_{1,\tau}\|_{L^{1}(0,h;L^{\frac{5}{4}-}(B))}\|\nabla\zeta_{1}\|_{L^{\infty}(J;W^{1,p}(B))}$ and hence by the right hand side of \eqref{estimateI}. Next $|J_{3}|$ vanishes since $\zeta_{1}$ is compactly supported in $B\times J.$ The fourth and the fifth terms admits of the following estimates
	\begin{equation}\nonumber
		\begin{array}{ll}
			\displaystyle|J_{4}|&\displaystyle=\bigg|\int^{h}_{0}\int_{\Omega^{1}_{\tau}}\mu^{M}_{i}(\theta_{1,\tau})|Du_{1,\tau}|^{2}\zeta_{1}\circ\Psi^{1}_{\tau}\bigg|\\
			&\displaystyle\leqslant C\|\nabla u_{1,\tau}\|^{2}_{L^{2}(0,h;L^{2}(B))}\|\zeta_{1}\|_{L^{\infty}(B\times J)}\\
			&\displaystyle\leqslant C\bigg(\eta^{0},\theta^{0}_{i},h,\|\beta\|_{L^{2}(\Gamma\times[0,h])},\|w_{i}\|_{L^{2}(\Omega^{i}_{0}\times[0,h])}\bigg)\|\zeta_{1}\|_{L^{\infty}(J;W^{1,p}(B))}.
		\end{array}
	\end{equation}
	Using \eqref{fromCoret}, the fifth to eighth terms can be altogether estimated as follows
	\begin{equation}\nonumber
		\begin{array}{ll}
			\sum_{j=5}^{8} |J_{j}|\leqslant C\bigg(\eta^{0},\theta^{0}_{i},h,\|\beta\|_{L^{2}(\Gamma\times[0,h])},\|w_{i}\|_{L^{2}(\Omega^{i}_{0}\times[0,h])}\bigg)\|\zeta_{1}\|_{L^{\infty}(J;W^{1,p}(B))}.
		\end{array}
	\end{equation}
	This concludes the proof of \eqref{estimateI}. Since $|I_{1}|\leqslant |I|+|I_{2}|,$ \eqref{est2ndterm} and \eqref{estimateI} together furnishes 
	\begin{equation}\label{estI1}
		\begin{array}{l}
			\displaystyle	|I_{1}|\leqslant C\bigg(\eta^{0},\theta^{0}_{i},h,\|\beta\|_{L^{2}(\Gamma\times[0,h])},\|w_{i}\|_{L^{2}(\Omega^{i}_{0}\times[0,h])}\bigg)\|\zeta_{1}\|_{L^{\infty}(J;W^{2,p}_{0}(B))}.
		\end{array}
	\end{equation}
	Since $C^{\infty}_{c}(B\times J)$ is dense in the weak$^*$ topology of $L^{\infty}(J;W^{2,p}_{0}(B))$ one obtains
	\begin{equation}\label{dualbndtderthetatau}
		\begin{array}{l}
			\displaystyle \|\partial_{t}\theta_{1,\tau}\|_{L^{1}(J;W^{-2,p}(B))}\leqslant C\bigg(\eta^{0},\theta^{0}_{i},h,\|\beta\|_{L^{2}(\Gamma\times[0,h])},\|w_{i}\|_{L^{2}(\Omega^{i}_{0}\times[0,h])}\bigg).
		\end{array}
	\end{equation}
	Hence in view of \eqref{thetaW1s2}, the fact that $W^{1,\frac{5}{4}-}(B)$ is compactly embedded in $ L^{{s}}(B)$ for $s\in [1,\frac{15}{7})$ and the classical Aubin-Lions lemma one infers that
	\begin{equation}\label{strongconvthetatau}
		\begin{array}{l}
			\displaystyle\theta_{1,\tau}\rightarrow \theta_{1}\,\,\mbox{in}\,\,L^{s}(B\times(0,h)),\,\,\mbox{for any}\,s\in [1,\frac{15}{7}).
		\end{array}
	\end{equation} 
	From the arbitrariness of the cylinder $B\times J,$ \eqref{strongconvthetatau} implies the a.e. convergence of $\theta_{1,\tau}$ to $\theta_{1}$ in $\Omega^{1}_{\tau}.$ This a.e. convergence combined together with the $\tau$ independent bound of $\theta_{1,\tau}$ in $L^{s}(\Omega^{1}_{\tau}\times(0,h))$ for $s\in [1,\frac{5}{3})$ and Vitali convergence theorem renders \eqref{claimstrngconvuptobndry} for $i=1.$\\ 
	Similarly one can show that
	\begin{equation}\nonumber
		\begin{array}{l}
			\displaystyle\theta_{2,\tau}\rightarrow \theta_{2}\,\,\mbox{in}\,\,L^{s}(\Omega^{2}_{\tau}\times(0,h)),\,\,\mbox{for any}\,s\in [1,\frac{5}{3})
		\end{array}
	\end{equation}
	This completes the proof of \eqref{claimstrngconvuptobndry}.\\
	In view of the obtained convergences \eqref{weakconv} and \eqref{claimstrngconvuptobndry} we now pass $\tau\rightarrow 0$ in the approximate equations \eqref{equetatau}, \eqref{eqtheta1tau1} and \eqref{eqtheta1tau2}.
	\subsubsection{Construction of test functions and passage $\tau\rightarrow 0$ in \eqref{equetatau}}\label{passtau0}
	This section is devoted to pass limit $\tau\rightarrow 0$ in the approximate momentum equation \eqref{equetatau}.\\ 
	\underline{\textit{Construction of test functions:}} 
	We start by fixing a function $\psi\in C^{\infty}([0,T]\times \overline{\Omega})$ (we will denote by $\psi_{i},$ the restriction of $\psi$ in $\Omega^{i}$) such that $\div\,\psi=0$ and next define $b^{\tau}$ as follows:
	\begin{equation}\label{constestdeltalev2}
		b^{\tau}=\Tr_{\Sigma_{\eta^{\tau}}}\psi\cdot\nu,
	\end{equation}
	where the notion of trace $\Tr_{\Sigma_{\eta}}$ was introduced in Lemma \ref{Lem:TrOp}. Note that $\psi\in C^{\infty}([0,T]\times\overline{\Omega^{i}_{\tau}})$ for all $\tau>0$ sufficiently small.\\
	Now in view of the uniform in $\tau$ bounds of $\eta^{\tau}$ (we refer to \eqref{fromCoret}) one has the following convergences of $b^{\tau}$
	\begin{equation}\label{convbh2}
		\begin{alignedat}{2}
			b^{\tau}&\rightharpoonup^{*} b&&\mbox{ in } L^{\infty}(0,T;W^{k_{0}+1,2}(\Gamma))\cap W^{1,2}(0,T;W^{k_{0},2}(\Gamma)),\,\,\mbox{for}\,\,k_{0}>3,\\
			b&=\Tr_{\Sigma_{\eta}}&&\psi\cdot \nu.
		\end{alignedat}
	\end{equation}
	The convergence \eqref{convbh2}$_{1}$ further infers that
	\begin{equation}\label{strngconvbh2}
		b^{\tau}\rightarrow b\,\,\mbox{in}\,\, L^{\infty}(0,T;W^{2+,2}(\Gamma))
	\end{equation}
	by using the classical Aubin-Lions theorem.\\
	Precisely, one considers the pair of test functions $(\psi,b^{\tau})$ in \eqref{equetatau} and $(\psi,b)$ in the limiting momentum balance.\\
	\underline{\textit{Passing $\tau\rightarrow 0$ in \eqref{equetatau}}:} We obtain the following strong convergence (up to a non-relabeled subsequence) of $\eta_{\tau}$ as a consequence of \eqref{weakconv}$_{1}$ and \eqref{weakconv}$_{3}$ and the classical Aubin-Lions theorem
	\begin{equation}\label{strngconvetatau}
		\begin{array}{l}
			\eta_{\tau}\rightarrow \eta\,\,\mbox{in}\,\, L^{\infty}(0,T;W^{2,4}(\Gamma)).
		\end{array}
	\end{equation}
	Convergence \eqref{strngconvbh2}, \eqref{strngconvetatau} along with \eqref{weakconv} suffice to conclude that
	\begin{equation}\label{limpassageenenergyeta}
		\begin{array}{ll}
			&	\displaystyle \int^{t}_{0}\int_{\Gamma}\partial_{t}\widetilde{\eta}_{\tau}\partial_{t}b^{\tau}\rightarrow \int^{t}_{0}\int_{\Gamma}\partial_{t}\eta\partial_{t}b,\qquad \int^{t}_{0}\int_{\Gamma}\nabla^{k_{0}+1}\partial_{t}\widetilde{\eta}_{\tau}\partial_{t}b^{\tau}\rightarrow \int^{t}_{0}\int_{\Gamma}\nabla^{k_{0}+1}\partial_{t}\eta\partial_{t}b\\
			&\displaystyle	\int_0^t\langle DK_{\kappa}(\eta^{+}_{\tau}),b^{\tau}\rangle\rightarrow \int_0^t\langle D_{\kappa}(\eta),b\rangle,
		\end{array}
	\end{equation}
	for any $b\in C^{0}([0,T];W^{k_{0}+1,2}(\Gamma)),$ were we have used the structure of $DK_{\kappa}(\eta)$ (cf. \eqref{elasticityoperator},\eqref{amab},\eqref{Geta},\eqref{Gij'},\\ \eqref{exaR} and \eqref{DefDk}).\\
	Now let us consider the sixth term of \eqref{equetatau}. One readily has that $\mu^{M}_{i}(\theta_{i,\tau})\chi_{\Omega^{i}_{\tau}}\in L^{\infty}(\Omega\times (0,h))$ and it converges $a.e.$ to $\mu^{M}_{i}(\theta_{i})\chi_{\Omega^{i}}$ (which is a consequence of \eqref{strongconvthetatau}). Hence in particular (using Vitali convergence) $\mu^{M}_{i}\chi_{\Omega^{i}_{\tau}}$ converges strongly to $\mu^{M}_{i}(\theta_{i})\chi_{\Omega^{i}}$ in $L^{r}(\Omega\times(0,h))$ for any $1\leqslant r<\infty.$ Hence in view of the weak convergence \eqref{weakconv}$_{6}$ one has that
	\begin{equation}\label{conv6thterm}
		\begin{array}{l}
			\displaystyle 	\int^{t_{1}}_{0}\int_{\Omega^{i}_{\tau}}\mu^{M}_{i}(\theta_{i,\tau})D(u_{i,\tau}):D(\psi)\rightarrow \int^{t_{1}}_{0}\int_{\Omega^{i}}\mu^{M}_{i}(\theta_{i})D(u_{i}):D(\psi).
		\end{array}
	\end{equation}
	In view of \eqref{limpassageenenergyeta}, \eqref{conv6thterm} and the available weak convergences from \eqref{weakconv}, one can pass $\tau\rightarrow 0$ in \eqref{equetatau} to furnish the following:
	\begin{equation}\label{momentumafterpassingtau}
		\begin{array}{ll}
			&\displaystyle \int^{t_{1}}_{0}\langle DK_{\kappa}(\eta),b\rangle+\frac{1}{h}\int^{t_{1}}_{0}\int_{\Gamma}\left(\partial_{t}{\eta}-\beta\right)b+\sum_{i}\frac{1}{h}\int^{t_{1}}_{0}\int_{\Omega^{i}(t)}\left(u_{i}-w_{i}\circ(\Phi^{i})^{-1}\right)\psi\\
			&\displaystyle+{\kappa}\sum_{i}\int^{t_{1}}_{0}\int_{\Omega^{i}(t)}\nabla^{k_{0}}u_{i}:\nabla^{k_{0}}\psi+\sum_{i}\int^{t_{1}}_{0}\int_{\Omega^{i}(t)}\mu^{M}_{i}(\theta_{i})D(u_{i}):D(\psi)\\
			&\displaystyle +\kappa\int^{t_{1}}_{0}\int_{\Gamma}\nabla^{k_{0}+1}\partial_{t}\eta:\nabla^{k_{0}+1}b=0.
		\end{array}
	\end{equation}
	for a.e. $t_{1}\in [0,h]$ with $\partial_{t}\eta\nu=u_{i}(\eta,t)$ and $\psi\circ \eta=b$ on $\Gamma.$\\
	Now let us discuss the passage of $\tau\rightarrow 0$ to the approximated heat evolutions \eqref{eqtheta1tau1} and \eqref{eqtheta1tau2}.
	\subsubsection{Passage $\tau\rightarrow 0$ in  \eqref{eqtheta1tau1} and \eqref{eqtheta1tau2}}\label{passtau0heat} Let us now pass $\tau\rightarrow 0$ in \eqref{eqtheta1tau1} and \eqref{eqtheta1tau2}.\\
	We use globally defined test functions $\zeta\in C^{\infty}_{c}([0,h];C^{\infty}_{c}(\overline{\Omega})),$ $\zeta\geqslant0$ and restrict it to construct $\zeta_{i}=\zeta\suchthat_{\Omega^{i}_{\tau}}\in C^{\infty}_{c}([0,h];C^{\infty}_{c}(\overline{\Omega^{i}_{\tau}})).$ Further in the limiting equation we would use the same notation $\zeta_{i}$ to denote the restriction of $\zeta\in C^{\infty}_{c}([0,h];C^{\infty}_{c}(\overline{\Omega}))$ in $\overline{\Omega^{i}(t)}.$ Later using the fact that $C^{\infty}_{c}([0,h];C^{\infty}_{c}(\overline{\Omega^{i}(t)}))$ is dense in $C^{\infty}_{c}([0,h];W^{1,p}(\Omega^{i}(t)))$ (for $p>3,$ which is true since at this level the structure is regularized implying $\Omega^{i}(t)$ has a uniformly Lipschitz boundary for a.e. $t\in(0,h)$), we will conclude the weak formulation of the limiting heat equation (limit as $\tau$ approaches $0$) with test functions $\zeta_{i}\in C^{\infty}_{c}([0,h];W^{1,p}(\Omega^{i})).$\\
	In what follows we explain the limit passage in \eqref{eqtheta1tau1}. Similar arguments applies to \eqref{eqtheta1tau2}.\\
	Now in connection with the first term in the left hand side of \eqref{eqtheta1tau1} we compute the following
	\begin{equation}\label{rewritetimequottheta}
		\begin{array}{ll}
			& \displaystyle \int^{t_{1}}_{0}\int_{\Omega^{1}_{\tau}}\frac{(\theta_{1,\tau}(x+\tau u_{1,\tau},t+\tau)-\theta_{1,\tau}(x,t)}{\tau}(\zeta_{1}(t+\tau)\circ\Psi^{1}_{\tau})\,\\
			&\displaystyle =\int^{t_{1}+\tau}_{\tau}\int_{\Omega^{1}_{\tau}}\frac{\theta_{1,\tau}(x,t)}{\tau}\zeta_{1}(x,t)(\det\nabla\Psi^{1}_{\tau})^{-1}-\int^{t_{1}}_{0}\int_{\Omega^{1}_{\tau}}\frac{\theta_{1,\tau}(x,t)\zeta_{1}(x,t)}{\tau}\\
			&\,\,\displaystyle -\int^{t_{1}}_{0}\int_{\Omega^{1}_{\tau}}\theta_{1,\tau}(x,t)\frac{\zeta_{1}(x+\tau u_{1,\tau},t+\tau)-\zeta_{1}(x,t)}{\tau}\\
			&=\displaystyle -\frac{1}{\tau}\int^{\tau}_{0}\int_{\Omega^{1}_{0}}\theta_{1}^{0}(x)\zeta_{1}(x,t)+\frac{1}{\tau}\int^{t_{1}+\tau}_{t_{1}}\int_{\Omega^{1}_{\tau}}\theta_{1,\tau}(x,t)\zeta_{1}(x,t)\\
			&\,\,\displaystyle +\int^{t_{1}+\tau}_{\tau}\int_{\Omega_{\tau}^{1}}\frac{\theta_{1,\tau}(x,t)}{\tau}\zeta_{1}(x,t)\left((\det\nabla\Psi^{1}_{\tau})^{-1}-1\right)\\
			&\,\,\displaystyle-\int^{t_{1}}_{0}\int_{\Omega^{1}_{\tau}}\theta_{1,\tau}(x,t)\frac{\zeta_{1}(x+\tau u_{1,\tau},t+\tau)-\zeta_{1}(x,t)}{\tau}\\[2.mm]
			&\displaystyle=\sum_{i=1}^{4}I^{i}_{\theta}.	
		\end{array}
	\end{equation}
	for a.e. $t_{1}\in[0,h].$ \\ 
	Using Lebesgue differentiation theorem one obtains
	$$\lim_{\tau\rightarrow 0}I^{1}_{\theta}=-\int_{\Omega^{1}_{0}}\theta^{0}_{1}(x)\zeta_{1}(x,0).$$
	Next let us show that
	\begin{equation}\label{limI2thta}
		\begin{array}{l}
			\displaystyle	\lim_{\tau\rightarrow 0}I^{2}_{\theta}=\int_{\Omega^{1}(t_{1})}\theta_{1}(x,t_{1})\zeta_{1}(x,t_{1}).
		\end{array}
	\end{equation}
	In that direction we write
	\begin{equation}\label{decomI2}
		\begin{array}{l}
			\displaystyle \frac{1}{\tau}\int^{t_{1}+\tau}_{t_{1}}\int_{\Omega^{1}_{\tau}}\theta_{1,\tau}\zeta_{1}-\int_{\Omega^{1}(t_{1})}\theta_{1}\zeta_{1}=\frac{1}{\tau}\int^{t_{1}+\tau}_{t_{1}}\int_{\mathbb{R}^{3}}\bigg(\theta_{1,\tau}\mathcal{X}_{\Omega^{1}_{\tau}}-\theta_{i}\mathcal{X}_{\Omega^{1}(t_{1})}\bigg)\zeta_{1}.
		\end{array}
	\end{equation}
	Now note that
	\begin{equation}\label{showavfL1conv}
		\begin{array}{ll}
			&\displaystyle \int^{h}_{0}\left| \frac{1}{\tau}\int^{t_{1}+\tau}_{t_{1}}\int_{\mathbb{R}^{3}}\bigg(\theta_{1,\tau}\mathcal{X}_{\Omega^{1}_{\tau}}-\theta_{1}\mathcal{X}_{\Omega^{1}(t_{1})}\bigg)\zeta_{1} \right|\\[2.mm]
			&\displaystyle \leqslant \frac{1}{\tau}\bigg(\int^{h}_{0}\int^{t_{1}+\tau}_{0}\int_{\mathbb{R}^{3}}\left|\bigg(\theta_{1,\tau}\mathcal{X}_{\Omega^{1}_{\tau}}-\theta_{1}\mathcal{X}_{\Omega^{1}(t_{1})}\bigg)\zeta_{1}\right|-\int^{h}_{0}\int^{t_{1}}_{0}\int_{\mathbb{R}^{3}}\left|\bigg(\theta_{1,\tau}\mathcal{X}_{\Omega^{1}_{\tau}}-\theta_{1}\mathcal{X}_{\Omega^{1}(t_{1})}\bigg)\zeta_{1}\right|\bigg)\\[2.mm]
			&\displaystyle=\frac{1}{\tau}\bigg(\int^{h+\tau}_{\tau}\int^{t_{1}}_{0}\int_{\mathbb{R}^{3}}\left|\bigg(\theta_{1,\tau}\mathcal{X}_{\Omega^{1}_{\tau}}-\theta_{1}\mathcal{X}_{\Omega^{1}(t_{1})}\bigg)\zeta_{1}\right|+\int^{0}_{h}\int^{t_{1}}_{0}\int_{\mathbb{R}^{3}}\left|\bigg(\theta_{1,\tau}\mathcal{X}_{\Omega^{1}_{\tau}}-\theta_{1}\mathcal{X}_{\Omega^{1}(t_{1})}\bigg)\zeta_{1}\right|\bigg)\\[2.mm]
			&\displaystyle \leqslant \frac{1}{\tau}\int^{h+\tau}_{h}\int^{t_{1}}_{0}\int_{\mathbb{R}^{3}}\left|\bigg(\theta_{1,\tau}\mathcal{X}_{\Omega^{1}_{\tau}}-\theta_{1}\mathcal{X}_{\Omega^{1}(t_{1})}\bigg)\zeta_{1}\right|\\[2.mm]
			&\displaystyle \leqslant \left\|(\theta_{1,\tau}\mathcal{X}_{\Omega^{1}_{\tau}}-\theta_{1}\mathcal{X}_{\Omega^{1}(t_{1})})\zeta_{1}\right\|_{L^{1}(\mathbb{R}^{3}\times(0,h))}
		\end{array}
	\end{equation}
	In view of \eqref{claimstrngconvuptobndry} and \eqref{showavfL1conv} we conclude that the term appearing in \eqref{decomI2} converges to zero for $a.e.$ $t_{1}\in[0,h].$ This proves \eqref{limI2thta}.\\
	Further since $\det\nabla\Psi^{1}_{\tau}=1+o(\tau),$ one has the following concerning the integrand of $I^{3}_{\theta}$
	$$\frac{\theta_{1,\tau}(x,t)}{\tau}\zeta_{1}(x,t)\left((\det\nabla\Psi^{1}_{\tau})^{-1}-1\right)\chi_{\Omega^{1}_{\tau}\times(\tau,h)} \rightarrow 0,\,\,\mbox{a.e.\,as}\,\, \tau\rightarrow 0.$$
	This almost everywhere convergence along with \eqref{claimstrngconvuptobndry} and Vitali convergence theorem yields
	$$I^{3}_{\theta}\rightarrow 0,\,\,\mbox{as}\,\,\tau\rightarrow 0.$$
	Next since $\zeta_{1}$ is smooth both in time and space,
	$$\frac{\zeta_{1}(x+\tau u_{1,\tau},t+\tau)-\zeta_{1}(x,t)}{\tau}\rightarrow (\partial_{t}+u_{1}\cdot\nabla)\zeta_{1},\,\,\mbox{a.e.\,as}\,\tau\rightarrow 0.$$
	Further using mean value inequality and \eqref{claimstrngconvuptobndry} one has that the integrand corresponding to $I_{\theta}^{4}$ is bounded in $L^{r}(\Omega^{1}_{\tau}\times(0,h))$ for some $r>1$ by a constant $C(\|\zeta_{1}\|_{W^{1,\infty}},\|u_{1}\|_{L^{2}(W^{k_{0},2})})$ independent of $\tau.$ Hence once again applying Vitali convergence theorem infers
	$$\displaystyle I^{4}_{\theta}\rightarrow -\int^{t_{1}}_{0}\int_{\Omega^{1}}\theta_{i}(x,t)(\partial_{t}+u_{1}\cdot\nabla)\zeta_{1}\,\,\mbox{as}\,\,\tau\rightarrow 0.$$
	From the above convergences one concludes the following concerning the first term of \eqref{eqtheta1tau}:
	\begin{equation}\label{thetatimconv}
		\begin{array}{ll}
			&\displaystyle c_{1}\int^{h}_{0}\int_{\Omega^{1}_{\tau}}\frac{(\theta_{1,\tau}(x+\tau u_{1,\tau},t+\tau)-\theta_{1,\tau}(x,t)}{\tau}\zeta_{1}(t+\tau)\circ\Psi^{1}_{\tau}\\
			&\displaystyle\longrightarrow -c_{1}\int_{\Omega^{1}_{0}}\theta^{0}_{1}(x)\zeta_{1}(x,0)+c_{i}\int_{\Omega^{1}(t_{1})}\theta_{1}(x,t_{1})\zeta_{1}(x,t_{1}) -c_{i}\int^{h}_{0}\int_{\Omega^{1}}\theta_{1}(x,t)(\partial_{t}+u_{1}\cdot\nabla)\zeta_{1}
		\end{array}
	\end{equation}
	for any $\zeta_{1}$ (as the constructed at the beginning of this section) and for any $t_{1}\in[0,h].$\\
	The passage of limit $\tau\rightarrow 0$ in the second and the third terms of \eqref{eqtheta1tau1} is comparatively easier since the integrands are linear in $\theta_{1,\tau}.$ Indeed while performing this passage in the third term, one uses the weak convergence of the trace of $\theta_{1,\tau}$ (which follows from \eqref{weakconv}$_{6}$).\\
	Now considering the $\limsup_{\tau\rightarrow 0}$ of \eqref{eqtheta1tau1} infers
	\begin{equation}\label{aftrtau0heat1}
		\begin{array}{ll}
			&\displaystyle -c_{1}\int_{\Omega^{1}_{0}}\theta^{0}_{1}(x)\zeta_{1}(x,0)+c_{1}\int_{\Omega^{1}(t_{1})}\theta_{1}(x,t_{1})\zeta_{1}(x,t_{1})-c_{1}\int^{t_{1}}_{0}\int_{\Omega^{1}}\theta_{1}(x,t)(\partial_{t}+u_{1}\cdot\nabla)\zeta_{1}\\
			&\displaystyle+k_{1}\int^{t_{1}}_{0}\int_{\Omega^{1}}\nabla\theta_{1}\cdot\nabla\zeta_{1}+\lambda\int^{t_{1}}_{0}\int_{\partial\Omega^{2}}(\theta_{1}-\theta_{2})\zeta_{1}\\
			&\displaystyle+\limsup_{\tau\rightarrow 0}\bigg(-\int^{t_{1}}_{0}\int_{\Omega^{1}_{\tau}}\mu^{M}_{1}(\theta_{1,\tau})|Du_{1,\tau}|^{2}\zeta_{1}-{\kappa}\int^{t_{1}}_{0}\int_{\Omega^{1}_{\tau}}\min\bigg\{|\nabla^{k_{0}}u_{1,\tau}|^{2},\tau^{-1}\bigg\}\zeta_{1}\circ\Psi^{1}_{\tau}\\
			&\displaystyle-\frac{1}{2h}\int^{t_{1}}_{0}\int_{\Gamma}\bigg|\partial_{t}\widetilde{\eta}_{\tau}-\beta_{\tau}\bigg|^{2}\zeta_{1}-\frac{1}{2h}\int^{t_{1}}_{0}\int_{\Omega^{1}_{\tau}}|u_{1,\tau}-w_{1,\tau}\circ(\Phi^{1}_{\tau})^{-1}|^{2}\zeta_{1}\circ\Psi^{1}_{\tau}(\det\nabla\Phi^{1}_{\tau})^{-1}\\
			&\displaystyle -\kappa\int^{t_{1}}_{0}\int_{\Gamma}\min\bigg\{\bigg|\nabla^{k_{0}+1}\partial_{t}\widetilde{\eta}_{\tau}\bigg|^{2},\tau^{-1}\bigg\}\zeta_{1}\bigg)=0,
		\end{array}
	\end{equation}
	for any $\zeta_{1}\in C^{\infty}_{c}([0,h];C^{\infty}_{c}(\overline{\Omega^{1}_{\tau}})),$ and for $a.e.$ $t_{1}\in[0,h].$\\
	Similarly taking $\limsup_{\tau\rightarrow 0}$ into \eqref{eqtheta1tau2} furnishes:
	\begin{equation}\label{aftrtau0heat2*}
		\begin{array}{ll}
			&\displaystyle -c_{2}\int_{\Omega^{2}_{0}}\theta^{0}_{2}(x)\zeta_{2}(x,0)+c_{2}\int_{\Omega^{2}(t_{1})}\theta_{2}(x,t_{1})\zeta_{2}(x,t_{1})-c_{2}\int^{t_{1}}_{0}\int_{\Omega^{2}}\theta_{2}(x,t)(\partial_{t}+u_{2}\cdot\nabla)\zeta_{2}\\
			&\displaystyle+k_{2}\int^{t_{1}}_{0}\int_{\Omega^{2}}\nabla\theta_{2}\cdot\nabla\zeta_{2}-\lambda\int^{t_{1}}_{0}\int_{\partial\Omega^{2}}(\theta_{1}-\theta_{2})\zeta_{2}\\
			&\displaystyle+\limsup_{\tau\rightarrow 0}\bigg(-\int^{t_{1}}_{0}\int_{\Omega^{2}_{\tau}}\mu^{M}_{2}(\theta_{2,\tau})|Du_{2,\tau}|^{2}\zeta_{2}-{\kappa}\int^{t_{1}}_{0}\int_{\Omega^{2}_{\tau}}\min\bigg\{|\nabla^{k_{0}}u_{2,\tau}|^{2},\tau^{-1}\bigg\}\zeta_{2}\circ\Psi^{2}_{\tau}\\
			&\displaystyle-\frac{1}{2h}\int^{t_{1}}_{0}\int_{\Omega^{2}_{\tau}}|u_{2,\tau}-w_{2,\tau}\circ(\Phi^{2}_{\tau})^{-1}|^{2}\zeta_{2}\circ\Psi^{2}_{\tau}(\det\nabla\Phi^{2}_{\tau})^{-1}=0,
		\end{array}
	\end{equation}	
	for any $\zeta_{2}\in C^{\infty}_{c}([0,h];C^{\infty}_{c}(\overline{\Omega^{2}_{\tau}})),$ and for $a.e.$ $t_{1}\in[0,h].$\\
	Hence in particular using the test functions $\zeta_{i}\in C^{\infty}_{c}([0,h];C^{\infty}_{c}(\overline{\Omega^{i}_{\tau}}))$ with $\zeta_{i}\geqslant 0,$ we obtain the following formulation of the heat evolution from \eqref{aftrtau0heat1} and \eqref{aftrtau0heat2*}  
	\begin{equation}\label{aftrtau0heat2}
		\begin{array}{ll}
			&\displaystyle c_{i}\int^{t_{1}}_{0}\frac{d}{dt}\int_{\Omega^{i}(t)}\theta_{i}(x,t)\zeta_{i}(x,t)-c_{i}\int^{t_{1}}_{0}\int_{\Omega^{i}(t)}\theta_{i}(x,t)(\partial_{t}+u_{i}\cdot\nabla)\zeta_{i}+k_{i}\int^{t_{1}}_{0}\int_{\Omega^{i}(t)}\nabla\theta_{i}\cdot\nabla\zeta_{i}\\
			&\displaystyle+\lambda\mathcal{I}_{i}\int^{t_{1}}_{0}\int_{\partial\Omega^{2}(t)}(\theta_{1}-\theta_{2})\zeta_{i}
			\displaystyle  =\liminf_{\tau\rightarrow 0}\left\langle \mathcal{D}^{\tau,i},\zeta_{i}\right\rangle\\
		\end{array}
	\end{equation}
	for $a.e.$ $t_{1}\in [0,h],$ where the duality pairing $\left\langle \mathcal{D}^{\tau,i},\zeta_{i}\right\rangle$ is given as
	\begin{alignat}{2}
		\displaystyle \left\langle \mathcal{D}^{\tau,i},\zeta_{i}\right\rangle=&\displaystyle \int^{t_{1}}_{0}\int_{\Omega^{i}_{\tau}}\mu^{M}_{i}(\theta_{i,\tau})|Du_{i,\tau}|^{2}\zeta_{i}+{\kappa}\int^{t_{1}}_{0}\int_{\Omega^{i}_{\tau}}\min\bigg\{|\nabla^{k_{0}}u_{i,\tau}|^{2},\tau^{-1}\bigg\}\zeta_{i}\circ\Psi^{i}_{\tau}\nonumber
		\displaystyle\\
		&\displaystyle+\mathcal{X}_{i}\frac{1}{2h}\int^{t_{1}}_{0}\int_{\Gamma}\bigg|\partial_{t}\widetilde{\eta}_{\tau}-\beta_{\tau}\bigg|^{2}\zeta_{i}
		\displaystyle+\frac{1}{2h}\int^{t_{1}}_{0}\int_{\Omega^{i}_{\tau}}|u_{i,\tau}-w_{i,\tau}\circ(\Phi^{i}_{\tau})^{-1}|^{2}\zeta_{i}\circ\Psi^{i}_{\tau}(\det\nabla\Phi^{i}_{\tau})^{-1}\label{defH2}\\
		&\displaystyle +\kappa\mathcal{X}_{i}\int^{t_{1}}_{0}\int_{\Gamma}\min\bigg\{\bigg|\nabla^{k_{0}+1}\partial_{t}\widetilde{\eta}_{\tau}\bigg|^{2},\tau^{-1}\bigg\}\zeta_{i},\nonumber
	\end{alignat}
	where $\mathcal{I}_{1}=1,$ $\mathcal{I}_{2}=-1$ and $\mathcal{X}_{1}=1$ and $\mathcal{X}_{2}=0.$ Note that the notations $\mathcal{I}_{i}$ and $\mathcal{X}_{i}$ are introduced only to write the two heat equations solved by $\theta_{1}$ and $\theta_{2}$ in a unified manner.\\
	Now an interesting question is: \textit{`whether the limit (as $\tau\rightarrow 0$) of  $\mathcal{D}^{\tau,i}$ is merely a measure or it does exist as a $L^{1}$ function'}? In fact next we are going to prove that the limit of $\mathcal{D}^{\tau,i}$ exist and is integrable. As a corollary we are further going to recover the desired \textit{`energy balance'} solved by the limiting unknowns $\eta,$ $u_{i}$ and $\theta_{i}.$\\[4.mm]
	\underline{\textit{Identification of $\lim_{\tau\rightarrow 0}\left\langle \mathcal{D}^{\tau,i},\zeta_{i}\right\rangle$ upto a subsequence and energy balance:}} 
	The first step focuses in suitable testing of the approximating and the limiting momentum equation to identify strong limits of the dissipative terms. We wish to test \eqref{ELagmin1} by $\displaystyle-\bigg(\frac{\eta_{k+1}-\eta_{k}}{\tau},u_{k+1}\bigg)(h-\tau(k+1)),$ and sum over the indexing set $\{0,...,N-1\}$. Since $h-(k+1)\tau=\tau(N-k-1),$ one first observes that
	\begin{equation}\nonumber
		\begin{array}{ll}
			\langle DK_\kappa(\eta_{k+1}), (N-k-1)(\eta_{k+1}-\eta_k)\rangle &\displaystyle= \langle DK(\eta_{k+1}), (N-k-1)(\eta_{k+1}-\eta_k)\rangle \\[2.mm]
			&\displaystyle\quad+\kappa \langle \nabla^{k_0+1}\eta_{k+1}, (N-k-1)\nabla^{k_{0}+1}(\eta_{k+1}-\eta_k)\rangle.
		\end{array}
	\end{equation}
	Now in view of the following identity:
	\begin{align*}
		\sum_{k=0}^{N-1} (a_{k+1}-a_k) a_{k+1}(N-k-1)
		= - (N-1)\frac{a_{0}^2}{2} + \sum_{k=1}^{N-1}\frac{a_{k}^2}{2} + \sum_{k=0}^{N-1} (N-k-1)  \frac{(a_{k+1}-a_k)^2}{2}.
	\end{align*}
	one has the following (using $a_k\sim \nabla^{k_0+1}\eta_{k}$), 
	\begin{align*}
		&\sum_{k=0}^{N-1}\langle \nabla^{k_0+1}\eta_{k+1}, (N-k-1)\nabla^{k_{0}+1}(\eta_{k+1}-\eta_k)\rangle
		\\
		&\quad = - (N-1)\frac{\norm{\nabla^{k_0+1}\eta_{0}}^2_{L^{2}}}{2} + \sum_{k=1}^{N-1}\frac{\norm{\nabla^{k_0+1}\eta_{k}}^2_{L^{2}}}{2} + \sum_{k=0}^{N-1} (N-k-1)  \frac{\norm{\nabla^{k_0+1}\eta_{k+1}-\nabla^{k_0+1}\eta_{k}}^2_{L^{2}}}{2}
		\\
		&=: \frac{1}{\tau}\bigg(-(h-\tau)\frac{\norm{\nabla^{k_0+1}\eta_{0}}^2_{L^{2}}}{2} + \int_0^h\frac{\norm{\nabla^{k_0+1}\eta_\tau}^2_{L^{2}}}{2} +\frac{\mathcal{D}_\tau}{\kappa}\bigg),
	\end{align*}
	where we used the interpolant notations and introducing $\mathcal{D}_\tau$ as the dissipation related to the implicit time-stepping. Now, setting $t_\tau= (k+1)\tau$ on $[k\tau,(k+1)\tau)$ we find that the aforementioned testing implies
	\begin{equation}\label{testaproxTdiff}
		\begin{array}{ll}
			&\displaystyle	\frac{\kappa}{2}\int_{0}^{h}\int_{\Gamma} |\nabla^{k_{0}+1}\eta_{\tau}|^{2} +\frac{1}{h}\int^{h}_{0}\int_{\Gamma}|\partial_{t}\widetilde{\eta}_{\tau}|^{2}(h-t_\tau)+\kappa\int^{h}_{0}\int_{\Gamma} |\nabla^{k_{0}+1}\partial_{t}\widetilde{\eta}_{\tau}|(h-t_\tau)\\
			&\displaystyle+\frac{1}{h}\sum_{i}\int^{h}_{0}\int_{\Omega^{i}_{\tau}}|u_{i,\tau}|^{2}(\det\nabla\Phi^{i}_{\tau})^{-1}(h-t_\tau)+\kappa\sum_{i}\int^{h}_{0}\int_{\Omega^{i}_{\tau}}|\nabla^{k_{0}}u_{i,\tau}|^{2}(h-t_\tau)\\
			&\displaystyle+\int^{h}_{0}\int_{\Omega^{i}_{\tau}}\mu^{M}_{i}(\theta_{i,\tau})|Du_{i,\tau}|^{2}(h-t_{\tau})\\
			&\displaystyle=\frac{h-\tau}{2}\int_\Gamma\abs{\nabla^{k_0+1}\eta^{0}}^{2}+\int^{h}_{0}
			\langle DK(\eta_{\tau}), (t_\tau-h) \partial_t\tilde{\eta}_\tau\rangle\\
			&\displaystyle+
			\frac{1}{h}\int^{h}_{0}\int_{\Omega^{i}_{\tau}}w_{i,\tau}\circ(\Phi^{i}_{\tau})^{-1}(\det\nabla\Phi^{i}_{\tau})^{-1}\cdot u_{i,\tau }(h-t_\tau)
			+\frac{1}{h}\int^{h}_{0}\int_{\Gamma}\beta_{\tau} \partial_t\tilde{\eta}_\tau(h-t_\tau)-\mathcal{D}_{\tau}.
		\end{array}
	\end{equation}
	Observe that the right hand side is in a shape such that it will converge to a respective right hand side for the limit equation. This eventually implies the strong convergence of all the terms on the left hand side. This works due to the parabolic nature of the $\tau$ level.
	Hence let us test \eqref{momentumafterpassingtau} by $-(\partial_{t}\eta,u)(h-t)$ to furnish:
	\begin{equation}\label{testlimitTdiff}
		\begin{array}{ll}
			&\displaystyle	\frac{\kappa}{2}\int_{0}^{h}|\nabla^{k_{0}+1}\eta|^{2}+\frac{1}{h}\int^{h}_{0}\int_{\Gamma}|\partial_{t}{\eta}|^{2}(h-t)+\kappa\int^{h}_{0}\int_{\Gamma} |\nabla^{k_{0}+1}\partial_{t}{\eta}|^2(h-t)\\
			&\displaystyle+\frac{1}{h}\sum_{i}\int^{h}_{0}\int_{\Omega^{i}}|u_{i}|^{2}(h-t)+\kappa\sum_{i}\int^{h}_{0}\int_{\Omega^{i}}|\nabla^{k_{0}}u_{i}|^{2}(h-t)+\int^{h}_{0}\int_{\Omega^{i}}\mu^{M}_{i}(\theta_{i})|Du_{i}|^{2}(h-t)\\
			&\displaystyle=K_{\kappa}(\eta^{0})h-\int^{h}_{0}K(\eta)+\frac{1}{h}\int^{h}_{0}\int_{\Omega^{i}}|w_{i}\circ(\Phi^{i})^{-1}|^{2}\cdot u_{i}(h-t)\\
			&\displaystyle\quad+\frac{1}{h}\int^{h}_{0}\int_{\Gamma}\beta\partial_{t}\eta(h-t).
		\end{array}
	\end{equation}
	Now  considering the limit $\tau\rightarrow 0$ ($t_{\tau}\rightarrow t$) in \eqref{testaproxTdiff} and comparing with \eqref{testlimitTdiff} one infers:
	\begin{equation}\label{convleftside}
		\begin{array}{ll}
			&\displaystyle	\lim\limits_{\tau\rightarrow 0}\bigg(\frac{\kappa}{2}\int_{0}^{h}\int_{\Gamma} |\nabla^{k_{0}+1}\eta_{\tau}|^{2} +\frac{1}{h}\int^{h}_{0}\int_{\Gamma}|\partial_{t}\widetilde{\eta}_{\tau}|^{2}(h-t_\tau)+\kappa\int^{h}_{0}\int_{\Gamma} |\nabla^{k_{0}+1}\partial_{t}\widetilde{\eta}_{\tau}|(h-t_\tau)\\
			&\displaystyle+\frac{1}{h}\sum_{i}\int^{h}_{0}\int_{\Omega^{i}_{\tau}}|u_{i,\tau}|^{2}(\det\nabla\Phi^{i}_{\tau})^{-1}(h-t_\tau)+\kappa\sum_{i}\int^{h}_{0}\int_{\Omega^{i}_{\tau}}|\nabla^{k_{0}}u_{i,\tau}|^{2}(h-t_\tau)\\
			&\displaystyle+\int^{h}_{0}\int_{\Omega^{i}_{\tau}}\mu^{M}_{i}(\theta_{i,\tau})|Du_{i,\tau}|^{2}(h-t_{\tau})\bigg)\\
			&\displaystyle	=\frac{\kappa}{2}\int_{0}^{h}|\nabla^{k_{0}+1}\eta|^{2}+\frac{1}{h}\int^{h}_{0}\int_{\Gamma}|\partial_{t}{\eta}|^{2}(h-t)+\kappa\int^{h}_{0}\int_{\Gamma} |\nabla^{k_{0}+1}\partial_{t}{\eta}|^2(h-t)\\
			&\displaystyle+\frac{1}{h}\sum_{i}\int^{h}_{0}\int_{\Omega^{i}}|u_{i}|^{2}(h-t)+\kappa\sum_{i}\int^{h}_{0}\int_{\Omega^{i}}|\nabla^{k_{0}}u_{i}|^{2}(h-t)+\int^{h}_{0}\int_{\Omega^{i}}\mu^{M}_{i}(\theta_{i})|Du_{i}|^{2}(h-t),
		\end{array}
	\end{equation}
	since $|\displaystyle\mathcal{D}_{\tau}|\leqslant\kappa h\tau\frac{\|\nabla^{k_{0}+1}\partial_{t}\widetilde{\eta}_{\tau}\|^{2}_{L^{2}(\Gamma\times(0,h))}}{2}$ and hence $\mathcal{D}_{\tau}\rightarrow 0$ by using \eqref{fromCoret}.\\ 
	Now in view of the available weak convergences \eqref{weakconv} and 
	the strong convexity of the $L^{2}$ norms one in particular has the convergence of the individual integrals appearing in the left hand side of \eqref{convleftside} and hence the following convergences of the dissipative terms
	\begin{equation}\label{weightedconv}
		\begin{array}{lll}
			& \displaystyle \sqrt{\mu^{M}_{i}(\theta_{i,\tau})}|Du_{i,\tau}|\sqrt{(h-t_{\tau})}\longrightarrow \sqrt{\mu^{M}_{i}(\theta_{i})}|Du_{i}|\sqrt{(h-t)} &\,\,\mbox{in}\,\, L^{2}(\Omega^{i}_{\tau}\times(0,h)),\\[2.mm]
			&\displaystyle |\nabla^{k_{0}}u_{i,\tau}|\sqrt{(h-t_{\tau})}\longrightarrow |\nabla^{k_{0}}u_{i}|\sqrt{(h-t)}&\,\,\mbox{in}\,\, L^{2}(\Omega^{i}_{\tau}\times(0,h)),\\[2.mm]
			&\displaystyle |u_{i,\tau}-w_{i,\tau}\circ(\Phi^{i}_{\tau})^{-1}|\sqrt{(h-t_{\tau})}\det(\nabla\Phi^{i}_{\tau})^{-1}\longrightarrow |u_{i}-w_{i}\circ(\Phi^{i})^{-1}|\sqrt{(h-t)}&\,\,\mbox{in}\,\, L^{2}(\Omega^{i}_{\tau}\times(0,h)),\\[2.mm]
			&\displaystyle |\partial_{t}\widetilde{\eta}_{\tau}-\beta_{\tau}|\sqrt{(h-t_{\tau})}\longrightarrow |\partial_{t}\eta-\beta|\sqrt{(h-t)}&\,\,\mbox{in}\,\, L^{2}(\Gamma\times(0,h)),\\[2.mm]
			&\displaystyle |\nabla^{k_{0}+1}\partial_{t}\widetilde{\eta}_{\tau}|\sqrt{(h-t_{\tau})}\longrightarrow |\nabla^{k_{0}+1}\partial_{t}\eta|\sqrt{(h-t)}&\,\,\mbox{in}\,\, L^{2}(\Gamma\times(0,h)),
		\end{array}
	\end{equation}
	for any $t\in[0,h]$ when $\tau\rightarrow 0.$ Consequently one has 
	\begin{equation}\label{dropweight}
		\begin{array}{lll}
			& \displaystyle \sqrt{\mu^{M}_{i}(\theta_{i,\tau})}|Du_{i,\tau}|\longrightarrow \sqrt{\mu^{M}_{i}(\theta_{i})}|Du_{i}|&\,\,\mbox{in}\,\, L^{2}(\Omega^{i}_{\tau}\times(0,t_{1})),\\[2.mm]
			&\displaystyle |\nabla^{k_{0}}u_{i,\tau}|\longrightarrow |\nabla^{k_{0}}u_{i}|&\,\,\mbox{in}\,\, L^{2}(\Omega^{i}_{\tau}\times(0,t_{1})),\\[2.mm]
			&\displaystyle |u_{i,\tau}-w_{i,\tau}\circ(\Phi^{i}_{\tau})^{-1}\sqrt{|\det(\nabla\Phi^{i}_{\tau})^{-1}}\longrightarrow |u_{i}-w_{i}\circ(\Phi^{i})^{-1}|&\,\,\mbox{in}\,\, L^{2}(\Omega^{i}_{\tau}\times(0,t_{1})),\\[2.mm]
			&\displaystyle |\partial_{t}\widetilde{\eta}_{\tau}-\beta_{\tau}|\longrightarrow |\partial_{t}\eta-\beta|&\,\,\mbox{in}\,\, L^{2}(\Gamma\times(0,t_{1})),\\[2.mm]
			&\displaystyle |\nabla^{k_{0}+1}\partial_{t}\widetilde{\eta}_{\tau}|\longrightarrow |\nabla^{k_{0}+1}\partial_{t}\eta|&\,\,\mbox{in}\,\, L^{2}(\Gamma\times(0,t_{1})),
		\end{array}
	\end{equation}
	for any  $t_{1}\in (0,h).$\\
	In particular \eqref{dropweight}$_{2,5}$ guarantees that:
	\begin{equation}\label{conmins} 
		\begin{array}{ll}
			& \displaystyle\int^{t_{1}}_{0}\int_{\Omega^{i}_{\tau}}\min\bigg\{|\nabla^{k_{0}}u_{i,\tau}|^{2},\tau^{-1}\bigg\}\rightarrow \int^{t_{1}}_{0}\int_{\Omega^{i}} |\nabla^{k_{0}+1}u_{i}|^{2},\\ &\displaystyle\int^{t_{1}}_{0}\int_{\Gamma}\min\bigg\{\bigg|\nabla^{k_{0}+1}\partial_{t}\widetilde{\eta}_{\tau}\bigg|^{2},\tau^{-1}\bigg\}\rightarrow \int^{t_{1}}_{0}\int_{\Gamma} |\nabla^{k_{0}+1}\partial_{t}\eta|^{2}.
		\end{array}
	\end{equation}
		for $a.e.$ $t_{1}\in (0,h).$ Hence \eqref{dropweight} and \eqref{conmins} at once infers 
		\begin{equation}\label{strngconvHtauizeta1}
			\begin{array}{l}
				\displaystyle \lim_{\tau\rightarrow 0} \left\langle \mathcal{D}^{\tau,i},\zeta_{i}\right\rangle=\left\langle\mathcal{D}^{i},\zeta_{i}\right\rangle,
			\end{array}
		\end{equation}
		for $\zeta_{i}\in C^{\infty}_{c}([0,h];C^{\infty}_{c}(\overline{\Omega})),$ for any $t_{1}\in(0,h),$ where $\langle \mathcal{D}^{i},\zeta_{i}\rangle$ is defined as:
		\begin{equation}
			\begin{array}{ll}
				\displaystyle \left\langle \mathcal{D}^{i},\zeta_{i}\right\rangle =&\displaystyle  \int^{t_{1}}_{0}\int_{\Omega^{i}}\mu^{M}_{i}(\theta_{i})|Du_{i}|^{2}\zeta_{i}+ {\kappa}\int^{t_{1}}_{0}\int_{\Omega^{i}}|\nabla^{k_{0}}u_{i}|^{2}\zeta_{i}+\mathcal{X}_{i}\frac{1}{2h}\int^{t_{1}}_{0}\int_{\Gamma}\bigg|\partial_{t}{\eta}-\beta\bigg|^{2}\zeta_{i}\\[4.mm]
				&\displaystyle+\frac{1}{2h}\int^{t_{1}}_{0}\int_{\Omega^{i}}|u_{i}-w_{i}\circ(\Phi^{i})^{-1}|^{2}\zeta_{i}
				+\kappa\mathcal{X}_{i}\int^{t_{1}}_{0}\int_{\Gamma}\bigg|\nabla^{k_{0}+1}\partial_{t}{\eta}\bigg|^{2}\zeta_{i},
			\end{array}
		\end{equation}
		\underline{\textit{Obtainment of a heat identity}} As a consequence of \eqref{strngconvHtauizeta1} and \eqref{aftrtau0heat2}, one obtains the following  equation solved by the temperature $\theta_{i}:$
		\begin{equation}\label{aftrtau0heatfinaltau}
			\begin{array}{ll}
				&\displaystyle c_{i}\int^{t_{1}}_{0}\frac{d}{dt}\int_{\Omega^{i}(t)}\theta_{i}(x,t)\zeta_{i}(x,t)-c_{i}\int^{t_{1}}_{0}\int_{\Omega^{i}(t)}\theta_{i}(x,t)(\partial_{t}+u_{i}\cdot\nabla)\zeta_{i}+k_{i}\int^{t_{1}}_{0}\int_{\Omega^{i}(t)}\nabla\theta_{i}\cdot\nabla\zeta_{i}\\
				&\displaystyle+\lambda\mathcal{I}_{i}\int^{t_{1}}_{0}\int_{\partial\Omega^{2}(t)}(\theta_{1}-\theta_{2})\zeta_{i}
				\displaystyle  =\left\langle \mathcal{D}^{i},\zeta_{i}\right\rangle\\
			\end{array}
		\end{equation}
		for $a.e.$ $t_{1}\in (0,h),$ for any $\zeta_{i}\in C^{\infty}_{c}([0,h];C^{\infty}_{c}(\overline{\Omega^{i}}))$ and $\zeta_{i}\geqslant 0.$ 
		The recovery of an identity for the heat evolution is a particular quality of our approximation taht seems necessary for the construction; that might actually make it a suitable candidate for numerical approximations.\\[2.mm]
		\underline{\textit{Obtainment of the energy balance:}} Let us use test functions of the form $(b,\psi)=(\partial_{t}\eta,u_1\chi_{\Omega^1}+u_2\chi_{\Omega^2})$ in \eqref{momentumafterpassingtau} (which is possible in view of the fact that $\partial_{t}\eta\in L^{2}(0,h;W^{k_{0}+1,2}(\Gamma))$, giving a sense to the duality pairing $\langle DK_{\kappa}(\eta),\partial_{t}\eta\rangle$) to furnish \eqref{frommomentumaftertestingdiss} and next use $\zeta_{i}=1$ in \eqref{aftrtau0heatfinaltau}, sum the resulting expressions to furnish the energy balance \eqref{hlevelenergy}. 
		\subsubsection{Obtainment of entropy evolution inequality}
		As a consequence, one can further use test functions of the form $\varphi'(\theta_{i}),$ (where $\varphi(\theta)$ is monotone and concave for $\theta>0$) in  \eqref{aftrtau0heat2} (the use of such test functions can be justified by a density argument) and use Raynolds transport theorem to conclude the following entropy evolution
		\begin{equation}\label{entropybh}
			\begin{array}{ll}
				\displaystyle\sum_{i=1}^{2}\left(c_{i}\frac{d}{dt}\int_{\Omega^{i}}\varphi(\theta_{i})+k_{i}\int_{\Omega^{i}}|\nabla\theta_{i}|^{2}\varphi''(\theta_{i})\right)+\lambda\int_{\partial\Omega^{2}}\left(\theta_{1}-\theta_{2}\right)\left(\varphi'(\theta_{1})-\varphi'(\theta_{2}))\right)\geqslant 0.
			\end{array}
		\end{equation}
		Indeed $\varphi(\cdot),$ should be chosen such that all the integrals in \eqref{entropybh} make sense. For instance, since $\theta_{i}\geqslant\gamma>0$ can be chosen to be $\theta^{(\lambda+1)}_{i}$ for $\lambda\in (-1,0)$ or the physical entropy $\ln\theta_{i}.$
		\subsubsection{Conclusion and the proof of Theorem \ref{Thmmaintaulev}}
		So far we have proved the existence of a triplet $(u_{i},\theta_{i},\eta)$ solving all the assertions of Theorem \ref{Thmmaintaulev} except the fact that the viscosity coefficient of the fluid $\mu^{M}_{i}(\cdot)$ ($c.f.$ \eqref{approxmui} ) is a penalised version of the original viscosity $\mu_{i}.$ Recall that the only reason behind this penalisation was to guarantee a positive lower bound of the viscosity in the absence of a minimal principle of temperature. Now since we have the non degeneracy \eqref{positivitytheta} of $\theta_{i}$ and all the required estimates ($c.f.$ Corollary \eqref{estimatesinterpole1}) independent of the penalisation parameter $M,$ we can pass $M\rightarrow\infty$ to conclude the proof of the assertions $(i)$ to $(viii)$ of Definition \ref{Def:taulayer} and consequently Theorem \ref{Thmmaintaulev} follows.
		\section{Inertial layer {$h\rightarrow 0$}}\label{ht0lev}
		The goal of this section is to pass to the limit $h\rightarrow 0$ in  \eqref{discretemomen1}-\eqref{intenballh}  and to obtain a solution to the regularized system (with $\kappa$ being fixed). Let us define the notion of weak solution for the regularized system.
		\begin{definition}\label{Def:reg}
			A triplet $(u_{i},\theta_{i},\eta)$ is said to be a solution of the regularized system (with $\kappa$ being the regularizing parameter) in $(0,T)$ if:
			\begin{equation}\label{regularity3}
				\begin{array}{ll}
					&\displaystyle u_{i}\in  L^{2}(0,T;W^{k_{0},2}(\Omega^{i}))\cap L^{\infty}(0,T;L^{2}(\Omega^{i})),\,\,\mathrm{div}\,u_{i}=0\,\,\mbox{in}\,\,\Omega^{i},\\
					& \displaystyle\theta_{i}\in L^{\infty}(0,T;L^{1}(\Omega^{i}))\cap L^{p}(\Omega^{i}\times(0,T))\cap L^{s}(0,T;W^{1,s}(\Omega^{i})),\,\,\mbox{for all}\,\,(p,s)\in[1,\frac{5}{3})\times [1,\frac{5}{4}),\\
					&\displaystyle \eta \in L^{\infty}(0,T;W^{k_{0}+1,2}(\Gamma))\cap W^{1,2}(0,T;W^{k_{0}+1,2}(\Gamma)),\,\,\mbox{for}\,\, k_{0}>3,\\
				\end{array}
			\end{equation}
			and the following hold\\[2.mm]
			(i) A decomposition of $\Omega$ of the form \eqref{Omegat12} with \eqref{varphieta}-\eqref{Sigmat} hold.\\[2.mm]
			(ii) {\textit{The momentum balance holds in the sense:}}\\
			\begin{equation}\label{momentplvlh}
				\begin{array}{ll}
					&\displaystyle \int_{0}^{t}\bigg(\langle DK_{\kappa}(\eta),b\rangle-\int_{\Gamma}\partial_{t}\eta\partial_{t}b+\kappa\int_{\Gamma}\nabla^{k_{0}+1}\partial_{t}\eta\nabla^{k_{0}+1}b\bigg)\\
					&\displaystyle+\sum_{i}\int^{t}_{0}\bigg(-\langle u_{i},\partial_{t}\psi-u_{i}\cdot\nabla\psi\rangle+\kappa\int_{\Omega^{i}}\nabla^{k_{0}}u_{i}:\nabla^{k_{0}}\psi+\int_{\Omega^{i}}\mu_{i}(\theta_{i})D(u_{i}):D(\psi)\bigg)\\[2.mm]
					&\displaystyle=-\langle \partial_{t}\eta(\cdot,t)b(\cdot,t)\rangle +\langle \eta_{1}^{0},b(\cdot,0)\rangle-\langle u_{i}(\cdot,t),\psi(\cdot,t)\rangle+\langle u^{0}_{i},\psi(\cdot,0)\rangle
				\end{array}
			\end{equation}
			for $a.e.$ $t\in (0,T),$ for all $b\in L^{\infty}(0,T;W^{k_{0}+1,2}(\Gamma))\cap W^{1,\infty}(0,T;L^{2}(\Gamma)),$ for all $\psi\in C^{\infty}_0([0,T]\times \Omega)$  with $\mbox{div}\,\psi=0$ on $\Omega$ satisfying
			\begin{equation}
				\begin{array}{l}
					\displaystyle \Tr_{\Sigma_{\eta}}u_{i}=\partial_{t}\eta\nu\quad\mbox{and}\quad \Tr_{\Sigma_{\eta}}\psi=b\nu\quad\mbox{on}\quad \Gamma.
				\end{array}
			\end{equation}\\[2.mm]
			(iii) {\textit{The heat evolution holds in the sense:}}\\
			\begin{equation}\label{heatfnlhlev}
				\begin{array}{ll}
					&\displaystyle c_{i}\int^{t}_{0}\frac{d}{dt}\langle \theta_{i},\zeta_{i}\rangle-c_{i}\int^{t}_{0}\langle \theta_{i},(\partial_{t}+u_{i}\cdot\nabla)\zeta_{i}\rangle+k_{i}\int^{t_{1}}_{0}\int_{\Omega^{i}}\nabla\theta_{i}\cdot\nabla\zeta_{i}
					\displaystyle+\lambda\mathcal{I}_{i}\int^{t}_{0}\int_{\partial\Omega^{2}}(\theta_{1}-\theta_{2})\zeta_{i}\\
					&\displaystyle=\int^{t}_{0}\int_{\Omega^{i}}\mu_{i}(\theta_{i})|Du_{i}|^{2}\zeta_{i}+\kappa\int^{t}_{0}\int_{\Omega^{i}}|\nabla^{k_{0}}u_{i}|^{2}\zeta_{i}+\kappa\mathcal{X}_{i}\int^{t}_{0}\int_{\Gamma}|\nabla^{k_{0}+1}\partial_{t}\eta|^{2}\zeta_{i}
				\end{array}
			\end{equation}
			for $a.e.$ $t\in [0,T],$ and for $\zeta_{i}\in C^{\infty}_{c}([0,T];C^{\infty}_{c}(\overline{\Omega^{i}}))$ where $\mathcal{I}_{1}=1,$ $\mathcal{I}_{2}=-1,$ $\mathcal{X}_{1}=1$ and $\mathcal{X}_{2}=0.$\\[2.mm] 
			(iv) \textit{Positivity of temperature:}\\
			\begin{equation}\label{postempht0lev}
				\begin{array}{l}
					\displaystyle \theta_{i}\geqslant \gamma\,\,\mbox{in}\,\,\Omega^{i}\times(0,T)\,\,\mbox{provided}\,\,\theta^{0}_{i}\geqslant\gamma>0\,\,\mbox{on}\,\,\Omega^{i}_{0}.
				\end{array}
			\end{equation}
			(v) \textit{The total energy is conserved:}\\
			\begin{equation}\label{energybalhlev}
				\begin{array}{ll}
					\mathbb{E}_{tot}(\theta_{i},u_{i},\partial_{t}\eta,\eta)(t)=\mathbb{E}_{tot}(\theta^{0}_{i},u^{0}_{i},\eta^{0},\eta^{0}_{1}),
				\end{array}
			\end{equation}
			where the total energy of the system is defined as
			$$\mathbb{E}_{tot}(\theta_{i},u_{i},\partial_{t}\eta,\eta)=\sum^{2}_{i=1}\int_{\Omega^{i}}\bigg(c_{i}\theta_{i}+\frac{1}{2}|u_{i}|^{2}\bigg)+\int_{\Gamma}\bigg(\frac{|\partial_{t}\eta|^{2}}{2}+K_{\kappa}(\eta)\bigg).$$\\[2.mm]
			(vi)\textit{An identity involving the dissipation terms}:
			\begin{equation}\label{disstypesthlevfhlvl}
				\begin{array}{ll}
					&\displaystyle K_{\kappa}(\eta(t))+\frac{1}{2}\int_{\Gamma}|\partial_{t}\eta(t)|^{2}+\frac{1}{2}\sum_{i}\int_{\Omega^{i}}|u_{i}(t)|^{2}\\
					&\displaystyle+\sum_{i}\left(\int^{t}_{0}\int_{\Omega^{i}}\mu_{i}(\theta_{i})|Du_{i}|^{2}+\kappa\int^{t}_{0}\int_{\Omega^{i}}|\nabla^{k_{0}}u_{i}|^{2}+\kappa\int^{t}_{0}\int_{\Gamma}|\nabla^{k_{0}+1}\partial_{t}\eta|^{2}\right)\\
					&\displaystyle = K_{\kappa}(\eta^{0})+\frac{1}{2}\int_{\Gamma}|\eta^{0}_{1}|^{2}+\frac{1}{2}\int_{\Omega^{i}_{0}}|u^{0}_{i}|^{2}
				\end{array}
			\end{equation}
			for $a.e.$ $t\in[0,T].$\\[2.mm]
			(vii) \textit{A general entropy inequality holds in the following sense:}\\
			\begin{equation}\label{entrpybalhlev}
				\begin{array}{l}
					\displaystyle\sum_{i=1}^{2}\left(c_{i}\frac{d}{dt}\int_{\Omega^{i}}\varphi(\theta_{i})+k_{i}\int_{\Omega^{i}}|\nabla\theta_{i}|^{2}\varphi''(\theta_{i})\right)+\lambda\int_{\partial\Omega^{2}}\left(\theta_{1}-\theta_{2}\right)\left(\varphi'(\theta_{1})-\varphi'(\theta_{2}))\right)\geqslant 0,
				\end{array}
			\end{equation}
			for any $\varphi(\theta)$ such that $\varphi(\cdot)$ is concave and monotone for $\theta>0,$ such that all the integrals above make sense.
		\end{definition}
		Now let us state the central result of the present section.\\
		\begin{theorem}\label{Th:mainhlev}
			Let all the assumptions $(i)-(v)$ stated in Theorem \ref{Th:main} hold along with the further regularity $\eta_{0}\in W^{k_{0},2}(\Gamma).$ Then there exist $T_F^{\kappa}\in(0,\infty]$ and a weak solution to the regularized problem in the interval $(0,T)$ for any $T<T_F^{\kappa}$ in the sense of Definition \ref{Def:reg}.\\
			Finally, $T_F^{\kappa}$ is finite only if 
			\begin{equation}\label{degenfstkind2}
				\text{ either }\lim_{s\to T_F^{\kappa}}\eta(s,y)\searrow a_{\partial\Omega}\text{ or }\lim_{s\to T_F^{\kappa}}\eta(s,y)\nearrow b_{\partial\Omega}
			\end{equation}
			for some $y\in\Gamma.$\\ 
			Note that in view of \eqref{tblrnbd}-\eqref{safedis} the moving elastic interface avoids any collision with the outer boundary $\partial\Omega$ in the maximal interval of existence $(0,T_{F}^{\kappa}).$
		\end{theorem}
		\subsection{Construction of iterative approximation}
		Now we construct an approximative solution of a heat conducting FSI problem using solutions obtained in the previous section. More particularly we use Theorem \ref{Thmmaintaulev} first in the time interval $[0,h].$ Next we use the resulting functions $\theta_{i}(\cdot,h),$ $\eta(\cdot,h)$ $\partial_{t}\eta(\cdot,h)$ and a modified version of $u_{i}\circ\Phi(t)\circ\Phi(h)^{-1}$ as $\theta^{0}_{i},$ $\eta^{0},$ $\beta$ and $w_{i}$ on $(h,2h).$ In view of the energy type equality \eqref{energybalhlev} they do qualify as valid initial data. We repeat this procedure in an iterative manner throughout the interval $I=[0,T].$\\
		Given a solutiion $(\theta^{h}_{i,l},\eta^{h}_{l},u^{h}_{i,l},\Phi^{h}_{i,l})$  in the interval $[0,h],$ (note that later this will correspond to [$(l-1)h,lh],$ but for now we will consider each of these intervals as $[0,h]$ and distinguish them by index $l$), we apply Theorem \ref{Thmmaintaulev} with
		\begin{equation}\label{setincond}
			\begin{array}{ll}
				&\displaystyle \eta^{h}_{l}(h)\,\,\mbox{as}\,\,\eta_{0},\\
				&\displaystyle\Omega=\cup^{2}_{i=1}\Omega^{i,h}_{l}\\
				&\displaystyle \partial\Omega^{2,h}_{l}=\Sigma_{\eta^{h}_{l}}\,\,\mbox{as}\,\, \partial\Omega^{2}_{0}\,\,\mbox{and}\,\, \partial\Omega^{1,h}_{l}=\Sigma_{\eta^{h}_{l}}\cup\partial\Omega\,\,\mbox{as}\,\, \partial\Omega^{1}_{0}\\
				&\displaystyle \theta^{h}_{i,l}(h)\,\,\mbox{as}\,\,\theta^{0}_{i},\\
				&\displaystyle \partial_{t}\eta^{h}_{l}(t)\,\,\mbox{as}\,\,\beta(t)\,\,\mbox{for all}\,\,t\in[0,h],\\
				&\displaystyle u^{h}_{i,l}(\cdot,t)\circ \Phi^{h}_{i,l}(t)\circ\Phi^{h}_{i,l}(h)^{-1}\,\,\mbox{as}\,\, w_{i}(\cdot,t),\,\,\mbox{for all}\,\, t\in [0,h].
			\end{array}
		\end{equation}
		The solution constructed by Theorem \ref{Thmmaintaulev} is denoted by $(\theta^{h}_{i,l+1},\eta^{h}_{l+1},u^{h}_{i,l+1},\Phi^{h}_{i,l+1}).$\\
		With these solutions in hand, we can now construct an $h-$approximation on the entire interval $I.$ Specifically, we set
		\begin{equation}\label{defapproxhlev}
			\begin{array}{lll}
				&\displaystyle \eta^{h}(t)=\eta^{h}_{l}(t-(l-1)h)& \mbox{for}\,\,t\in[(l-1)h,lh]\\
				&\displaystyle \theta^{h}_{i}(t)=\theta^{h}_{i,l}(t-(l-1)h)& \mbox{for}\,\,t\in[(l-1)h,lh]\\
				& \displaystyle u^{h}_{i}(t)=u^{h}_{i,l}(t-(l-1)h)& \mbox{for}\,\,t\in[(l-1)h,lh]\\
				&\displaystyle \Omega=\cup^{2}_{i=1}\Omega^{i,h}(t),\,\,\mbox{where}\,\, \Omega^{i,h}(t)=\Omega^{i,h}_{l}(t-(l-1)h)& \mbox{for}\,\,t\in[(l-1)h,lh]
			\end{array}
		\end{equation}
		and a redefined flow map for $t\in[(l-1)h,lh]$
		\begin{equation}\label{redefflmap}
			\begin{array}{lll}
				&\displaystyle \Phi^{i,h}_{s}(t)=\Phi^{h}_{i,l}(s+t-(l-1)h)\circ\Phi^{h}_{i,l}(t-(l-1)h)^{-1}& \mbox{for}\,\,t+s\in[(l-1)h,lh]\\
				&\displaystyle \Phi^{i,h}_{s}(t)=\Phi^{h}_{i,l+1}(s+t-lh)\circ\Phi^{h}_{i,l}(h)\circ\Phi^{h}_{i,l}(t-(l-1)h)^{-1}& \mbox{for}\,\,t+s\in[lh,(l+1)h]\\
				&\displaystyle \Phi^{i,h}_{s}(t)=\Phi^{h}_{i,l-1}(s+t-(l-2)h)\circ\Phi^{h}_{i,l-1}(h)^{-1}\circ\Phi^{h}_{i,l}(t-(l-1)h)^{-1}& \mbox{for}\,\,t+s\in[(l-2)h,(l-1)h]\\
			\end{array}
		\end{equation}
		and so on. 
		Now we write down the following equations solved by the quadruple $(\theta^{h}_{i},\eta^{h}_{i},u^{h}_{i},\Phi^{h}_{i,s}).$\\
		(i) \textit{Momentum balance:}\\
		\begin{equation}\label{mombalhapp}
			\begin{array}{ll}
				&\displaystyle \int^{t}_{0}\langle DK_{\kappa}(\eta^{h}),b\rangle+\int^{t}_{0}\int_{\Gamma}\frac{\partial_{t}\eta^{h}(\cdot)-\partial_{t}\eta^{h}(\cdot-h)}{h}b+\kappa\int^{t}_{0}\int_{\Gamma}\nabla^{k_{0}+1}\partial_{t}\eta^{h}\nabla^{k_{0}+1}b\\[2.mm]
				&\displaystyle +\sum_{i}\int^{t}_{0}\int_{\Omega^{i,h}}\frac{u^{h}_{i}(\cdot)-u^{h}_{i}(\cdot-h)\circ\Phi^{i,h}_{-h}}{h}\psi+\kappa\sum_{i}\int^{t}_{0}\int_{\Omega^{i,h}}\nabla^{k_{0}}u^{h}_{i}:\nabla^{k_{0}}\psi\\[2.mm]
				&\displaystyle+\sum_{i}\int^{t}_{0}\int_{\Omega^{i,h}}\mu_{i}(\theta_{i}^{h})Du^{h}_{i}:D\psi=0			\end{array}
		\end{equation}
		for a.e. $t\in[0,T],$ for all $b\in L^{\infty}(0,T;W^{k_{0}+1,2}(\Gamma))\cap W^{1,\infty}(0,T;L^{2}(\Gamma)),$ for all $\psi\in C^{\infty}_0([0,T]\times \Omega)$  with $\mbox{div}\,\psi=0$ on $[0,T]\times \Omega$ satisfying the compatibility conditions
		\begin{equation}
			\begin{array}{l}
				\displaystyle \Tr_{\Sigma_{\eta^{h}}}u_{i}^{h}=\partial_{t}\eta^{h}\nu\quad\mbox{and}\quad \Tr_{\Sigma_{\eta^{h}}}\psi=b\nu\quad\mbox{on}\quad \Gamma.
			\end{array}
		\end{equation}\\[2.mm]
		(ii)\textit{Heat evolution:}\\
		\begin{equation}\label{aftrtau0heat4}
			\begin{array}{ll}
				&\displaystyle c_{i}\int^{t}_{0}\frac{d}{dt}\int_{\Omega^{i,h}}\theta^{h}_{i}(x,t)\zeta_{i}(x,t)-c_{i}\int^{t}_{0}\int_{\Omega^{i,h}}\theta^{h}_{i}(x,t)(\partial_{t}+u^{h}_{i}\cdot\nabla)\zeta_{i}+k_{i}\int^{t}_{0}\int_{\Omega^{i,h}}\nabla\theta^{h}_{i}\cdot\nabla\zeta_{i}\\
				&\displaystyle+\lambda\mathcal{I}_{i}\int^{t}_{0}\int_{\partial\Omega^{2,h}}(\theta^{h}_{1}-\theta^{h}_{2})\zeta_{i}=\left\langle \mathcal{D}^{h,i},\zeta_{i}\right\rangle
			\end{array}
		\end{equation}
		for $a.e.$ $t\in [0,T]$ and for $\zeta_{i}\in C^{\infty}_{c}([0,T];C^{\infty}_{c}(\overline{\Omega^{i,h}}))$ and $\zeta_{i}\geqslant0.$ The duality coupling $ \left\langle \mathcal{D}^{h,i},\zeta_{i}\right\rangle$ is given as
		\begin{equation}\label{defHhi1zeta}
			\begin{array}{ll}
				&\displaystyle \left\langle \mathcal{D}^{h,i},\zeta_{i}\right\rangle\\
				&\displaystyle = \int^{t_{1}}_{0}\int_{\Omega^{i}}\mu_{i}(\theta_{i}^{h})|Du_{i}^{h}|^{2}\zeta_{i}+ {\kappa}\int^{t_{1}}_{0}\int_{\Omega^{i}}|\nabla^{k_{0}}u_{i}^{h}|^{2}\zeta_{i}+\mathcal{X}_{i}\frac{1}{2h}\int^{t_{1}}_{0}\int_{\Gamma}\bigg|\partial_{t}{\eta^{h}}-\partial_{t}\eta^{h}(\cdot-h)\bigg|^{2}\zeta_{i}\\[4.mm]
				&\displaystyle+\frac{1}{2h}\int^{t_{1}}_{0}\int_{\Omega^{i}}\bigg|u_{i}-u_{i}^{h}\circ(\Phi^{i,h}_{-h})\bigg|^{2}\zeta_{i}
				+\kappa\mathcal{X}_{i}\int^{t_{1}}_{0}\int_{\Gamma}\bigg|\nabla^{k_{0}+1}\partial_{t}{\eta}\bigg|^{2}\zeta_{i},
			\end{array}
		\end{equation}
		where $\mathcal{I}_{1}=1,$ $\mathcal{I}_{2}=-1,$ $\mathcal{X}_{1}=1$ and $\mathcal{X}_{2}=0.$\\
	(iii)\textit{Positivity of temperature:}\\
	\begin{equation}\label{postemphlev}
		\begin{array}{l}
			\displaystyle \theta_{i}^{h}\geqslant \gamma\,\,\mbox{in}\,\,\Omega^{i,h}\times(0,T)\,\,\mbox{provided}\,\,\theta^{0}_{i}\geqslant\gamma>0\,\,\mbox{on}\,\,\Omega^{i}_{0}.
		\end{array}
	\end{equation}
	(iv)\textit{Energy balance:}\\
	\begin{equation}\label{enbalancehlvl}
		\begin{array}{ll}
			&\displaystyle \sum_{i}\int_{\Omega^{i,h}}\theta^{h}_{i}(x,t)+K_{\kappa}(\eta^{h})+\frac{1}{2h}\int^{t}_{t-h}\int_{\Gamma}|\partial_{t}\eta^{h}|^{2}+\frac{1}{2h}\sum_{i}\int_{t-h}^{t}\int_{\Omega^{i,h}}|u_{i}^{h}|^{2}\\[2.mm]
			&\displaystyle = K_{\kappa}(\eta^{0})+\sum_{i}\int_{\Omega^{i}_{0}}\theta^{0}_{i}+\frac{1}{2}\int_{\Gamma}|\eta^{0}_{1}|^{2}+\frac{1}{2}\int_{\Omega^{i}_{0}}|u^{0}_{i}|^{2}.
		\end{array}
	\end{equation}
	for $a.e.$ $t\in[0,T].$\\
	(v)\textit{An identity involving the dissipation terms}:\\
	\begin{equation}\label{disstypesthlev}
		\begin{array}{ll}
			&\displaystyle K_{\kappa}(\eta^{h}(t))+\frac{1}{2h}\int^{t}_{t-h}\int_{\Gamma}|\partial_{t}\eta^{h}|^{2}+\frac{1}{2h}\sum_{i}\int_{t-h}^{t}\int_{\Omega^{i,h}(t)}|u_{i}^{h}|^{2}+\sum_{i} \left\langle \mathcal{D}^{h,i},1\right\rangle\\\
			&\displaystyle = K_{\kappa}(\eta^{0})+\frac{1}{2}\int_{\Gamma}|\eta^{0}_{1}|^{2}+\frac{1}{2}\int_{\Omega^{i}_{0}}|u^{0}_{i}|^{2}
		\end{array}
	\end{equation}
	for $a.e.$ $t\in[0,T],$ where $ \left\langle \mathcal{D}^{h,i},1\right\rangle$ can be obtained by using \eqref{defHhi1zeta}.\\
	(vi)\textit{Entropy inequality:}\\
	\begin{equation}\label{entropyineqhlvl}
		\begin{array}{l}
			\displaystyle\sum_{i=1}^{2}\left(c_{i}\frac{d}{dt}\int_{\Omega^{i,h}}\varphi(\theta^
			{h}_{i})+k_{i}\int_{\Omega^{i,h}}|\nabla\theta_{i}^{h}|^{2}\varphi''(\theta_{i}^{h})\right)+\lambda\int_{\partial\Omega^{2,h}}\left(\theta_{1}^{h}-\theta_{2}^{h}\right)\left(\varphi'(\theta_{1}^{h})-\varphi'(\theta_{2}^{h}))\right)\geqslant 0,
		\end{array}
	\end{equation}
	for any $\varphi(\theta)$ such that $\varphi(\cdot)$ is concave and monotonicaly increasing for $\theta>0.$\\[2.mm]
	{\textit{A further observation concerning the uniform estimate of the difference quotients in time variable:}}\\
	In order to give a sense to  $(\partial_{tt}\eta,\partial_{t}u_{i})$ (this will play a crucial role later in Section \ref{testsolhlev} while using $(u_{i},\partial_{t}\eta)$ as a test function in the mechanical equation) one observes the following as a consequence of a density argument and \eqref{mombalhapp}:
	\begin{equation}\label{diffquotest}
		\begin{array}{ll}
			&\displaystyle \bigg|\bigg\langle \bigg(\frac{\partial_{t}\eta^{h}(\cdot)-\partial_{t}\eta^{h}(\cdot-h)}{h}, \frac{u^{h}_{i}(\cdot)-u^{h}_{i}(\cdot-h)\circ \Phi^{i,h}_{-h}}{h} \bigg), (b,\psi_{i})\bigg\rangle_{\mathcal{Y}^{*},\mathcal{Y}}\bigg|
			\displaystyle \leqslant C\bigg\|(b,\psi_{i})\bigg\|_{\mathcal{Y}},
		\end{array}
	\end{equation}
	where $\mathcal{Y}$ denotes the space:
\begin{equation}\label{DefmY}
	\begin{array}{l}
\displaystyle\mathcal{Y}=\{(b,\psi_{i})\in L^{2}(0,T;W^{k_{0},2}(\Gamma))\times L^{2}(0,T; W^{k_{0}+1,2}(\Omega^{i,h}))\suchthat \,\,\mbox{div}\,u^{h}_{i}=0,\,\,\mbox{tr}_{\Sigma_{\eta^{h}}}\psi=b\nu\,\,\mbox{on}\,\,\Gamma\},
\end{array}
\end{equation}
	and the constant $C$ is independent of $h.$\\[2.mm]
	Based on the estimates \eqref{enbalancehlvl} and \eqref{disstypesthlev} one at once has the following convergences:
	\begin{equation}\label{weakconvhlev}
		\begin{array}{lll}
			&\displaystyle {\eta}^{h}\rightharpoonup^* \eta\,\,&\displaystyle\mbox{in}\,\, L^{\infty}(0,T;W^{k_{0}+1,2}(\Gamma))\,\,k_{0}>3,\\[2.mm]
			&\displaystyle \eta^{h}\rightarrow \eta\,\,&\displaystyle\mbox{in}\,\,C^{\alpha}(\Gamma\times[0,T])\,\,\mbox{for some}\,\,\alpha\in(0,1),\\[2.mm]
			&\displaystyle \partial_{t}{\eta}^{h}\rightharpoonup \partial_{t}\eta\,\,&\displaystyle\mbox{in}\,\, L^{2}(0,T;W^{k_{0}+1,2}(\Gamma)),\\[2.mm]
			&\displaystyle u_{i}^{h}\rightharpoonup u_{i}\,\,&\displaystyle \mbox{in}\,\, L^{2}(0,T;W^{k_{0},2}(\Omega^{i,h})),\\[2.mm]
			&\displaystyle \theta^{h}_{i}\rightharpoonup^* \theta_{i}\,\,&\displaystyle \mbox{in}\,\, L^{\infty}(0,T;L^{1}(\Omega^{i,h}))
		\end{array}
	\end{equation}
	Using \cite[Lemma 3.6]{BreitKamSch} one  can show the follwoing (for the details of argument we refer to \cite[Section 4.1]{BreitKamSch}):
	\begin{equation}\label{bnddet}
		\begin{array}{l}
			\displaystyle \underline{c}\leqslant \det\nabla\Phi^{i,h}_{s}\leqslant\overline{c},
		\end{array}
	\end{equation}
	where $\underline{c}$ and $\overline{c}$ are inependent of $h$ and $s$ (however diverging as $T\rightarrow\infty$ or $\kappa\rightarrow 0$) and
	\begin{equation}\label{convPhih}
		\begin{array}{l}
			\displaystyle \Phi^{i,h}_{s}\rightarrow \Phi^{i}\,\,\displaystyle \mbox{in}\,\, C^{0}([0,T];C^{1,\alpha}(\Omega^{i}_{0}))\,\,\mbox{for some}\,\,\alpha\in(0,1).
		\end{array}
	\end{equation}
	\subsubsection{Further estimates of the temperature}\label{frthrtempest}
	Using test functions of the form $\zeta_{i}(x,t)=\varphi'(\theta^{h}_{i})$ with $\varphi(\theta^{h}_{i})=(\theta^{h}_{i})^{\lambda+1},$ $\lambda\in(-1,0)$ (which can be justified by a density argument) in \eqref{aftrtau0heat4} and summing over $i\in\{1,2\}$ one arrives at
	\begin{equation}\label{aprioriestthetahi}
		\begin{array}{ll}
			&\displaystyle\sum_{i=1}^{2}\int_{I}\int_{\Omega^{i,h}}\mu_{i}(\theta_{i}^{h})|D(u_{i}^{h})|^{2}(\theta^{h}_{i})^{\lambda}-\sum_{i=1}^{2}k_{i}\int_{I}\int_{\Omega^{i,h}}\lambda|\nabla\theta_{i}^{h}|^{2}(\theta^{h}_{i})^{\lambda-1}\\
			&\displaystyle \leqslant\frac{1}{(\lambda+1)}\left(\sum_{i=1}^{2}c_{i}\|(\theta_{i}^{h}(T))^{\lambda+1}\|_{L^{1}(\Omega^{i}_{\eta(T)})}-\sum_{i=1}^{2}c_{i}\|(\theta^{0}_{i})^{\lambda+1}\|_{L^{1}(\Omega^{i}_{0})}\right),
		\end{array}
	\end{equation}
	where $I=(0,T),$ since 	$$\left(\varphi'(\theta_{1}^{h})-\varphi'(\theta_{2}^{h})\right)\cdot(\theta_{1}^{h}-\theta_{2}^{h})=\varphi''(\xi_{\theta_{1}^{h},\theta_{2}^{h}})(\theta_{1}^{h}-\theta_{2}^{h})^{2}\leqslant 0 \,\,\mbox{on}\,\,\Sigma_{\eta},$$
	for some $\xi_{\theta_{1}^{h},\theta_{2}^{h}}$ lying on the line joining $\theta^{h}_{1}$ and $\theta^{h}_{2}.$\\
	From \eqref{aprioriestthetahi} only clearly has an unifrom bound on $\|(\nabla\theta_{i}^{h})^{\frac{\lambda+1}{2}}\|_{L^{2}(\Omega^{i,h})},$ which along with \eqref{enbalancehlvl} and standard interpolation arguments furnish:
	\begin{equation}\label{Lpthetaih}
		\begin{array}{ll}
			&\displaystyle\|\theta_{i}^{h}\|_{L^{p}(0,T;L^{p}(\Omega^{i,h}))}\leqslant C(\eta^{0},\theta^{0}_{i},\eta^{0}_{1},u^{0}_{i})\,\,\mbox{for}\,\,p\in[1,\frac{5}{3}),\\
			&\displaystyle 	\displaystyle\int_{(0,T)}\|\theta_{i}^{h}\|^{s}_{W^{1,s}(\Omega^{i,h})}\leqslant C(\eta^{0},\theta^{0}_{i},\eta^{0}_{1},u^{0}_{i})\,\,\mbox{for all}\,\, s\in[1,\frac{5}{4}).
		\end{array}
	\end{equation}
	\subsection{Compactness of temperature}\label{compacttemhlev}
	We consider a parabolic cylinder of the form $B\times J$ such that $2B\times 2J\Subset \Omega^{i,h}\times (0,T),$ for small enough $h,$ which is possible because of the uniform convergence \eqref{weakconvhlev}$_{2}$. Now consider $\zeta_{i}\in C^{\infty}_{c}(B\times J)$ in \eqref{aftrtau0heat4}.\\ 
	Consequently one has the following estimate:
	\begin{equation}\label{estimatederttheta}
		\begin{array}{ll}
			&\displaystyle  \bigg|\iint_{B\times J}\partial_{t}\theta^{h}_{i}\zeta_{i}\bigg|\\
			&\displaystyle \leqslant C\bigg(\|\theta^{h}_{i}\|_{L^{\infty}(L^{1})}\|u^{h}_{i}\|_{ L^{2}(W^{k_{0},2})}\|\nabla\zeta_{i}\|_{L^{\infty}(L^{\infty})}+\|\theta^{h}_{i}\|_{L^{\frac{5}{4}-}(W^{1,\frac{5}{4}-})}\|\nabla\zeta_{i}\|_{L^{\infty}(L^{\infty})}\\
			&\displaystyle\qquad\qquad + \left|\left\langle \mathcal{D}^{h,i},1\right\rangle\right| \|\zeta_{i}\|_{L^{\infty}(L^{\infty})}\bigg)\\
			&\displaystyle \leqslant C(\eta^{0},\theta^{0}_{i},\eta^{0}_{1},u^{0}_{i})\|\zeta_{i}\|_{L^{\infty}(J;W^{2,p}(B))},
		\end{array}
	\end{equation}
	for $p>3,$ where we have used \eqref{Lpthetaih}$_{2}$, $L^{\infty}(L^{1})$ norm bound of $\theta^{h}_{i}$ from \eqref{enbalancehlvl} and the bound on $\mathcal{D}^{h,i}(t)$ from \eqref{disstypesthlev}. Now one can use \eqref{Lpthetaih}$_{1}$ and imitate the exact line of arguments used in showing \eqref{claimstrngconvuptobndry} to prove that
	\begin{equation}\label{claimstrngconvuptobndrythetah}
		\begin{array}{l}
			\displaystyle\theta_{i}^{h}\rightarrow \theta_{i}\,\,\mbox{in}\,\,L^{s}(\Omega^{i,h}\times(0,T)),\,\,\mbox{for any}\,s\in [1,\frac{5}{3}).
		\end{array}
	\end{equation}
	\subsection{Compactness of average velocities}
	This section is devoted to prove the strong convergence of the fluid velocities as well as the velocity of the structure. We follow a strategy based on an application of a modified Aubin-Lions Theorem  \ref{thmBbL}.  Such a strategy is inspired from \cite{MuhaSch, malSpSch}.  Since the arguments we use differ considerably from the ones used in  \cite{MuhaSch, malSpSch}, we would prefer to present a detailed proof.\\
	For further use let us now introduce the notation $(\cdot)_{h}$ to denote time averaged quantities, to be more specific we denote:
	\begin{equation}\label{tmavg}
		\begin{array}{ll}
			&\displaystyle (u^{h}_{i})_{h}(t)=\frac{1}{h}\int^{t}_{t-h}u^{h}_{i}(\cdot,s)ds,\quad (\partial_{t}\eta^{h})_{h}(t)=\frac{1}{h}\int^{t}_{t-h}\partial_{t}\eta^{h}(\cdot,s)ds.
		\end{array}
	\end{equation} 
	Note that $(u^{h}_{i})_{h}$ and $(\partial_{t}\eta^{h})_{h}$ have similar convergence properties as that of $u^{h}_{i}$ and $\partial_{t}\eta^{h}$ respectively which are listed in \eqref{weakconvhlev}.
	In order to state the next result we introduce the notation $(\partial_{t}\eta^{h})_{\delta}$ to denote a smooth approximation (which is constructed by a standard mollification) of the structural velocity $\partial_{t}\eta^{h}.$ At this moment one further recalls the smooth solinoidal extension operator $E_{i,\eta,\delta}$ introduced in Corollary \ref{corextension}. 
	\begin{prop}\label{strngconvavg}
		Upto a subsequence one has the following
		\begin{equation}\label{convdual}
			\begin{array}{ll}
				&\displaystyle \int^{T}_{0}\|(\partial_{t}\eta^{h})_{h}\|^{2}_{L^{2}(\Gamma)}+\int^{T}_{0}\|(u^{h}_{i})_{h}\|^{2}_{L^{2}(\Omega^{i,h})}\\
				&\displaystyle =\int^{T}_{0}\langle (\partial_{t}\eta^{h})_{h},(\partial_{t}\eta^{h})_{h} \rangle_{\Gamma}+\int^{T}_{0}\langle (u^{h}_{i})_{h},(u^{h}_{i})_{h} \rangle_{\Omega^{i,h}} \longrightarrow \int^{T}_{0}\|\partial_{t}\eta\|^{2}_{L^{2}(\Gamma)}+\int^{T}_{0}\|u_{i}\|^{2}_{L^{2}(\Omega^{i})}.
			\end{array}
		\end{equation}
		Further, the same subsequence may be assumed to satisfy $(\partial_{t}\eta^{h})_{h}\longrightarrow \partial_{t}\eta$ strongly in $L^{2}(\Gamma\times(0,T))$ and $(u^{h}_{i})_{h}\longrightarrow u_{i}$ in $L^{2}(\Omega^{i,h}\times(0,T)).$\\
	\end{prop}
	\begin{proof}
		Let us first observe that
		\begin{equation}\label{I1pl2}
			\begin{array}{ll}
				&\displaystyle \int^{T}_{0}\langle (\partial_{t}\eta^{h})_{h},(\partial_{t}\eta^{h})_{h} \rangle_{\Gamma}+\int^{T}_{0}\langle (u^{h}_{i})_{h},(u^{h}_{i})_{h} \rangle_{\Omega^{i,h}} \\
				&\displaystyle= \int^{T}_{0}\bigg(\langle (\partial_{t}\eta^{h})_{h},(\partial_{t}\eta^{h})_{h}\rangle_{\Gamma}+\bigg\langle (u^{h}_{i})_{h},\mathcal{F}^{div}_{i,\eta^{h}}\bigg((\partial_{t}\eta^{h})_{h}-\mathcal{K}_{i,\eta^{h}}((\partial_{t}\eta^{h})_{h})\bigg)\bigg\rangle_{\Omega^{i,h}}\bigg)\\
				&\displaystyle \qquad+\int^{T}_{0}\bigg\langle (u^{h}_{i})_{h},(u^{h}_{i})_{h}-\mathcal{F}^{div}_{i,\eta^{h}}\bigg((\partial_{t}\eta^{h})_{h}-\mathcal{K}_{i,\eta^{h}}((\partial_{t}\eta^{h})_{h})\bigg)\bigg\rangle_{\Omega^{i,h}}\\
				&\displaystyle =\mathbb{I}_{1}+\mathbb{I}_{2}
			\end{array}
		\end{equation}
		\textit{Convergence of $\mathbb{I}_{1}:$}\\
		To this end we will use Theorem \ref{thmBbL}. For $\mathbb{I}_{1}$ we choose
		\begin{equation}\label{fgh}
			\begin{array}{ll}
				&\displaystyle f_{h}=\left((\partial_{t}\eta^{h})_{h}-\mathcal{K}_{i,\eta^{h}}((\partial_{t}\eta^{h})_{h}),\mathcal{F}^{div}_{i,\eta^{h}}\bigg((\partial_{t}\eta^{h})_{h}-\mathcal{K}_{i,\eta^{h}}((\partial_{t}\eta^{h})_{h})\bigg)\right),\\ &\displaystyle g_{h}=\left((\partial_{t}\eta^{h})_{h},(u^{h}_{i})_{h}\right).
			\end{array}
		\end{equation}
		Note that $\mathcal{F}^{div}_{i,\eta^{h}}\bigg((\partial_{t}\eta^{h})_{h}-\mathcal{K}_{i,\eta^{h}}((\partial_{t}\eta^{h})_{h})\bigg)$ is a solenoidal extension of $(\partial_{t}\eta^{h})_{h}-\mathcal{K}_{i,\eta^{h}}((\partial_{t}\eta^{h})_{h}).$\\
		Now to proceed further one defines
		$$X=L^{2}(\Gamma)\times W^{-s,2}(\Omega^{i}_{m,M}),\,\,
		X'=L^{2}(\Gamma)\times W^{s,2}(\Omega^{i}_{m,M})$$
		and
		$$Z=H^{s_{0}}(\Gamma)\times H^{s_{0}}(\Omega^{i}_{m,M}),$$
		where we choose $0<s<s_{0}<\frac{1}{4},$ while for integrability in time we restrict ourselves to $L^{2}.$ We further define
		$$f_{h,\delta}=\bigg(((\partial_{t}\eta^{h})_{h})_{\delta}-\mathcal{K}_{i,\eta^{h}}(((\partial_{t}\eta^{h})_{h})_{\delta}),E_{i,\eta^{h},\delta}((\partial_{t}\eta^{h})_{h})\bigg),$$
		where one recalls that 
		$$E_{i,\eta^{h},\delta}((\partial_{t}\eta^{h})_{h})=\mathcal{F}^{div}_{i,\eta^{h}}\bigg(((\partial_{t}\eta^{h})_{h})_{\delta}-\mathcal{K}_{i,\eta^{h}}((\partial_{t}\eta^{h})_{h})_{\delta}\bigg)$$
		The first assumption of Theorem \ref{thmBbL} follows by weak compactness in Hilbert spaces. In particular one uses the strong convergence \eqref{weakconvhlev}$_{2}$ of $\eta^{h}$ and the structure \eqref{consf2eta}-\eqref{deff2eta} of $\mathcal{F}^{div}_{2,\eta}$ (similarly of $\mathcal{F}^{div}_{1,\eta}$) to furnish the required weak convergences.\\
		Next the second assumption of Theorem \ref{thmBbL} follows by using the linearity of the maps $\mathcal{K}_{i,\eta},$ $\mathcal{F}^{div}_{i,\eta}$ the Corollary \ref{corextension} and a standard estimate for mollification:
		\begin{equation}\label{difffhdel}
			\begin{array}{l}
				\displaystyle \|f_{h}-f_{h,\delta}\|_{L^{2}}\leqslant c\|(\partial_{t}\eta^{h})_{h}*\psi_{\delta}-(\partial_{t}\eta^{h})_{h}\|_{L^{2}(\Gamma)}\leqslant c\delta^{s}\|(\partial_{t}\eta^{h})_{h}\|_{W^{s,2}(\Gamma)}\leqslant  c\delta^{s}\|\partial_{t}\eta^{h}\|_{W^{s,2}(\Gamma)},
			\end{array}
		\end{equation}
		where $\psi_{\delta}$ is the mollifier function.\\
		Finally we check the third assumption of Theorem \ref{thmBbL}. To that end we compute the following for $0<t<\sigma<t+\sigma_{0}$
		\begin{equation}\label{diffgnts}
			\begin{array}{ll}
				&\displaystyle \bigg|\langle g_{h}(t)-g_{h}(\sigma),f_{h,\delta}(t)\rangle\bigg|\\
				&\displaystyle\leqslant \bigg|\langle\int^{\sigma}_{t}\partial_{t}g_{h}(s)ds, f_{h,\delta}(t)\rangle\bigg|\\
				&\displaystyle \leqslant\bigg|\int^{\sigma}_{t}\langle \partial_{t}(\partial_{t}\eta^{h})_{h}(s),(((\partial_{t}\eta^{h})_{h})_{\delta}-\mathcal{K}_{i,\eta^{h}}(((\partial_{t}\eta^{h})_{h})_{\delta}))(t)\rangle_{\Gamma}ds\\
				&\displaystyle\qquad+\int^{\sigma}_{t}\langle \partial_{t}(u^{h}_{i})_{h}(s),E_{i,\eta^{h}(s),\delta}((\partial_{t}\eta^{h})_{h}(t))\rangle_{\Omega^{i}_{m,M}} ds\bigg|\\
				&\displaystyle \qquad+\bigg|\int^{\sigma}_{t}\langle \partial_{t}(u^{h}_{i})_{h}(s),E_{i,\eta^{h}(s),\delta}((\partial_{t}\eta^{h})_{h}(t))-E_{i,\eta^{h}(t),\delta}((\partial_{t}\eta^{h})_{h}(t)) \rangle_{\Omega^{i}_{m,M}}\bigg|\\
				&\displaystyle =\mathcal{A}_{1}+\mathcal{A}_{2}.
			\end{array}
		\end{equation}
		In the above calculation we have followed our usual convention and extended $\partial_{t}(u^{h}_{i})_{h}$ by zero outside $\Omega^{i,h}.$ In form of the next lemma we derive an estimate for $\mathcal{A}_{1}.$ which is essential to verify the third assumption of Theorem \ref{thmBbL}. 
		\begin{lem}
			There exists a constant $C(\delta)>0$ depending only on the final time $T$ and the initial data such that
			\begin{equation}\label{estA1}
				\begin{array}{ll}
					&\displaystyle	\mathcal{A}_{1}=\bigg|\int^{t}_{\sigma}\bigg(\langle \partial_{t}(\partial_{t}\eta^{h})_{h}(s), (((\partial_{t}\eta^{h})_{h})_{\delta}-\mathcal{K}_{i,\eta^{h}}(((\partial_{t}\eta^{h})_{h})_{\delta}))(t)\rangle_{\Gamma}\\
					&\displaystyle\qquad+\langle \partial_{t}(u^{h}_{i})_{h}(s),E_{i,\eta^{h}(s),\delta}((\partial_{t}\eta^{h})_{h}(t))\rangle_{\Omega^{i}_{m,M}} ds\bigg)\bigg|\\
					&\displaystyle \qquad\leqslant C(\delta)|t-\sigma|^{\frac{1}{4}}
				\end{array}
			\end{equation}
			for $\sigma_{0}$ sufficiently small.
		\end{lem}
		\begin{proof}
			Let us first observe that
			\begin{equation}\label{caldual}
				\begin{array}{ll}
					&\displaystyle \mathcal{I}=\int^{t}_{\sigma}\bigg(\langle \partial_{t}(\partial_{t}\eta^{h})_{h}(s),(((\partial_{t}\eta^{h})_{h})_{\delta}-\mathcal{K}_{i,\eta^{h}}(((\partial_{t}\eta^{h})_{h})_{\delta}))(t)\rangle_{\Gamma}\\[3.mm]
					&\displaystyle\qquad\qquad+\langle \partial_{t}(u^{h}_{i})_{h}(s),E_{i,\eta^{h},\delta}((\partial_{t}\eta^{h})_{h}(t))\rangle_{\Omega^{i}_{m,M}}\\
					&\displaystyle \qquad -\int_{\Omega^{i,h}}\frac{1}{h}\bigg(u^{h}_{i}(s-h)-u^{h}_{i}(s-h)\circ\Phi^{i,h}_{-h}(s)\bigg)E_{i,\eta^{h}(s),\delta}((\partial_{t}\eta^{h})_{h}(t))
					\bigg)ds\\[3.mm]
					&\displaystyle =\int^{t}_{\sigma}\bigg(\int_{\Gamma}\bigg(\frac{\partial_{t}\eta^{h}(s)-\partial_{t}\eta^{h}(s-h)}{h}\bigg)(((\partial_{t}\eta^{h})_{h})_{\delta}-\mathcal{K}_{i,\eta^{h}}(((\partial_{t}\eta^{h})_{h})_{\delta}))(t)\\
					&\displaystyle\qquad+\int_{\Omega^{i,h}}\bigg(\frac{u^{h}_{i}(s)-u^{h}_{i}(s-h)\circ\Phi^{i,h}_{-h}(t)}{h}E_{i,\eta^{h}(s),\delta}((\partial_{t}\eta^{h})_{h}(t))\bigg)\bigg)ds\\[3.mm]
					
				\end{array}
			\end{equation}
			Now by using the momentum balance \eqref{mombalhapp}, more specifically using $((((\partial_{t}\eta^{h})_{h})_{\delta}-\mathcal{K}_{i,\eta^{h}}(((\partial_{t}\eta^{h})_{h})_{\delta}))(t)\\
			,E_{i,\eta^{h}(s),\delta}((\partial_{t}\eta^{h})_{h}(t)))$ as a test function in \eqref{mombalhapp}, one furnishes the following from \eqref{caldual}
			\begin{equation}\label{frommomcal}
				\begin{array}{ll}
					\displaystyle \mathcal{I}
					\displaystyle =&\displaystyle-\int^{t}_{\sigma}\bigg(\langle DK_{\kappa}(\eta^{h}(s)),(((\partial_{t}\eta^{h})_{h})_{\delta}-\mathcal{K}_{i,\eta^{h}}(((\partial_{t}\eta^{h})_{h})_{\delta}))(t) \rangle_{\Gamma}\\
					&\displaystyle+\kappa\int_{\Gamma}\nabla^{k_{0}+1}\partial_{t}\eta^{h}(s)\nabla^{k_{0}+1}(((\partial_{t}\eta^{h})_{h})_{\delta}-\mathcal{K}_{i,\eta^{h}}((\partial_{t}\eta^{h})_{h})_{\delta})(t)\\
					&\displaystyle +\kappa\sum_{i}\int_{\Omega^{i,h}}\nabla^{k_{0}}u^{h}_{i}(\cdot,s)\nabla^{k_{0}}(E_{i,\eta^{h}(s),\delta}(\partial_{t}\eta^{h})_{h}(t))\\
					&\displaystyle +\sum_{i}\int_{\Omega^{i,h}}\mu_{i}(\theta_{i}^{h})Du^{h}_{i}(\cdot,s):D(E_{i,\eta^{h}(s),\delta}(\partial_{t}\eta^{h})_{h}(t))\bigg)ds.
				\end{array}
			\end{equation}
			Now using bounds of $\eta^{h},$ $u^{h}_{i}$ and $\theta^{h}_{i}$ obtained from \eqref{enbalancehlvl} - \eqref{disstypesthlev} and further applying \eqref{higherderest1*} in order to estimate $\nabla^{k_{0}}(E_{i,\eta^{h}(s),\delta}(\partial_{t}\eta^{h})_{h}(t))$ (note that we have considered $\xi(s)$ as $((\partial_{t}\eta^{h})_{h})_{\delta}(t)$ in \eqref{higherderest1*}, which is a constant function in $s$ variable), one infers
			\begin{equation}\label{estmcI}
				\begin{array}{l}
					\displaystyle |\mathcal{I}|\leqslant C(\delta)|t-\sigma|^{\frac{1}{2}}\|(\partial_{t}\eta^{h})_{h}(t)\|_{L^{2}(\Gamma)}\leqslant C(\delta)|t-\sigma|^{\frac{1}{2}}\|\partial_{t}\eta^{h}(t)\|_{L^{2}(\Gamma)}\leqslant C(\delta)|t-\sigma|^{\frac{1}{2}}.
				\end{array}
			\end{equation}
			Now  in what follows we estimate the third term in the expression of $\mathcal{I}:$
			\begin{equation}\label{estthrdtrmI}
				\begin{array}{ll}
					&\displaystyle \bigg|\int^{t}_{\sigma}\int_{\Omega^{i,h}}\frac{1}{h}\bigg(u^{h}_{i}(s-h)-u^{h}_{i}(s-h)\circ\Phi^{i,h}_{-h}(s)\bigg)E_{i,\eta^{h}(s),\delta}((\partial_{t}\eta^{h})_{h}(t))ds\bigg|\\
					&\displaystyle =\int^{t}_{\sigma}\bigg|\int_{\Omega^{i,h}(s-h)}u^{h}_{i}(s-h),\frac{E_{i,\eta^{h}(s),\delta}((\partial_{t}\eta^{h})_{h}(t))-E_{i,\eta^{h}(s),\delta}((\partial_{t}\eta^{h})_{h}(t))\circ\Phi^{i,h}_{h}(s-h)}{h}\bigg|ds\\[3.mm]
					&\displaystyle \leqslant C\int^{t}_{\sigma}\|u^{h}_{i}(s)\|_{L^{4}(\Omega^{i,h})}\bigg\|\frac{E_{i,\eta^{h}(s),\delta}((\partial_{t}\eta^{h})_{h}(t))-E_{i,\eta^{h}(s),\delta}((\partial_{t}\eta^{h})_{h}(t))\circ\Phi^{i,h}_{h}(s-h)}{h}\bigg\|_{L^{\frac{4}{3}}(\Omega^{i,h})}ds\\[3.mm]
					&\displaystyle \leqslant C\int^{t}_{\sigma}\|u^{h}_{i}(s)\|_{L^{4}(\Omega^{i,h})}\|\nabla E_{i,\eta^{h}(s),\delta}((\partial_{t}\eta^{h})_{h}(t))\|_{L^{4}(\Omega^{i,h})}ds\\[3.mm]
					&\displaystyle \leqslant C\|u^{h}_{i}(s)\|_{L^{2}(0,T;L^{4}(\Omega^{i,h}))}\bigg(\int^{t}_{\sigma}\|\nabla E_{i,\eta^{h}(s),\delta}((\partial_{t}\eta^{h})_{h}(t))\|^{2}_{L^{4}(\Omega^{i,h})}ds\bigg)^{\frac{1}{2}}\\
					&\displaystyle \leqslant C(\delta)|t-\sigma|^{\frac{1}{4}},
				\end{array}
			\end{equation}
			where in the above calculation to get the fourth line from the third we have used the first inequality of \eqref{nxtestdiffmean} with $p=\frac{4}{3}.$ The final line of \eqref{estthrdtrmI} from the penultimate one follows by using \eqref{Fdivetaest1}$_{1}.$\\
			Note that in the previous calculation it is important to use the fact that $(\partial_{t}\eta^{h})_{h}(t)$ is independent of the $s$ variable.\\ 
			Finally in view of \eqref{estmcI} and \eqref{estthrdtrmI} one concludes the proof of \eqref{estA1}.
		\end{proof}
		Let us now estimate the term $\mathcal{A}_{2}$ appearing in \eqref{diffgnts}.
		\begin{lem}
			There exists a constant $C(\delta)>0$ depending only on the final time $T$ and the initial data such that
			\begin{equation}\label{estA2}
				\begin{array}{ll}
					\displaystyle	\mathcal{A}_{2}&\displaystyle=\bigg|\int^{\sigma}_{t}\langle \partial_{t}(u^{h}_{i})_{h}(s),E_{i,\eta^{h}(s),\delta}((\partial_{t}\eta^{h})_{h}(t))-E_{i,\eta^{h}(t),\delta}((\partial_{t}\eta^{h})_{h}(t)) \rangle_{\Omega^{i}_{m,M}}\bigg|\\
					&\displaystyle \leqslant C(\delta)|t-\sigma|^{\frac{1}{2}}
				\end{array}
			\end{equation}
			for $\sigma_{0}$ sufficiently small.
		\end{lem}
		\begin{proof}
			Let us estimate $\mathcal{A}_{2}$ as follows
			\begin{equation}\label{estA2comm}
				\begin{array}{ll}
					\displaystyle	\mathcal{A}_{2}&\displaystyle=\bigg|\int^{\sigma}_{t}\langle \partial_{t}(u^{h}_{i})_{h}(s),E_{i,\eta^{h}(s),\delta}((\partial_{t}\eta^{h})_{h}(t))-E_{i,\eta^{h}(t),\delta}((\partial_{t}\eta^{h})_{h}(t)) \rangle_{\Omega^{i}_{m,M}}\bigg|\\[3.mm]
					&\displaystyle =\bigg|\int^{\sigma}_{t}\langle \partial_{t}(u^{h}_{i})_{h}(s),\int^{s}_{t}\partial_{\alpha}E_{\eta^{h}(\alpha),\delta}((\partial_{t}\eta^{h})_{h}(t))d\alpha\rangle ds\bigg|\\[3.mm]
					&\displaystyle \leqslant \bigg|\langle (u^{h}_{i})_{h}(\sigma),\int^{\sigma}_{t}\partial_{\alpha}E_{\eta^{h}(\alpha),\delta}((\partial_{t}\eta^{h})_{h}(t))d\alpha\rangle\bigg|+\bigg|\int^{\sigma}_{t}\langle (u^{h}_{i})_{h}(s),\partial_{s}E_{\eta^{h}(s)}((\partial_{t}\eta^{h})_{h}(t))\rangle ds\bigg|\\[3.mm]
					&\displaystyle\leqslant |t-\sigma|^{\frac{1}{2}}\|(u^{h}_{i})_{h}\|_{L^{\infty}(L^{2}(\Omega^{i}_{m,M}))}\|\partial_{s}E_{\eta^{h}(s)}(\partial_{t}\eta^{h})_{h}(t)\|_{L^{2}(L^{2}(\Omega^{i}_{m,M}))}\\[3.mm]
					&\displaystyle \leqslant C(\delta)|t-\sigma|^{\frac{1}{2}},
				\end{array}
			\end{equation}
			
			where in the previous computation, the forth line from the third follows since
			\begin{equation}\label{estdalphaE}
				\begin{array}{ll}
					&\displaystyle \bigg\|\int^{\sigma}_{t}\partial_{\alpha} E_{\eta^{h}(\alpha)}((\partial_{t}\eta^{h})_{h}(t))d\alpha\bigg\|^{2}_{L^{2}(\Omega^{i}_{m,M})}\\[4.mm]
					&\displaystyle =\int_{\Omega^{i}_{m,M}}\bigg|\int^{\sigma}_{t}\partial_{\alpha}E_{\eta^{h}(\alpha)}(\partial_{t}\eta^{h})_{h}(t)d\alpha\bigg|^{2}\\[3.mm]
					&\displaystyle =|t-\sigma|^{2}\int_{\Omega^{i}_{m,.M}}\bigg|\frac{1}{|t-\sigma|}\int^{\sigma}_{t}\partial_{\alpha}E_{\eta^{h}(\alpha)}(\partial_{t}\eta^{h})_{h}(t)d\alpha\bigg|^{2}\\[3.mm]
					&\displaystyle \leqslant|t-\sigma|\int_{\Omega^{i}_{m,M}}\int^{\sigma}_{t}|\partial_{\alpha}E_{\eta^{h}(\alpha)}(\partial_{t}\eta^{h})_{h}(t)|^{2}d\alpha,
				\end{array}
			\end{equation}
			please note that we have used Jensen's inequality above.\\
			The final line of \eqref{estA2comm} from the penultimate one is a consequence of the estimates \eqref{disstypesthlev} and \eqref{estEetadel}.\\
			This finishes the proof of \eqref{estA2}.
		\end{proof}
		Now in view of \eqref{diffgnts}, \eqref{estA1} and \eqref{estA2} one concludes that for $\sigma_{0}$ sufficiently small the following holds
		\begin{equation}\label{esteqcont}
			\begin{array}{l}
				\displaystyle  \bigg|\langle g_{h}(t)-g_{h}(\sigma),f_{h,\delta}(t)\rangle\bigg|\leqslant C(\delta)|t-\sigma|^{\frac{1}{4}}
			\end{array}
		\end{equation}
		Now in order to prove the follwoing by an application of Theorem \ref{thmBbL}
		\begin{equation}\label{strngconvI1}
			\begin{array}{l}
				\displaystyle	\mathbb{I}_{1}\longrightarrow \int^{T}_{0}\bigg(\|\partial_{t}\eta\|^{2}_{L^{2}(\Gamma)}+\langle u_{i},\mathcal{F}^{div}_{i,\eta}(\partial_{t}\eta)\rangle_{\Omega^{i}}\bigg).
			\end{array}
		\end{equation}
		one needs to show that 
		\begin{equation}\label{Keta0}
			\begin{array}{l}
				\mathcal{K}_{i,\eta}(\partial_{t}\eta)=0.
			\end{array}
		\end{equation}
		To that end, one first uses $\mbox{div}\,\mathcal{F}^{div}_{i,\eta}(\partial_{t}\eta-\mathcal{K}_{i,\eta}(\partial_{t}\eta))=0$ to have
		\begin{equation}\label{gaussapp}
			\begin{array}{ll}
				&\displaystyle \mathcal{K}_{i,\eta}(\partial_{t}\eta)\int_{\partial\Omega^{2,h}}\nu(\varphi^{-1}(\pi(x)))\cdot \nu_{\eta(x)}dx=\int_{\partial\Omega^{2,h}} u_{i}\cdot\nu_{\eta}(x)dx=0
			\end{array}
		\end{equation}
		since $\div\,u_{i}=0$ and $u_{i}(x)=\partial_{t}\eta\nu(\varphi^{-1}(\pi(x)))$ for $x\in \partial\Omega^{2,h}.$\\
		Since $\eta(\cdot)$ can be considered as  a graph over the surface $\partial\Omega$ which has a well defined tangent plane almost everywhere, we conclude that
		$$\nu(\varphi^{-1}(\pi(x)))\cdot \nu_{\eta}(x)>0\,\,\mbox{a.e.}$$
		In view of \eqref{gaussapp} the last inequality infers \eqref{Keta0}.
		Now let us show the convergence of $\mathbb{I}_{2}.$ This finishes the proof of \eqref{strngconvI1}. \\
		Next we focus in proving the convergence of $\mathbb{I}_{2}$ by using Theorem \ref{thmBbL}.\\
		\textit{Convergence of $\mathbb{I}_{2}:$}\\
		Let us set $g_{h}=(u^{h}_{i})_{h}\mathcal{X}_{\Omega^{i,h}}$ and $f_{h}=\bigg((u^{h}_{i})_{h}-\mathcal{F}^{div}_{i,\eta^{h}}\bigg((\partial_{t}\eta^{h})_{h}-\mathcal{K}_{i,\eta^{h}}((\partial_{t}\eta^{h})_{h})\bigg)\bigg)\mathcal{X}_{\Omega^{i,h}}.$ We further set $X=W^{-s,2}(\Omega^{i}_{m,M})$ and consequently $X'=W^{s,2}(\Omega^{i}_{m,M})$ for some $s\in(0,\frac{1}{4}).$ Note that as usual we will extend the functions to $\Omega^{i}_{m,M}$ by defining them zero outside $\Omega^{i,h}.$ Further in time we consider our functions to be square integrable.\\
		Now the assumptions $(1)$ and $(4)$ of Theorem \ref{thmBbL} follows by the standard compactness arguments. In particular one applies \cite[Prop. 2.28]{LeRu14} to furnish the uniform $W^{s,2}(\Omega^{i}_{m,M})$ ($s\leqslant\frac{1}{4}$) bound of $g_{h}.$\\
		For the proof of item $(2)$ one uses similar arguments as the ones used in showing \eqref{difffhdel}. In particular since $(u^{h}_{i})_{h}$ is extended by zero outside $\Omega^{i,h}$ to define a function on $\Omega^{i}_{m,M},$ we have enough room for mollification.\\ 
		In view of \eqref{weakconvhlev}$_{2}$ we can find an $h_{\delta}>0$ and $\tau_{\delta}>0$ small enough such that
		\begin{equation}\label{diffetaetah}
			\begin{array}{l}
				\displaystyle \sup_{h\leqslant h_{\delta}}\sup_{\tau\in(t-\tau_{\delta},t+\tau_{\delta})\cap[0,T]}\|\eta(x,t)-\eta^{h}(x,t)\|_{\infty}\leqslant \delta.
			\end{array}
		\end{equation}
		Second, let us fix, $0<s_{0}<s$ and $\epsilon>0.$ Using \cite[Lemma A.13]{LeRu14}, there exist a $\delta_{\epsilon}$ and a sequence $\widetilde{f_{h,\delta}}$ such that
		$$\displaystyle \mbox{supp}\,(\widetilde{f_{h,\delta}}(t))\subset\Omega^{i}_{\eta^{h}(t)-3\delta},$$
		(one recalls that $\Omega^{i}_{\eta^{h}(t)}\equiv\Omega^{i,h}$) for all $3\delta\leqslant \delta_{\epsilon},$ that is divergence free and  
		\begin{equation}\label{difffhfhdel}
			\begin{array}{l}
				\displaystyle\|f_{h}-\widetilde{f_{h,\delta}}\|_{W^{-s_{0},2}(\Omega^{i}_{m,M})}\leqslant \epsilon\|f_{h}\|_{L^{2}(\Omega^{i}_{m,M})}.
			\end{array}
		\end{equation}
		Next one mollifies the solinoidal function $\widetilde{f_{h,\delta}}$ to define
		\begin{equation}\label{molification}
			\begin{array}{l}
				\displaystyle f_{h,\delta}=\widetilde{f_{h,\delta}}*\psi_{\delta},
			\end{array}
		\end{equation}
		where $\psi$ is the standard mollifier in space. In the direction of showing item $(3)$ of Theorem \ref{thmBbL}, one computes the following by using \eqref{difffhfhdel}
		\begin{equation}\label{diffestfhdel}
			\begin{array}{ll}
				&\displaystyle \|f_{h}-f_{h,\delta}\|_{W^{-s,2}(\Omega^{i}_{m,M})}\\
				&\displaystyle \leqslant \|\widetilde{f_{h,\delta}}-f_{h,\delta}\|_{W^{-s,2}(\Omega^{i}_{m,M})}+\|\widetilde{f_{h,\delta}}-f_{h}\|_{W^{s_{0},2}(\Omega^{i}_{m,M})}\\
				&\displaystyle \leqslant C\delta^{s-s_{0}}\|\widetilde{f_{h,\delta}}\|_{W^{s_{0},2}(\Omega^{i}_{m,M})}+\epsilon\|f_{h}\|_{L^{2}(\Omega^{i}_{m,M})}\\
				&\displaystyle \leqslant C\epsilon \|f_{h}\|_{L^{2}(\Omega^{i}_{m,M})}
			\end{array}
		\end{equation}
		for $\delta$ small enough independent of $(s-s_{0}).$ Further in view of \eqref{diffetaetah} and properties of mollification we have that $f_{h,\delta}$ is compactly supported in $\Omega^{i,h}.$ This enables us to use the pair $(0,f_{h,\delta})$ (note that $f_{h,\delta}$ is solenoidal by construction) as  a test function in the momentum balance equation \eqref{mombalhapp} over the interval $[t,t+\tau]$ for $t\in[0,T-\tau].$ Hence item $(3)$ of Theorem \ref{thmBbL} follows by mimicing the line of arguments used in proving \eqref{estA1}.\\
		This finishes the proof of the fact that (indeed one uses \eqref{Keta0})
		\begin{equation}\label{convI2}
			\begin{array}{l}
				\displaystyle \mathbb{I}_{2}\longrightarrow \int^{T}_{0}\langle u_{i},u_{i}-\mathcal{F}^{div}_{i,\eta}(\partial_{t}\eta)\rangle_{\Omega^{i}}.
			\end{array}
		\end{equation}
		Finally \eqref{strngconvI1} and \eqref{convI2} together infers \eqref{convdual}.\\
		The convergence \eqref{convdual} along with the uniform convexity of $L^{2}$ implies that $(\partial_{t}\eta^{h})_{h}\longrightarrow \partial_{t}\eta$ strongly in $L^{2}(\Gamma\times(0,T))$ and $(u^{h}_{i})_{h}\longrightarrow u_{i}$ in $L^{2}(\Omega^{i,h}\times(0,T)).$
	\end{proof}
	\subsection{Choice of test functions and passage to the limit in the momentum equation}\label{constructiontesthlev}
	Indeed the test function pairs $(\psi,b^{h})$ used in the approximate layer ($cf.$ \eqref{mombalhapp}) solve a compatibility condition of the form $\psi\circ\varphi_{\eta^{h}}=b^{h}$ on $\Sigma_{\eta^{h}}$ but they might not solve the same on the limiting interface $\Sigma_{\eta}.$ Hence one needs to be careful while making the choice of suitable test function pairs for the momentum equation.\\
	In the spirit of the first part of Section \ref{passtau0}, entitled `\textit{Construction of test functions}' we can construct pairs $(\psi^{h},b^{h})$ and $(\psi,b)$ such that the first can be used as a test function in \eqref{mombalhapp} and the later in the limiting momentum balance. Furthermore, one has 
	\begin{equation}\label{constestdeltalev}
		\begin{array}{ll}
			&\displaystyle 	b^{h}=\Tr_{\Sigma_{\eta^{h}}}\psi\cdot\nu,\quad 	b=\Tr_{\Sigma_{\eta}}\psi\cdot\nu,\\
			&\displaystyle 	b^{h}\rightarrow b\,\,\mbox{in}\,\, L^{\infty}(0,T;W^{k_{0}+1,2}(\Gamma))
		\end{array}
	\end{equation}
	\subsubsection{Passage to the limit in momentum balance}\label{ctestfn}
	Next we use the pair of test function $(b^{h},\psi^{h})$ in the discrete momentum balance equation \eqref{mombalhapp} and proceed towards the limit passage $h\rightarrow 0.$\\
	In view of the strong convergence and the structure \eqref{constestdeltalev} of $b^{h}$ and the weak convergence \eqref{weakconvhlev}$_{3}$ one obtains the following
	\begin{equation}\label{convdbltmder}
		\begin{array}{l}
			\displaystyle	\int^{t}_{0}\int_{\Gamma}\frac{\partial_{t}\eta^{h}(\cdot)-\partial_{t}\eta^{h}(\cdot-h)}{h}b^{h}=-\int^{t}_{0}\int_{\Gamma}\partial_{t}\eta^{h}(\cdot)\frac{b^{h}(\cdot)-b^{h}(\cdot-h)}{h}\rightarrow -\int^{t}_{0}\int_{\Gamma}\partial_{t}\eta(\cdot)\partial_{t}b(\cdot).
		\end{array}
	\end{equation}
	Once again using the fact that
	$$\eta^{h}\rightarrow \eta\,\,\mbox{in}\,\, L^{\infty}(0,T;W^{2+,2}(\Gamma))$$
	(which follows from \eqref{weakconvhlev}$_{1,3}$ ) 
	one renders the following
	\begin{equation}\label{convDKlevh}
		\begin{array}{l}
			\displaystyle \int^{t}_{0}\langle DK_{\kappa}(\eta^{h}),b^{h}\rangle \rightarrow \int^{t}_{0}\langle DK_{\kappa}(\eta),b\rangle.
		\end{array}
	\end{equation}
	Indeed in order to show \eqref{convDKlevh}, one exploits the structure of $DK_{\kappa},$ cf. \eqref{elasticityoperator},\eqref{amab},\eqref{Geta},\eqref{Gij'},\eqref{exaR} and \eqref{DefDk}.\\
	Next in order to show the convergence 
	\begin{equation}\label{convtmderu}
		\begin{array}{ll}
			&\displaystyle \int^{t}_{0}\int_{\Omega^{i,h}}\frac{u^{h}_{i}(\cdot)-u^{h}_{i}(\cdot-h)\circ\Phi^{i,h}_{-h}}{h}\psi^{h}\rightarrow -\int^{t}_{0}\left\langle  u_{i},\partial_{t}\psi+u_{i}\cdot\nabla\psi\right\rangle_{\Omega^{i}},
		\end{array}
	\end{equation}
	one can imitate the arguments used in \cite[Proof of Theorem 1.2, Section 4]{BenKamSch} with no further difficulties. Indeed to implement the arguments from \cite{BenKamSch}, one needs the convergence
	\begin{equation}\label{convcutoff}
		\begin{array}{ll}
			\displaystyle \int^{t}_{0}\langle (u^{h}_{i})_{\delta},A^{i}_{\delta}(u^{h}_{i})_{h}\rangle_{\Omega^{h}_{i}}\rightarrow \int^{t}_{0}\langle (u_{i})_{\delta},A^{i}_{\delta}u_{i}\rangle_{\Omega^{i}}
		\end{array}
	\end{equation}
	as $h\rightarrow 0$ where $A^{i}_{\delta}\in C^{0}([0,T];C^{\infty}_{c}(\Omega^{i}_{h})),$ for $h>0$ sufficiently small and satisfies $A^{i}_{\delta}(t)\rightarrow_{\delta}\mathcal{X}_{\Omega^{i}(t)}$ almost everywhere. Further note that the notation $(u^{h}_{i})_{\delta}$ in \eqref{convcutoff} denotes the mollification in space of $u^{h}_{i}.$ In our case the convergence \eqref{convcutoff} is a consequence of the strong convergence $(u^{h}_{i})_{h}\rightarrow u_{i}$ in $L^{2}(\Omega^{i,h}\times(0,T))$ ($c.f.$ Proposition \ref{strngconvavg}) and the weak convergence \eqref{weakconvhlev}$_{4}.$\\
	Now we pass the limit into the sixth term of \eqref{mombalhapp}. In view of the positive lower bound of the temperature and the strong convergence \eqref{claimstrngconvuptobndrythetah}, one readily has that $\mu_{i}(\theta_{i}^{h})\chi_{\Omega^{i,h}}\in L^{\infty}(\Omega\times (0,T))$ and it converges $a.e.$ to $\mu_{i}(\theta_{i})\chi_{\Omega^{i}}.$ Hence in particular $\mu_{i}(\theta_{i}^{h})D\psi^{h}_{i}\mathcal{X}^{i,h}$ converges strongly to $\mu_{i}(\theta_{i})D\psi\mathcal{X}_{\Omega^{i}}$ in $L^{r}(\Omega\times(0,T))$ for any $1\leqslant r<\infty.$ Consequently using \eqref{weakconvhlev}$_{4}$ one obtains
	\begin{equation}\label{L2convDu2}
		\begin{array}{ll}
			\displaystyle	\int^{t}_{0}\int_{\Omega^{i,h}}\mu_{i}(\theta^{h}_{i})Du^{h}_{i}:D\psi^{h}\rightarrow  	\int^{t}_{0}\int_{\Omega^{i}}\mu_{i}(\theta_{i})Du_{i}:D\psi.
		\end{array}
	\end{equation}  
	Now in view of the convergences listed in \eqref{weakconvhlev} and further exploiting the linearity one can at once pass limit in the rest of the terms appearing in \eqref{mombalhapp}.\\
	This fimishes the proof of th identity \eqref{momentplvlh}.\\
	{\textit{A comment on the bound of $(\partial_{tt}\eta,\partial_{t}u_{i}+(u_{i}\cdot\nabla)u_{i})$}:} One may also construct suitable test functions in the spirit of Section \ref{constructiontesthlev}, and pass limit in  \eqref{diffquotest}, to have the following duality estimate of the distributional time derivative of $(\partial_{t}\eta,u_{i}):$
	\begin{equation}\label{dualesttmder}
		\begin{array}{l}
			\displaystyle\bigg \| \bigg( \partial_{tt}\eta, (\partial_{t}+u_{i}\cdot\nabla) u_{i} \bigg)\bigg\|_{\mathcal{Y}^{*}}\leqslant C,
		\end{array}
	\end{equation} 
	where $\mathcal{Y}$ is as introduced in \eqref{DefmY}, for some constant independent of $h,$ but may depend on the given data.
	\subsection{Passage to the limit in the heat evolution and obtainment of energy balance and entropy inequality}\label{heatrecovery}
	One can choose globally defined test functions $\zeta\in C^{\infty}_{c}([0,T];C^{\infty}_{c}(\overline{\Omega})),$ $\zeta\geqslant 0$ and use their restrictions $\zeta_{i}=\zeta\suchthat_{\Omega^{i,h}}$ as test functions in the approximated fluid domain $\Omega^{i,h}.$ Concerning the limiting heat equation we will use the same notation $\zeta_{i}$ to denote the restriction of $\zeta$ into $\Omega^{i}.$\\
	We can now pass $h\rightarrow 0$ into \eqref{aftrtau0heat4}, use \eqref{claimstrngconvuptobndrythetah} and \eqref{weakconvhlev}$_{4}$ to furnish
	\begin{equation}\label{aftrhlevheat}
		\begin{array}{ll}
			&\displaystyle c_{i}\int^{t_{1}}_{0}\frac{d}{dt}\int_{\Omega^{i}(t)}\theta_{i}(x,t)\zeta_{i}(x,t)-c_{i}\int^{t_{1}}_{0}\int_{\Omega^{i}(t)}\theta_{i}(x,t)(\partial_{t}+u_{i}\cdot\nabla)\zeta_{i}+k_{i}\int^{t_{1}}_{0}\int_{\Omega^{i}(t)}\nabla\theta_{i}\cdot\nabla\zeta_{i}\\
			&\displaystyle+\lambda\mathcal{I}_{i}\int^{t_{1}}_{0}\int_{\partial\Omega^{2}(t)}(\theta_{1}-\theta_{2})\zeta_{i}
			\displaystyle  =\liminf_{h\rightarrow 0}\left\langle \mathcal{D}^{h,i},\zeta_{i}\right\rangle\\
		\end{array}
	\end{equation}
	for $a.e.$ $t_{1}\in [0,h],$ where $\mathcal{I}_{1}=1,$ $\mathcal{I}_{2}=-1$ and the duality $\left\langle \mathcal{D}^{h,i},\zeta_{i}\right\rangle$ is introduced in \eqref{defHhi1zeta}.\\
	Next we are going to identify the limit of $\mathcal{D}^{h,i}$ with an integrable function.  
	\subsubsection{{\textit{Identification of $\lim_{h\rightarrow 0}\left\langle \mathcal{D}^{h,i},\zeta_{i}\right\rangle$ and energy balance:}}}\label{testsolhlev}
	In this section we derive an energy balance identity for the regularized system ($\kappa$ being the regularization parameter). We first use test function  $(b,\psi)=(\partial_{t}\eta,u_{i})$ in \eqref{momentplvlh} and next track the artificial dissipation terms by using the heat evolution equation. The use of such a test function is possible since $\partial_{t}\eta\in L^{2}(0,h;W^{k_{0}+1,2}(\Gamma)),$ which is crucial to give a sense to the duality pairing $\langle DK_{\kappa}(\eta),\partial_{t}\eta\rangle.$ In the aforementioned testing, one further uses \eqref{dualesttmder} to infer suitable meaning to $\displaystyle \bigg\langle  \bigg(\partial_{tt}\eta,\partial_{t}u_{i}+(u_{i}\cdot\nabla)u_{i}\bigg),(\partial_{t}\eta,u_{i})\bigg \rangle_{\mathcal{Y}^{*},\mathcal{Y}},$ where $\mathcal{Y}$ is introduced in \eqref{DefmY}.
	
	Consequently one arrives at the following identity involving \underline{\textit{dissipation}}:
	\begin{equation}\label{momenaftrtest}
		\begin{array}{ll}
			&\displaystyle K_{\kappa}(\eta(t))+\frac{1}{2}\int_{\Gamma}|\partial_{t}\eta(t)|^{2}+\frac{1}{2}\sum_{i}\int_{\Omega_{i}}|u_{i}(t)|^{2}\\[2.mm]
			&\displaystyle +\int^{t}_{0}\bigg(\kappa\int_{\Gamma}|\nabla^{k_{0}+1}\partial_{t}\eta|^{2}+\kappa\sum_{i}\int_{\Omega^{i}}|\nabla^{k_{0}}u_{i}|^{2}+\sum_{i}\int_{\Omega^{i}}\mu_{i}(\theta_{i})|Du_{i}|^{2}\bigg)\\[3.mm]
			&\displaystyle =K_{\kappa}(\eta_{0})+\frac{1}{2}\int_{\Gamma}|\eta^{0}_{1}|^{2}+\frac{1}{2}\sum_{i}\int_{\Omega^{i}_{0}}|u^{0}_{i}|^{2}.
		\end{array}
	\end{equation}
	On the other hand considering $\zeta_{i}=1$ in \eqref{aftrhlevheat}, summing over $i\in\{1,2\}$ and next adding the resulting expression with \eqref{momenaftrtest}, we furnish
	\begin{equation}\label{aftersummingheatandtestedmom2}
		\begin{array}{ll}
			&\displaystyle \sum_{i}c_{i}\int_{\Omega^{i}(t)}\theta_{i}(t)+K_{\kappa}(\eta(t))+
			\frac{1}{2}\int_{\Gamma}|\partial_{t}\eta(t)|^{2}+\frac{1}{2}\sum_{i}\int_{\Omega_{i}}|u_{i}(t)|^{2}\\
			&\displaystyle = K_{\kappa}(\eta_{0})+\sum_{i}\int_{\Omega^{i}_{0}}\theta^{0}_{i}+\frac{1}{2}\int_{\Gamma}|\eta^{0}_{1}|^{2}+\frac{1}{2}\sum_{i}\int_{\Omega^{i}_{0}}|u^{0}_{i}|^{2}
			\displaystyle +\sum_{i}\bigg(\liminf_{h\rightarrow 0}\left\langle \mathcal{D}^{h,i},1\right\rangle-\left\langle\mathcal{D}^{i},1\right\rangle\bigg)
		\end{array}
	\end{equation}
	where  $\left\langle \mathcal{D}^{h,i},1\right\rangle$ is as introduced in \eqref{defHhi1zeta} (corresponding to $\zeta_{i}=1$) and $\left\langle \mathcal{D}^{i},1\right\rangle$ is defined as follows
	\begin{equation}
		\begin{array}{l}
			\displaystyle	\left\langle \mathcal{D}^{i},1\right\rangle=\kappa\int^{t}_{0}\int_{\Gamma}|\nabla^{k_{0}+1}\partial_{t}\eta|^{2}+\kappa\int^{t}_{0}\int_{\Omega^{i}}|\nabla^{k_{0}}u_{i}|^{2}+\int^{t}_{0}\int_{\Omega^{i}}\mu_{i}(\theta_{i})|Du_{i}|^{2}.
		\end{array}
	\end{equation}
	Further  the identity \eqref{enbalancehlvl} along with the available convergences (in particular Proposition \ref{strngconvavg} ) and weak lower semi-continuity of convex functionals at once renders
	\begin{equation}\label{afterwklwrsemenergy}
		\begin{array}{ll}
			&\displaystyle \sum_{i}c_{i}\int_{\Omega^{i}(t)}\theta_{i}(t)+K_{\kappa}(\eta(t))+
			\frac{1}{2}\int_{\Gamma}|\partial_{t}\eta(t)|^{2}+\frac{1}{2}\sum_{i}\int_{\Omega_{i}}|u_{i}(t)|^{2}\\
			&\displaystyle \leqslant K_{\kappa}(\eta_{0})+\sum_{i}\int_{\Omega^{i}_{0}}\theta^{0}_{i}+\frac{1}{2}\int_{\Gamma}|\eta^{0}_{1}|^{2}+\frac{1}{2}\sum_{i}\int_{\Omega^{i}_{0}}|u^{0}_{i}|^{2}.
		\end{array}
	\end{equation}
	Now \eqref{aftersummingheatandtestedmom2} and \eqref{afterwklwrsemenergy} together furnish the following \underline{\textit{energy balance}} 
	\begin{equation}\label{fromenidentity2}
		\begin{array}{ll}
			&\displaystyle \sum_{i}c_{i}\int_{\Omega^{i}(t)}\theta_{i}(t)+K_{\kappa}(\eta(t))+
			\frac{1}{2}\int_{\Gamma}|\partial_{t}\eta(t)|^{2}+\frac{1}{2}\sum_{i}\int_{\Omega_{i}}|u_{i}(t)|^{2}\\
			&\displaystyle = K_{\kappa}(\eta_{0})+\sum_{i}\int_{\Omega^{i}_{0}}\theta^{0}_{i}+\frac{1}{2}\int_{\Gamma}|\eta^{0}_{1}|^{2}+\frac{1}{2}\sum_{i}\int_{\Omega^{i}_{0}}|u^{0}_{i}|^{2}.
		\end{array}
	\end{equation}
	and further one obtains that
	\begin{equation}\label{identifyliminfnumdiss2}
		\begin{array}{ll}
			\liminf_{h\rightarrow 0} \left\langle \mathcal{D}^{h,i},1\right\rangle=\left\langle \mathcal{D}^{i},1\right\rangle.
		\end{array}
	\end{equation}
	On the other hand please observe that the role of $\liminf_{h\rightarrow 0}$ in \eqref{aftrhlevheat} can be replaced  by $\limsup_{h\rightarrow 0}$ and afterwards following the arguments used in showing \eqref{identifyliminfnumdiss2} we get that
	\begin{equation}\nonumber
		\begin{array}{ll}
			\limsup_{h\rightarrow 0} \left\langle \mathcal{D}^{h,i},1\right\rangle=\left\langle \mathcal{D}^{i},1\right\rangle,
		\end{array}
	\end{equation}
	which together with \eqref{identifyliminfnumdiss2} furnish
	\begin{equation}\label{limH11tau2}
		\begin{array}{l}
			\displaystyle 	\lim_{h\rightarrow 0} \left\langle \mathcal{D}^{h,i},1\right\rangle=\left\langle \mathcal{D}^{i},1\right\rangle.
		\end{array}
	\end{equation}
	Now the convergence \eqref{limH11tau2} along with the uniform convexity of $L^{2}$ norms one first concludes the strong convergence of the individual integrands involved in the expression of $\mathcal{D}^{h,i}$ and thereby infers that 
	\begin{equation}\label{strngconvHtauizeta12}
		\begin{array}{l}
			\displaystyle \lim_{h\rightarrow 0} \left\langle \mathcal{D}^{h,i},\zeta_{i}\right\rangle=\left\langle \mathcal{D}^{i},\zeta_{i}\right\rangle.
		\end{array}
	\end{equation}
	\subsubsection{Recovery of heat identity} As a consequence of \eqref{aftrhlevheat} and \eqref{strngconvHtauizeta12}, one obtains the following  equation solved by the temperature $\theta_{i}:$
	\begin{equation}\label{aftrtau0heatfinalh}
		\begin{array}{ll}
			&\displaystyle c_{i}\int^{t}_{0}\frac{d}{dt}\int_{\Omega^{i}(t)}\theta_{i}(x,t)\zeta_{i}(x,t)-c_{i}\int^{t}_{0}\int_{\Omega^{i}(t)}\theta_{i}(x,t)(\partial_{t}+u_{i}\cdot\nabla)\zeta_{i}+k_{i}\int^{t}_{0}\int_{\Omega^{i}(t)}\nabla\theta_{i}\cdot\nabla\zeta_{i}\\
			&\displaystyle+\lambda\mathcal{I}_{i}\int^{t}_{0}\int_{\partial\Omega^{2}(t)}(\theta_{1}-\theta_{2})\zeta_{i}\\
			&\displaystyle  =\kappa\int^{t}_{0}\int_{\Gamma}|\nabla^{k_{0}+1}\partial_{t}\eta|^{2}\zeta_{i}+\kappa\int^{t}_{0}\int_{\Omega^{i}}|\nabla^{k_{0}}u_{i}|^{2}\zeta_{i}+\int^{t}_{0}\int_{\Omega^{i}}\mu_{i}(\theta_{i})|Du_{i}|^{2}\zeta_{i}\\
		\end{array}
	\end{equation}
	for $a.e.$ $t\in [0,T]$ and for $\zeta_{i}\in C^{\infty}_{c}([0,T];C^{\infty}_{c}(\overline{\Omega^{i}})),$ $\zeta_{i}\geqslant 0.$
	\subsubsection{Obtainment of entropy evolution inequality}
	One can further use test functions of the form $\varphi'(\theta_{i}),$ (where $\varphi(\theta)$ is monotone and concave for $\theta>0$) in  \eqref{aftrhlevheat} (the use of such test functions can be justified by a density argument) to conclude the following entropy evolution
	\begin{equation}\label{entropybh2}
		\begin{array}{ll}
			\displaystyle\sum_{i=1}^{2}\left(c_{i}\frac{d}{dt}\int_{\Omega^{i}}\varphi(\theta_{i})+k_{i}\int_{\Omega^{i}}|\nabla\theta_{i}|^{2}\varphi''(\theta_{i})\right)+\lambda\int_{\partial\Omega^{2}}\left(\theta_{1}-\theta_{2}\right)\left(\varphi'(\theta_{1})-\varphi'(\theta_{2}))\right)\geqslant 0.
		\end{array}
	\end{equation}
	Indeed $\varphi(\cdot),$ should be chosen such that all the integrals in \eqref{entropybh2} make sense. For instance, since $\theta_{i}\geqslant\gamma>0,$ $\varphi(\theta_{i})$ can be chosen to be $\theta^{(\lambda+1)}_{i}$ for $\lambda\in (-1,0)$ or $\ln\theta_{i}$ (the physical entropy).
		\subsection{Minimal principle for temperature} As a direct consequence of \eqref{claimstrngconvuptobndrythetah} and \eqref{postemphlev}, one has that
		\begin{equation}\label{minpritemphfnl}
			\begin{array}{l}
				\displaystyle \theta_{i}\geqslant \gamma>0\quad\mbox{in}\quad \Omega^{i}\times (0,T)\,\,\mbox{provided}\,\, \theta^{0}_{i}\geqslant\gamma>0\,\,\mbox{on}\,\,\Omega^{i}_{0}.
			\end{array}
		\end{equation}
		\subsection{Conclusion and the proof of Theorem \ref{Th:mainhlev}} Throughout the present section we have shown all the assertions listed in Theorem \ref{Th:mainhlev}:\\
		$(i)$ The regularity \eqref{regularity3} of $(u_{i},\theta_{i},\eta)$ follows from \eqref{weakconvhlev}, \eqref{Lpthetaih} and  \eqref{fromenidentity2}.\\
		$(ii)$ The decomposition of $\Omega$ as asserted in item $(ii)$ of Definition \ref{Def:reg} follows from \eqref{setincond}$_{2,3}$ and the strong cpnvergence \eqref{weakconvhlev}$_{2}.$\\
		$(iii)$ The momentum balance and the heat evolution listed in item $(ii)$ and $(iii)$ are proved in Section \ref{ctestfn} and \eqref{heatrecovery}.\\
		$(iii)$ The minimal principle of temperatures are obtained in \eqref{minpritemphfnl}.\\
		$(iv)$ As a consequence of recovering the heat equation, we have shown the energy balance and an identity involving the dissipation terms in \eqref{fromenidentity2} and \eqref{momenaftrtest} respectively.\\
		$(iv)$ Finally we track the entropy evolution in \eqref{entropybh2}.\\
		This finishes the proof of Theorem \ref{Th:mainhlev}.
	
	
	\section{Removal of the regularizing parameter $\kappa$ and proof of Theorem \ref{Th:main}}
	First let us write down the following equations and inequalities solved by the triplet $(u^{\kappa}_{i},\theta^{\kappa}_{i},\eta^{\kappa}_{i}),$ which are obtained in Theorem \ref{Th:mainhlev}.\\
	(i)\textit{The following decomposition of $\Omega$ holds:}
	\begin{equation}\label{Omegat1kappa}
		\begin{array}{ll}
			&\Omega={{\Omega^{1,\kappa}(t)}}\cup\overline{\Omega^{2,\kappa}(t)},\quad \partial{\Omega^{2,\kappa}(t)}=\Sigma_{\eta^{\kappa}},\\[2.mm] &\partial({\Omega^{1,\kappa}(t)})=\partial\Omega\cup \Sigma_{\eta^{\kappa}},
		\end{array}
	\end{equation}  
	where \eqref{varphieta}-\eqref{Sigmat} hold with $\eta$ and $\Sigma_{\eta}$ replaced respectively by $\eta^{\kappa}$ and $\Sigma_{\eta^{\kappa}}.$\\[2.mm]
	(ii) {\textit{A momentum balance holds in the sense:}}\\
	\begin{equation}\label{momentplvlh2}
		\begin{array}{ll}
			&\displaystyle \int_{0}^{t}\bigg(\langle DK_{\kappa}(\eta^{\kappa}),b\rangle-\int_{\Gamma}\partial_{t}\eta^{\kappa}\partial_{t}b+\kappa\int_{\Gamma}\nabla^{k_{0}+1}\partial_{t}\eta^{\kappa}\nabla^{k_{0}+1}b\bigg)\\
			&\displaystyle+\sum_{i}\int^{t}_{0}\bigg(-\langle u_{i}^{\kappa},\partial_{t}\psi-u_{i}^{\kappa}\cdot\nabla\psi\rangle+\kappa\int_{\Omega^{i,\kappa}}\nabla^{k_{0}}u_{i}^{\kappa}:\nabla^{k_{0}}\psi+\int_{\Omega^{i,\kappa}}\mu_{i}(\theta_{i}^{\kappa})D(u_{i}^{\kappa}):D(\psi)\bigg)\\
			&\displaystyle=-\langle \partial_{t}\eta^{\kappa}(\cdot,t),b(\cdot,t)\rangle +\langle \eta_{1}^{0},b(\cdot,0)\rangle-\langle u_{i}^{\kappa}(\cdot,t),\psi(\cdot,t)\rangle+\langle u^{0}_{i},\psi(\cdot,0)\rangle
		\end{array}
	\end{equation}
	for $a.e.$ $t\in (0,T),$ for all $b\in L^{\infty}(0,T;W^{k_{0}+1,2}(\Gamma))\cap W^{1,\infty}(0,T;L^{2}(\Gamma)),$  $\psi\in C^{\infty}_0([0,T]\times \Omega)$ and $\mbox{div}\,\psi=0.$ Further the following compatibility conditions hold:
	\begin{equation}\label{couplingkappa}
		\begin{array}{l}
			\displaystyle \Tr_{\Sigma_{\eta^{\kappa}}}u_{i}^{\kappa}=\partial_{t}\eta^{\kappa}\nu\quad\mbox{and}\quad \Tr_{\Sigma_{\eta^{\kappa}}}\psi=b\nu\quad\mbox{on}\quad \Gamma.
		\end{array}
	\end{equation}\\[2.mm]
(iii) {\textit{The heat evolution holds in the sense:}}\\
\begin{equation}\label{heatfnlhlev2}
\begin{array}{ll}
	&\displaystyle c_{i}\int^{t}_{0}\frac{d}{dt}\langle \theta_{i}^{\kappa},\zeta_{i}\rangle-c_{i}\int^{t}_{0}\langle \theta_{i}^{\kappa},(\partial_{t}+u_{i}^{\kappa}\cdot\nabla)\zeta_{i}\rangle+k_{i}\int^{t_{1}}_{0}\int_{\Omega^{i,\kappa}}\nabla\theta_{i}^{\kappa}\cdot\nabla\zeta_{i}
	\displaystyle+\lambda\mathcal{I}_{i}\int^{t}_{0}\int_{\partial\Omega^{2,\kappa}}(\theta_{1}^{\kappa}-\theta_{2}^{\kappa})\zeta_{i}\\
	&\displaystyle=\int^{t}_{0}\int_{\Omega^{i,\kappa}}\mu_{i}(\theta_{i}^{\kappa})|Du_{i}^{\kappa}|^{2}\zeta_{i}+\kappa\int^{t}_{0}\int_{\Omega^{i,\kappa}}|\nabla^{k_{0}}u_{i}^{\kappa}|^{2}\zeta_{i}+\kappa\mathcal{X}_{i}\int^{t}_{0}\int_{\Gamma}|\nabla^{k_{0}+1}\partial_{t}\eta^{\kappa}|^{2}\zeta_{i}=\left\langle \mathcal{D}^{\kappa,i},\zeta_{i}\right\rangle
\end{array}
\end{equation}
where 
for $a.e.$ $t\in [0,T]$ and for $\zeta_{i}\in C^{\infty}_{c}([0,T];C^{\infty}_{c}(\overline{\Omega^{i,\kappa}})),$ where where $\mathcal{I}_{1}=1,$ $\mathcal{I}_{2}=-1,$ $\mathcal{X}_{1}=1$ and $\mathcal{X}_{2}=0.$\\[2.mm] 
(iv) \textit{Positivity of temperature:}\\
\begin{equation}\label{postempht0lev2}
\begin{array}{l}
	\displaystyle \theta_{i}^{\kappa}\geqslant \gamma\,\,\mbox{in}\,\,\Omega^{i,\kappa}\times(0,T)\,\,\mbox{provided}\,\,\theta^{0}_{i}\geqslant\gamma>0\,\,\mbox{on}\,\,\Omega^{i}_{0}.
\end{array}
\end{equation}
(v) \textit{The total energy is conserved:}\\
\begin{equation}\label{energybalhlev2}
\begin{array}{ll}
	\mathbb{E}_{tot}(\theta_{i}^{\kappa},u_{i}^{\kappa},\partial_{t}\eta^{\kappa},\eta^{\kappa})(t)=\mathbb{E}_{tot}(\theta^{0}_{i},u^{0}_{i},\eta^{0},\eta^{0}_{1}),
\end{array}
\end{equation}
where the total energy of the syatem is defined as
$$\mathbb{E}_{tot}(\theta_{i},u_{i},\partial_{t}\eta,\eta)=\sum^{2}_{i=1}\int_{\Omega^{i,\kappa}}\bigg(c_{i}\theta_{i}+\frac{1}{2}|u_{i}|^{2}\bigg)+\int_{\Gamma}\bigg(\frac{|\partial_{t}\eta|^{2}}{2}+K_{\kappa}(\eta)\bigg).$$\\[2.mm]
(vi)\textit{An identity estimating the dissipative terms}:\\
\begin{equation}\label{disstypesthlevfhlvl2}
\begin{array}{ll}
	&\displaystyle K_{\kappa}(\eta^{\kappa}(t))+\frac{1}{2}\int_{\Gamma}|\partial_{t}\eta^{\kappa}(t)|^{2}+\frac{1}{2}\sum_{i}\int_{\Omega^{i,\kappa}}|u_{i}^{\kappa}(t)|^{2}\\
	&\displaystyle+\sum_{i}\left(\int^{t}_{0}\int_{\Omega^{i,\kappa}}\mu_{i}(\theta_{i}^{\kappa})|Du_{i}^{\kappa}|^{2}+\kappa\int^{t}_{0}\int_{\Omega^{i,\kappa}}|\nabla^{k_{0}}u_{i}^{\kappa}|^{2}+\kappa\int^{t}_{0}\int_{\Gamma}|\nabla^{k_{0}+1}\partial_{t}\eta^{\kappa}|^{2}\right)\\
	&\displaystyle = K_{\kappa}(\eta^{0})+\frac{1}{2}\int_{\Gamma}|\eta^{0}_{1}|^{2}+\frac{1}{2}\int_{\Omega^{i}_{0}}|u^{0}_{i}|^{2}
\end{array}
\end{equation}
for $a.e.$ $t\in[0,T].$\\[2.mm]
(vii) \textit{A general entropy inequality holds in the following sense:}\\
\begin{equation}\label{entrpybalhlev2}
\begin{array}{l}
	\displaystyle\sum_{i=1}^{2}\left(c_{i}\frac{d}{dt}\int_{\Omega^{i,\kappa}}\varphi(\theta_{i}^{\kappa})+k_{i}\int_{\Omega^{i,\kappa}}|\nabla\theta_{i}^{\kappa}|^{2}\varphi''(\theta_{i}^{\kappa})\right)+\lambda\int_{\partial\Omega^{2,\kappa}}\left(\theta_{1}^{\kappa}-\theta_{2}^{\kappa}\right)\left(\varphi'(\theta_{1}^{\kappa})-\varphi'(\theta_{2}^{\kappa}))\right)\geqslant 0,
\end{array}
\end{equation}
for any $\varphi(\theta)$ such that $\varphi(\cdot)$ is concave and monotone for $\theta>0$ such that all the integrals above make sense.\\
Based on the estimates \eqref{energybalhlev2} and \eqref{disstypesthlevfhlvl2} and the positive lower bound of the temperature, one at once has the following convergences:
\begin{equation}\label{weakconvhlev2}
\begin{array}{lll}
	&\displaystyle {\eta}^{\kappa}\rightharpoonup^* \eta\,\,&\displaystyle\mbox{in}\,\, L^{\infty}(0,T;W^{2,2}(\Gamma)),\\[2.mm]
	&\displaystyle \eta^{\kappa}\rightarrow \eta\,\,&\displaystyle\mbox{in}\,\,C^{\frac{1}{4}}(\Gamma\times[0,T]),\\[2.mm]
	&\displaystyle \partial_{t}{\eta}^{\kappa}\rightharpoonup \partial_{t}\eta\,\,&\displaystyle\mbox{in}\,\, L^{\infty}(0,T;L^{2}(\Gamma)),\\[2.mm]
	&\displaystyle u_{i}^{\kappa}\rightharpoonup u_{i}\,\,&\displaystyle \mbox{in}\,\, L^{2}(0,T;W^{1,2}(\Omega^{i,\kappa})),\\[2.mm]
	&\displaystyle u_{i}^{\kappa}\rightharpoonup^{*} u_{i} \,\,&\displaystyle \mbox{in}\,\, L^{\infty}(0,T;L^{2}(\Omega^{i,\kappa})),\\[2.mm]
	&\displaystyle \theta^{\kappa}_{i}\rightharpoonup^* \theta_{i}\,\,&\displaystyle \mbox{in}\,\, L^{\infty}(0,T;L^{1}(\Omega^{i,\kappa}))
\end{array}
\end{equation}
Indeed the convergence \eqref{weakconvhlev2}$_{4}$ implies that
\begin{equation}
\begin{array}{ll}
	u^{\kappa}_{i}\rightharpoonup u_{i}\,\,\mbox{in}\,\, L^{2}(0,T;L^{6-}(\Omega^{i,\kappa})).
\end{array}
\end{equation}
Further one can use similar arguments as the ones used in Section \ref{frthrtempest} to show the following improved uniform estimates on the temperature
\begin{equation}\label{Lpthetaih2}
\begin{array}{ll}
	&\displaystyle\|\theta_{i}^{\kappa}\|_{L^{p}(0,T;L^{p}(\Omega^{i,\kappa}))}\leqslant C(\eta^{0},\theta^{0}_{i},\eta^{0}_{1},u^{0}_{i})\,\,\mbox{for}\,\,p\in[1,\frac{5}{3}),\\
	&\displaystyle 	\displaystyle\int_{(0,T)}\|\theta_{i}^{\kappa}\|^{s}_{W^{1,s}(\Omega^{i,\kappa})}\leqslant C(\eta^{0},\theta^{0}_{i},\eta^{0}_{1},u^{0}_{i})\,\,\mbox{for all}\,\, s\in[1,\frac{5}{4}).
\end{array}
\end{equation}
Now one uses the heat equation \eqref{heatfnlhlev2} and follow arguments used in Section \ref{compacttemhlev} to show the strong convergence 
\begin{equation}\label{strngconvtemp2}
\begin{array}{l}
	\theta^{\kappa}_{i}\rightarrow \theta_{i}\,\,\mbox{in}\,\,L^{s}(\Omega^{i,\kappa}\times(0,T)),\,\,\mbox{for any}\,\,  s\in[1,\frac{5}{3}).
	\displaystyle 
\end{array}
\end{equation}
Towards proving \eqref{strngconvtemp2}, one need to slightly modify the calculation presented in \eqref{estimatederttheta} by replacing the first term of the second line by $\|\theta^{\kappa}_{i}\|_{L^{\frac{5}{3}-}(L^{\frac{5}{3}-})}\|u^{\kappa}_{i}\|_{L^{\frac{5}{2}+}(L^{\frac{5}{2}+})}\|\nabla\zeta_{i}\|_{L^{\infty}(L^{\infty})},$ which is uniformly bounded since $u^{\kappa}_{i}\in L^{2}(0,T;L^{6-}(\Omega^{i,\kappa}))\cap L^{\infty}(0,T;L^{2}(\Omega^{i,\kappa}))\hookrightarrow L^{\frac{5}{2}+}(0,T;L^{\frac{5}{2}+}(\Omega^{i,\kappa})).$
One further needs to use the uniform bound on the dissipation $\left\langle \mathcal{D}^{\kappa,i},1\right\rangle$ furnished by \eqref{disstypesthlevfhlvl2}.
\subsection{Compactness of the shell-energy}\label{compactnessshelenergy}
The next result is a very important one stating that $\eta^{\kappa}$ (uniformly in $\kappa$) enjoys better regularity in space. On one hand this extra regularity helps in the limit passage $\kappa\rightarrow 0_{+}$ in the term $\int^{t}_{0}\langle DK_{\kappa}(\eta^{\kappa}),b\rangle$ (concerning the shell energy) and on the other hand it renders the fact that the fluid boundary is not just H\"{o}lder but is Lipschitz in space (to be precise in $L^{2}(C^{0,1})$).\\ 
In the context of an incompressible fluid-structure interaction problem the ingenious idea of improving the structural regularity first appeared in \cite{MuhaSch} and was later adapted for compressible fluid-structure interaction problems in \cite{Breit, MitNes}. 
\begin{lemma}\label{improvebndetadelta}
Let the quadruple $(u^{\kappa}_{i},\eta^{\kappa},\theta^{\kappa}_{i})$ solve \eqref{momentplvlh2}-\eqref{heatfnlhlev2}. Then the following holds uniformly in $\kappa:$
\begin{align}\label{improvedregeta}
	\begin{aligned}
		\sqrt{\kappa}\|\eta^{\kappa}\|_{L^{2}(0,T;W^{k_0+1+\sigma^{*},2}(\Gamma))}+\sqrt{\kappa}\|\eta^{\kappa}\|_{L^{\infty}(0,T;W^{k_{0}+\sigma^{*}})}&\leqslant c\\
		\|\eta^{\kappa}\|_{L^{2}(0,T;W^{2+\sigma^{*},2}(\Gamma))} +
		\|\partial_{t}\eta^{\kappa}\|_{L^{2}(0,T;W^{\sigma^{*},2}(\Gamma))} &\leqslant c,
	\end{aligned}
\end{align}
for some $0<\sigma^{*}<\frac{1}{2},$ where the constant $c$ in the last inequality may depend on $\Gamma,$ the initial data and the $W^{2,2}$ coercivity size of $\overline{\gamma}(\eta)$ ($\overline{\gamma}(\eta)$ has been introduced in \eqref{ovgamma}).
\end{lemma}  
\begin{proof}
The proof follows almost line by line as it was shown  \cite[Section 4]{MuhaSch}: The fractional differentiability of the time derivative follows by trace estimates. The higher fractional differentiability of $\eta$ follows by testing with double difference quotients. In order to give a sketch of that we introduce,
\[
D^s_
h(\xi)(x):=\frac{\xi(x+he)-\xi(x)}{h^s}\text{ where }e\text{ is any unit vector.}
\]
Then one uses the following test function
\[
(\mathcal{F}^{div}_{i,\eta^{\kappa}}(D^s_
{-h} D^s_
h\eta^{\kappa}-\mathcal{K}_{i,\eta^\kappa}(D^s_
{-h} D^s_
h\eta^{\kappa})),D^s_
{-h} D^s_
h\eta-\mathcal{K}_{i,\eta}(D^s_
{-h} D^s_
h\eta^{\kappa}))
\]
in \eqref{momentplvlh2}, relying on the extension introduced in Proposition~\ref{smestdivfrex}. Let us first discuss what happens on the shell. The information comes from the stress related to the Koiter energy and is deduced using \cite[Lemma 4.5]{MuhaSch}. The estimates relating to the inertia term is explained in \cite[Section 4.2]{MuhaSch}. Further as $\mathcal{K}_{i,\eta}(D^s_
{-h} D^s_
h\eta)$ is constant in space it has no impact on the regularization terms. This means that the structural regularizers (both for the dissipation and energy) contribute the following positive terms:
\[
\kappa \norm{\nabla^{{k_0}+1}D^s_
	h\eta^\kappa}^2_{L^2((0,T)\times \Gamma)}+\kappa \norm{\nabla^{k_0}D^s_
	h\eta^\kappa}^2_{L^\infty((0,T);L^{2}(\Gamma))},
\]
assuming that the initial condition $\eta_{0}\in W^{k_{0}+1,2}(\Gamma).$
All terms appearing in the fluid equation can be estimated as explained in  \cite[Subsection 4.2]{MuhaSch}. In particular, the appearance of the temperature in the viscosity does not change the argument at all. Since we have reserved one derivative more for the deformation $\eta^{\kappa},$ the additional regularization term for the fluid part can be estimated as follows:
\begin{equation}\nonumber
	\begin{array}{ll} &\displaystyle\kappa\int^{t}_{0}\int_{\Omega^{i,\kappa}}\nabla^{k_0} u^\kappa_{i}:\nabla^{k_{0}}(\mathcal{F}^{div}_{i,\eta^{\kappa}}(D^s_
		{-h} D^s_
		h\eta^{\kappa}-\mathcal{K}_{i,\eta}(D^s_
		{-h} D^s_
		h\eta^{\kappa})))\leqslant c,
	\end{array}
\end{equation}
where $c$ in the last inequality may depend on $\Omega,$ $\Gamma,$ the initial data and the $W^{2,2}$ coercivity size of $\overline{\gamma}(\eta)$ ($\overline{\gamma}(\eta)$ has been introduced in \eqref{ovgamma}). The last inequality is obtained as a consequence of \eqref{disstypesthlevfhlvl2}, the estimate \eqref{higherderest1*} and the uniform in $\kappa$ bound of $\eta^{\kappa}$ in $L^{\infty}(0,T;W^{2,2}(\Gamma)\cap \sqrt{\kappa}W^{k_{0}+1,2}(\Gamma))\cap L^{2}(0,T;\sqrt{\kappa}W^{k_{0}+1+s,2}(\Gamma)).$\\
This finishes the sketch of the proof of \eqref{improvedregeta}.
\end{proof}
As a consequence of \eqref{improvedregeta}$_{2}$ and the classical Aubin-Lions lemma we have the following strong convergence of $\eta^{\kappa}:$\\
\begin{equation}\label{strngconvetadel}
\eta^{\kappa}\rightarrow \eta \,\,\mbox{in}\,\, L^{2}(0,T;W^{2+,2}(\Gamma)).
\end{equation}
The convergence \eqref{strngconvetadel} will be used in particular for the limit passage in the term related to the Koiter energy in the weak-formulation of the momentum equation. The notation $2+$ signifies a number greater than $2$.
\subsection{Compactness of $(u^{\kappa}_{i},\partial_{t}\eta^{\kappa})$} The next lemma infers the strong convergence of the approximate velocity field $(u^{\kappa}_{i},\partial_{t}\eta^{\kappa}).$ 
\begin{lem}\label{strngconvvel}
Let the quadruple $(u^{\kappa}_{i},\eta^{\kappa},\theta^{\kappa}_{i})$ solve \eqref{momentplvlh2}-\eqref{heatfnlhlev2}. Then the following strong convergence hold
\begin{equation}\label{strngconvvel2}
	\left(\partial_{t}\eta^{\kappa}u^{\kappa}_{i}\right)\rightarrow (\partial_{t}\eta,u_{i})\,\,\mbox{in}\,\,L^{2}(0,T;L^{2}(\Omega^{i,\kappa}))
\end{equation}
\end{lem}
The proof of Lemma \ref{strngconvvel} shares some similarities with proof of Proposition \ref{strngconvavg}, but simpler. For a precise reference we would like to quote \cite[Lemma 6.3]{MuhaSch}. One needs to slightly modify the proof of \cite[Lemma 6.3]{MuhaSch} due to the appearance of $$\displaystyle\kappa\int^{t}_{\sigma}\int_{\Gamma}\nabla^{k_{0}+1}\partial_{t}\eta^{\kappa}(s)\nabla^{k_{0}+1}(\partial_{t}\eta^{\kappa}(t))_{\delta}dy ds,\,\,\kappa\int^{t}_{\sigma}\int_{\Gamma}\nabla^{k_{0}+1}\eta^{\kappa}(s)\nabla^{k_{0}+1}(\partial_{t}\eta^{\kappa}(t))_{\delta}dy ds$$
and 
$$\kappa\int^{t}_{\sigma}\int_{\Omega^{i,\kappa}}\nabla^{k_{0}}u^{\kappa}_{i}\nabla^{k_{0}}(E_{i,\eta^{\kappa}(s),\delta}(\partial_{t}\eta^{\kappa}(t)))dxds,$$
where $(\cdot)_{\delta}$ denotes the mollification in the space variable and $E_{i,\eta(s),\delta}(\cdot)$ is as introduced in Corollary \ref{corextension}.
In view of the uniform in $\kappa$ bounds of $\|\partial_{t}\eta^{\kappa}\|_{L^{2}(0,T;\sqrt{\kappa}W^{k_{0},2}(\Gamma))\cap L^{\infty}(0,T;L^{2}(\Gamma))},$ $\|\eta^{\kappa}\|_{L^{\infty}(0,T;\sqrt{\kappa}W^{k_{0}+1}(\Gamma))}$ and $\|u\|_{L^{2}(0,T;\sqrt{\kappa}(W^{k_{0},2}(\Omega^{i,\kappa})))}$ the above mentioned adaptation can be performed in a straight forward manner.
\subsection{Conclusion and the proof of Theorem \ref{Th:main}}
In view of the obtained convergences of $(u^{\kappa}_{i},\eta^{\kappa},\theta^{\kappa}_{i}),$ let us now prove Theorem \ref{Th:main}.\\
The regularity of the limiting functions \eqref{regularity}, follows from \eqref{weakconvhlev2}, \eqref{Lpthetaih2} and \eqref{improvedregeta}.\\
The decomposition of $\Omega$ of the form \eqref{Omegat12}- \eqref{varphieta}-\eqref{Sigmat} hold in view of \eqref{Omegat1kappa}.\\
The coupling criterian $\Tr_{\Sigma_{\eta}}u=\partial_{t}\eta\nu$ is a consequence of the available convergences and \eqref{couplingkappa}, whereas the positivity \eqref{postempht0levmain} follows from \eqref{postempht0lev2}.\\
Now let us discuss the passage of $\kappa\rightarrow 0$ in the equations \eqref{momentplvlh2} to \eqref{entrpybalhlev2}.
\subsubsection{Passage to the limit in \eqref{momentplvlh2}} As was already discussed in Section \ref{passtau0} and Section \ref{constructiontesthlev}, here too we need look for test functions which solve the compatibility condition $b\nu=\Tr_{\Sigma_{\eta}}\psi$ at the limiting interface $\Sigma_{\eta},$ since it is not apriori guaranteed by a similar compatibity at the approximate layer. Following the approach of Section \ref{passtau0}, one constructs a pair $(\psi,b^{\kappa} ),$ such that $\psi\in C^{\infty}([0,T]\times \overline{\Omega})$ and
\begin{equation}\label{constestkappalev}
\begin{array}{ll}
	&\displaystyle 	b^{\kappa}=\Tr_{\Sigma_{\eta^{\kappa}}}\psi\cdot\nu,\quad 	b=\Tr_{\Sigma_{\eta}}\psi\cdot\nu,\\
	&\displaystyle 	b^{\kappa}\rightarrow b\,\,\mbox{in}\,\, L^{2}(0,T;W^{2+,2}(\Gamma)),
\end{array}
\end{equation}
where the last convergence is a consequence of the improved regularity \eqref{improvedregeta}. The convergence \eqref{constestkappalev} along with \eqref{strngconvetadel} and \eqref{strngconvvel2} infer that
\begin{equation}\label{convintetatest}
\begin{array}{ll}
	&\displaystyle \int^{t}_{0}\int_{\Gamma}\partial_{t}\eta^{\kappa}\partial_{t}b^{\kappa}\rightarrow \int^{t}_{0}\int_{\Gamma}\partial_{t}\eta\partial_{t}b,\qquad\displaystyle \int^{t}_{0}\langle DK_{\kappa}(\eta^{\kappa}),b^{\kappa}\rangle \rightarrow \int^{t}_{0}\langle DK(\eta),b\rangle.
\end{array}
\end{equation}
Specifically for the second convergence in \eqref{convintetatest}, one first verifies that $DK_{\kappa}(\eta^{\kappa})\rightharpoonup DK(\eta)$ in $L^{2}(0,T;L^{1+}(\Gamma)),$ in view of the convergences \eqref{weakconvhlev2}$_{1}$, \eqref{strngconvetadel} and the structure of $DK_{\kappa}$ (cf. \eqref{elasticityoperator},\eqref{amab},\eqref{Geta},\eqref{Gij'},\eqref{exaR} and \eqref{DefDk}). In particlular one can observe that as $\kappa\rightarrow 0,$ $\displaystyle\kappa\int^{t}_{0}\int_{\Gamma}\nabla^{k_{0}+1}\eta^{\kappa}\nabla^{k_{0}+1}b^{\kappa}\rightarrow 0,$
since 
$$\left|\kappa\int^{t}_{0}\int_{\Gamma}\nabla^{k_{0}+1}\eta^{\kappa}\nabla^{k_{0}+1}b^{\kappa}\right|\leqslant \sqrt{\kappa}\|\nabla^{k_{0}+1}\eta^{\kappa}\|_{L^{2}(0,T;L^{2}(\Gamma))}\sqrt{\kappa}\|\nabla^{k_{0}+1}b^{\kappa}\|_{L^{2}(0,T;L^{2}(\Gamma))},$$
where the first term is bounded and the second converges to zero.\\
A similar reasoning renders the fact that $\displaystyle\kappa\int^{t}_{0}\int_{\Gamma} \nabla^{k_{0}+1}\partial_{t}\eta^{\kappa}\nabla^{k_{0}+1}b\rightarrow 0$ and $\displaystyle\kappa \int^{t}_{0}\int_{\Omega^{i,\kappa}}\nabla^{k_{0}}u^{\kappa}_{i}:\nabla^{k_{0}}\psi\rightarrow 0$ as $\kappa\rightarrow 0.$\\
In view of the arguments above and the weak convergences \eqref{weakconvhlev2} and the strong convergences \eqref{strngconvtemp2} and \eqref{strngconvvel2} (the strong convergence of $u^{\kappa}_{i}$ is essential to pass limit in the fluid convection) one concludes the obtainment of \eqref{momentplvlhmain}.
\subsubsection{Passage to the limit in \eqref{heatfnlhlev2} and \eqref{energybalhlev2}} In view of \eqref{strngconvtemp2} one observes that $\mu_{i}(\theta_{i}^{\kappa})\zeta^{\kappa}_{i}\mathcal{X}_{\Omega^{i,\kappa}}$ converges strongly to $\mu_{i}(\theta_{i})\zeta_{i}\mathcal{X}_{\Omega^{i}}$ in $L^{r}(\Omega\times(0,T)),$ for any $1\leqslant r<\infty$ and any $\zeta\in C^{\infty}_{c}([0,T];C^{\infty}_{c}(\overline{\Omega}))$ with $\zeta\geqslant 0.$ Note that $\zeta^{\kappa}_{i}$ and $\zeta_{i}$ are respectively the restrictions of $\zeta$ into $\Omega^{i,\kappa}$ and $\Omega^{i}.$ This convergence of $\mu_{i}(\theta_{i}^{\kappa})\zeta^{\kappa}_{i}\mathcal{X}_{\Omega^{i,\kappa}}$ along with the weak convergence \eqref{weakconvhlev2}$_{4}$ and the uniform bound on\\ $\|\sqrt{\mu_{i}(\theta^{\kappa}_{i})\zeta_{i}^{\kappa}}Du^{\kappa}_{i}\|_{L^{2}(0,T;L^{2}(\Omega^{i,\kappa}))}$ (see \eqref{disstypesthlevfhlvl2}), we obtain
$$\sqrt{\mu_{i}(\theta^{\kappa}_{i})\zeta_{i}^{\kappa}}Du^{\kappa}_{i}\mathcal{X}_{\Omega^{i,\kappa}}\hookrightarrow \sqrt{\mu_{i}(\theta_{i})\zeta_{i}}Du_{i}\mathcal{X}_{\Omega^{i}}\,\,\mbox{in}\,\, L^{2}(L^{2}).$$
Now using weak lower semicontinuity of norms one concludes the proof of \eqref{heatfnlhlevmain} after passing $\kappa\rightarrow 0$ in \eqref{heatfnlhlev2}.\\
Next we would explain the passage of $\kappa\rightarrow 0$ in \eqref{energybalhlev2}. Observe that by interpolation
$$\displaystyle\sqrt{\kappa}\|\nabla^{k_{0}+1}\eta^{\kappa}(t)\|_{L^{2}(\Gamma)}\leqslant C\kappa^{\frac{\sigma^{*}}{2(k_{0}-1+\sigma^{*})}}\left(\sqrt{\kappa}\|\eta^{\kappa}(t)\|_{W^{k_{0}+1+\sigma^{*},2}(\Gamma)}\right)^{\frac{k_{0}-1}{k_{0}-1+\sigma^{*}}}\|\eta^{\kappa}(t)\|^{\frac{\sigma^{*}}{k_{0}-1+\sigma^{*}}}_{W^{2,2}(\Gamma)}.$$
Now in view of \eqref{improvedregeta}$_{1}$ and \eqref{disstypesthlevfhlvl2}, both the quantities $\sqrt{\kappa}\|\eta^{\kappa}(t)\|_{W^{k_{0}+1+\sigma^{*},2}(\Gamma)}$ and $\|\eta^{\kappa}(t)\|_{W^{2,2}(\Gamma)}$ are bounded with constants independent of $\kappa$ (but can depend on $t$) for $a.e.$ $t\in[0,T].$ Hence as $\kappa\rightarrow 0,$ one recovers that
\begin{equation}\label{convregenergy}
\begin{array}{l}
	\displaystyle \kappa \int_{\Gamma} |\nabla^{k_{0}+1}\eta^{\kappa}(t)|^{2}\rightarrow 0,
\end{array}
\end{equation}
for $a.e.$ $t\in [0,T].$ Hence using \eqref{strngconvetadel}, \eqref{weakconvhlev2}, \eqref{convregenergy} and the structure of $K(\eta)$ ((cf. \eqref{elasticityoperator},\eqref{amab},\eqref{Geta},\eqref{Gij'},\eqref{exaR}), we get that
\begin{equation}\label{Kkappaconv}
\begin{array}{l}
	\displaystyle \int_{\Gamma} K_{\kappa}(\eta^{\kappa}(t))\rightarrow \int_{\Gamma} K(\eta(t)),\,\,\mbox{for a.e.}\,\,t\in[0,T].
\end{array}
\end{equation}
Finally using the convergences \eqref{strngconvtemp2}, \eqref{strngconvvel} and \eqref{Kkappaconv}, one furnishes the energy balance \eqref{energybalhlevmain} (for $a.e.$ $t\in[0,T]$) by passing $\kappa\rightarrow 0$ in \eqref{energybalhlev2}.

\subsubsection{A minimal time where the solution avoids degeneracy}\label{minintex}
The solution we obtained can undergo degeneracy at finite time, either in the form of \eqref{degenfstkind} or \eqref{degensndkind} due to the loss of $W^{2,2}(\Gamma)$ coercivity of the non-linear Koiter shell. But in the following we show that there always exist a time interval where both kind of degeneracies can be avoided.\\
\underline{\textit{Avoiding degeneracy of the form \eqref{degenfstkind}}} 
From \eqref{weakconvhlev2}$_{2},$ we have that for $(t,x)\in (0,T)\times\Gamma$
\begin{equation}\label{estetak}
\eta^{\kappa}_{0}-ct^\frac{1}{4}\leq \eta^{\kappa}_{0}(x)-|\eta^\kappa(t,x)-\eta^{\kappa}_{0}(x)|\leq \eta^\kappa(t,x)\leq \eta^{\kappa}_{0}(x)+|\eta^\kappa(t,x)-\eta^{\kappa}_{0}(x)|\leq \eta^{\kappa}_{0}+ct^\frac{1}{4}
\end{equation}
with $c$ independent of all regularizing parameters. Now since $\eta^{\kappa}_{0}$ uniformly converges to $\eta_{0},$ in view of \eqref{estetak}, one can furnish the existence of a small enough time where the degeneracy \eqref{degenfstkind} can be excluded.\\[2.mm]
\underline{\textit{Excluding degeneracy of the form \eqref{degensndkind}}} The next lemma excludes the possibility of \eqref{degensndkind} in a small enough time interval. The proof of Lemma \ref{Lem:Coer} can be done by following the arguments used in proving Lemma $6.5,$ \cite{MitNes}.
\begin{lemma}\label{Lem:Coer}
Let the assumptions of Theorem \ref{Th:main} be satisfied and $K_{\kappa}(\eta^{\kappa})$ is bounded as in \eqref{disstypesthlevfhlvl2}.
Then if $\overline{\gamma}(\eta)\neq 0$ we have $\eta(t)\in W^{2,2}(\Gamma)$ and moreover 
\begin{equation}\label{coercivityrel}
	\begin{array}{l}
		\displaystyle
		\sup_{t\in[0,T]}\int_{\Gamma}\overline{\gamma}^{2}(\eta)|\nabla^{2}\eta|^{2}\leqslant C_{0},
	\end{array}
\end{equation}
for some constant $C_{0}$ depending on the initial data. Furthermore, let $\overline{\gamma}(\eta_{0})>0$ then there is a minimal time $T_*$ depending on the initial configuration such that $\overline{\gamma}(\eta)>0$ and \eqref{coercivityrel} hold in $(0,T_{*})$.
\end{lemma}
\subsubsection{Maximal interval of existence}\label{Maxtimeex}
In Section \ref{minintex}, we have shown that for a positive minimal time degeneracy of the solution can be excluded. We can start by solving the problem in $[0,T_{\min}]$. Next considering  $(\eta,u_{i},\theta_{i})(T_{\min})$ as new initial conditions we repeat the existence proof in the interval $(T_{\min},2T_{\min}).$ We can iterate the procedure unitl a degeneracy occurs. That way we obtain a maximal time $T_F\in(0,\infty]$ such that $(\eta,u_{i},\theta_{i})$ is a weak solution on the interval $(0,T)$ for any $T<T_F$. If $T_F$ is finite than either the $W^{2,2}$--coercivity of the Koiter energy is violated or $\lim_{s\to T_F}\eta(s,y)\in\{a_{\partial\Omega},b_{\partial\Omega}\}$. The extension procedure is standard nowadays and we refer to \cite[Theorem 3.5]{LeRu14} and \cite[Theorem 1.1]{MuhaSch} for details.
\subsubsection{The superconductive case, $\lambda=\infty$.}
The construction of a solution for the case $\lambda=\infty$ is analogous to the case $\lambda\in [0,\infty)$. The difference is that in this case the temperature is a global object that is continuous over the interface. This means that already the minimisation in \eqref{min2}
(without the $\lambda$-term) has to be taken simply over all $\theta\in C^\infty(\Omega)$. By that minimization an approximative object is constructed, with which one may pass to the limit in exactly the same fashion as in the other cases.

\section{Appendix}
The following lemma allows to compare between certain mean-values and is in particular used during the calculation performed in \eqref{estthrdtrmI}:
\begin{lem}
\cite[Lemma 5.4.]{malSpSch} Let $q:\Omega^{i,h}\times[0,T]\rightarrow \mathbb{R}^{2}$ is integrable and define
$$\widetilde{q}(y,t):=\frac{1}{h}\int^{t}_{t-h}q(\Phi^{i,h}_{s}(t)(y),t+s)ds,$$
(where the notation $\Phi^{i,h}_{s}(t)$ was introduced in \eqref{redefflmap}) then for all $p\in [1,2),$ $p_{1}\in[1,\infty]$ and all $h>0,$ we find that if $\nabla q\in L^{p_{1}}(0,T;L^{p}(\Omega^{i,h})),$ then 
\begin{equation}\label{estdiffmean}
	\begin{array}{ll}
		&\displaystyle \bigg\|\frac{1}{h}\int^{t}_{t-h} q(y,t+s)ds-\widetilde{q}(t)\bigg\|_{L^{\infty}(0,T;L^{p}(\Omega^{i,h}))}\leqslant Ch^{\frac{p_{1}-1}{p_{1}}}\|\nabla q\|_{L^{p_{1}}(0,T;L^{\frac{2p}{2-p}}(\Omega^{i,h}))},
	\end{array}
\end{equation}
where $C$ depends on the energy estimates only. Further we find for $q\in L^{\frac{2p}{p-2}}(\Omega^{i,h})$ that
\begin{equation}\label{nxtestdiffmean}
	\begin{array}{ll}
		&\displaystyle \bigg\|q-q\circ\Phi^{i,h}_{-h}(t)\bigg\|_{L^{p}(\Omega^{i,h}(t))}\leqslant  Ch\|\nabla q\|_{L^{\frac{2p}{2-p}}(\Omega^{i,h})}\,\,\mbox{and}\,\, \bigg\|q-q\circ\Phi^{i,h}_{-h}(t)\bigg\|_{L^{2}(\Omega^{i,h}(t))}\leqslant Ch\mbox{Lip}(q),
	\end{array}
\end{equation}
where $\mbox{Lip}(q)$ denotes the Lipschitz constant for the function $q.$
\end{lem}

\section*{Acknowledgments}
The author S.Mitra is funded by the grant ANRF/ECRG/2024/002168/PMS under the`Prime Minister Early Career Research Grant Scheme, ANRF, Govt. of India'. S. Mitra also acknowledges the support from YFRSG scheme (IITI/YFRSG-Dream Lab/2025-26/Phase-IV/09) provided by IIT Indore to work with Assoc. Prof. Dr. S. Schwarzacher in Charles University, Prague in July, 2025. 


\end{document}